\documentclass[final,reqno]{amsart}

\usepackage[bookmarks=false,draft=false,breaklinks,colorlinks]{hyperref}
\usepackage{mathtools,amssymb,fullpage}
\usepackage[only llbracket,rrbracket,sslash]{stmaryrd}
\usepackage{showkeys,version}

\numberwithin{equation}{subsection}
\numberwithin{table}{subsection}
\allowdisplaybreaks
\usepackage{enumitem}
\setenumerate{label=(\arabic*),nosep}
\setitemize{nosep}
\setdescription{nosep}
\newlist{clist}{enumerate}{1}
\setlist*[clist]{label=(\roman*), nosep}

\newenvironment{Ack}%
{\par \vspace{\baselineskip}%
 \noindent \textbf{Acknowledgements.}}%
{\par \vspace{\baselineskip}}

\usepackage{cleveref}
\crefname{thm}{Theorem}{Theorems}
\crefname{dfn}{Definition}{Definitions}
\crefname{prp}{Proposition}{Propositions}
\crefname{lem}{Lemma}{Lemmas}
\crefname{cor}{Corollary}{Corollaries}
\crefname{fct}{Fact}{Facts}
\crefname{rmk}{Remark}{Remarks}
\crefname{eg}{Example}{Examples}
\crefname{mth}{Theorem}{Theorem}
\crefname{figure}{Figure}{Figures}
\crefname{table}{Table}{Tables}
\crefname{section}{\S\!}{\S\S\!}
\crefname{subsection}{\S\!}{\S\S\!}
\crefname{subsubsection}{\S\!}{\S\S\!}
\crefname{appendix}{\S}{\S\S\!}
\crefname{equation}{equation}{equations}

\usepackage{amsthm}
\theoremstyle{definition}
\newtheorem{thm}{Theorem}[subsection]
\newtheorem{dfn}[thm]{Definition}
\newtheorem{prp}[thm]{Proposition}
\newtheorem{lem}[thm]{Lemma}
\newtheorem{cor}[thm]{Corollary}

\newtheorem{fct}[thm]{Fact}
\newtheorem{rmk}[thm]{Remark}
\newtheorem{eg}[thm]{Example}

\newcommand{\ol}{\overline}

\newcommand{\wh}{\widehat}
\newcommand{\wt}{\widetilde}

\newcommand{\ve}{\varepsilon}
\newcommand{\hf}{\frac{1}{2}} 
\newcommand{\thf}{\tfrac{1}{2}} 
\newcommand{\ceq}{\coloneqq} 
\newcommand{\totimes}{\mathbin{\widetilde{\otimes}}}

\newcommand{\xr}{\xrightarrow}
\newcommand{\xrr}[1]{\xrightarrow{\, #1 \, }}
\newcommand{\inj}{\hookrightarrow}

\newcommand{\lto}{\longrightarrow}
\newcommand{\mto}{\mapsto}
\newcommand{\lmto}{\longmapsto}
\newcommand{\sto}{\xr{\sim}}
\newcommand{\lsto}{\xrr{\sim}}
\newcommand{\lrto}{\leftrightarrow}

\newcommand{\op}{\textup{op}}
\newcommand{\cop}{\textup{cop}}
\newcommand{\tor}{\textup{tor}}

\newcommand{\tprd}{{\textstyle \prod}}

\newcommand{\bbC}{\mathbb{C}}
\newcommand{\bbF}{\mathbb{F}}
\newcommand{\bbN}{\mathbb{N}}
\newcommand{\bbQ}{\mathbb{Q}}
\newcommand{\bbZ}{\mathbb{Z}}
\newcommand{\frg}{\mathfrak{g}}
\newcommand{\frS}{\mathfrak{S}}
\newcommand{\clA}{\mathcal{A}}
\newcommand{\clP}{\mathcal{P}}
\newcommand{\clQ}{\mathcal{Q}}
\newcommand{\fgl}{\mathfrak{gl}}
\newcommand{\fsl}{\mathfrak{sl}}
\newcommand{\frh}{\mathfrak{h}}
\newcommand{\wtH}{\wt{H}}
\newcommand{\wtgm}{\wt{\gamma}}

\newcommand{\abs}[1]{\left| #1 \right|}
\newcommand{\dbr}[1]{\llbracket #1 \rrbracket} 
\newcommand{\dpr}[1]{(\!( #1 )\!)}
\newcommand{\trs}[1]{{}^t \! #1}
\newcommand{\ipd}[2]{\left( #1 \, | \, #2 \right)}
\newcommand{\rst}[2]{\left. #1 \right|_{#2}}
\newcommand{\bnm}[3]{\genfrac{[}{]}{0pt}{}{#1}{#2}_{#3}}
\newcommand{\pair}[1]{\langle #1 \rangle}

\DeclareMathOperator{\id}{id}

\DeclareMathOperator{\clM}{\mathcal{M}}
\DeclareMathOperator{\End}{End}
\DeclareMathOperator{\diag}{diag}
\DeclareMathOperator{\Frac}{Frac}

\newenvironment{psm}{\left(\begin{smallmatrix}}{\end{smallmatrix}\right)}

\title{A dynamical analogue of Ding-Iohara quantum algebras}
\author{Masamune Hattori, Shintarou Yanagida}
\date{2022.10.05}
\address{Graduate School of Mathematics, Nagoya University.
 Furocho, Chikusaku, Nagoya, Japan, 464-8602.}
\email{m21039e@math.nagoya-u.ac.jp, yanagida@math.nagoya-u.ac.jp}

\begin{document}

\begin{abstract}
We introduce a family of dynamical Hopf algebroids $U_{q,p}(g,X_l)$ depending on a complex parameter $q$, a formal parameter $p$, a set $g$ of structure functions satisfying the so-called Ding-Iohara condition, and a finite root system of type $X_l$. If $g$ is set to be certain theta functions, then our family recovers the elliptic algebras $U_{q,p}(\widehat{\mathfrak{g}})$ for untwisted affine Lie algebras $\wh{\frg}$ studied by Konno (1998, 2009), Jimbo-Konno-Odake-Shiraishi (1999) and Farghly-Konno-Oshima (2014). Also, taking the limit $p \to 0$ in the case $X_l=A_l$, we recover the Hopf algebras $U_q(\overline{g},A_l)$ of type $A_l$ with structure functions $\overline{g} := \lim_{p \to 0} g$, introduced by Ding-Iohara (1998) as a generalization of Drinfeld quantum affine algebras. Thus, our Hopf algebroid $U_{q,p}(g,X_l)$ can be regarded as a dynamical analogue of the Ding-Iohara quantum algebras. As a byproduct, we obtain an extension of the Ding-Iohara quantum algebras to those of non-simply-laced type.
\end{abstract}

\maketitle

\setcounter{tocdepth}{2}
\tableofcontents

\setcounter{section}{-1}
\section{Introduction}\label{s:0}

\subsection{Motivation and backgrounds}

The motivation of this note is to investigate the Hopf algebroid structure of elliptic quantum groups and its relation to the Hopf algebra structure of Ding-Iohara quantum algebras. Let us explain briefly these two objects, Ding-Iohara quantum algebras and elliptic quantum groups.

\subsubsection*{Ding-Iohara quantum algebras}

Ding and Iohara \cite{DI} introduced a family of (topological or formal) Hopf algebras which generalize the Drinfeld realization \cite{D} of the quantum affine algebras $U_q(\wh{\fsl}_{l+1})$ with the Drinfeld comultiplication \cite{DI0}. This Hopf algebra depends on a quantum parameter $q$ and a set $g$ of structure functions satisfying certain conditions. we denote it by $U_q(g,A_l)$ and call it \emph{the Ding-Iohara quantum algebra}. A brief review will be given in \cref{ss:pre:DI}.

Let us briefly recall the algebra structure of the Ding-Iohara quantum algebra $U_q(g,A_l)$. See \cref{dfn:DI:U} for the detail. The algebra $U_q(g,A_l)$ is compactly presented by generating currents 
\[
 e_i(z) \ceq \sum_{n \in \bbZ}e_{i,n}z^{-n}, \quad 
 f_i(z) \ceq \sum_{n \in \bbZ}f_{i,n}z^{-n}, \quad 
 \psi^+_i(z) \ceq \sum_{n \in \bbZ}\psi^+_{i,n} z^{-n}, \quad 
 \psi^-_i(z) \ceq \sum_{n \in \bbZ}\psi^-_{i,n} z^{-n},
\]
for $i=1,\dotsc,l$ and current relations among them. The genuine generators of the algebra are the Fourier modes $e_{i,n},f_{i,n},\psi^\pm_{i,n}$, and the genuine relations are obtained by expanding the current relations. See \cref{rmk:DI:top} for the precise treatment. The current relations involve the structure functions $g=\{g_{ij}(z) \mid i,j =1,\dotsc,l\}$. If we set $g_{ij}(z) \ceq q^{-a_{ij}}(1-q^{a_{ij}}z)/(1-q^{-a_{ij}}z)$ with $(a_{ij})_{i,j=1}^l$ being the Cartan matrix of type $A_l$, then the algebra $U_q(g,A_l)$ recovers the Drinfeld realization \cite{D} of the quantum affine algebra $U_q(\wh{\fsl}_{l+1})$. See \cref{eg:DI:qaff} for the detail.

The algebra structure of $U_q(g,A_l)$ is designed to admit the \emph{Drinfeld-type comultiplication}, which involves infinite-summation of the genuine generators. In the case $g_{ij}(z) = q^{-a_{ij}}(1-q^{a_{ij}}z)/(1-q^{-a_{ij}}z)$, this comultiplication reduces to the Drinfeld comultiplication \cite{DI0} of the quantum affine algebra $U_q(\wh{\fsl}_{l+1})$. To admit such a comultiplication, the structure functions $g$ must satisfy a certain condition, which we call the Ding-Iohara condition. See \eqref{eq:DI:cond0} and \eqref{eq:DI:cond} for the detail.

Let us give one more comment on the algebra structure of $U_q(g,A_l)$, which is the starting point of our study. Among the current relations, we have the \emph{Serre-type relations} \eqref{eq:DI:Serre}. For the generating currents $e_i(z)$ and $e_j(z)$ with $a_{ij}=-1$, the Serre-type relation is given by 
\begin{align*}
   & e_i(z_1) e_i(z_2) e_j(w) - h_{ij}(\tfrac{z_1}{w},\tfrac{z_2}{w}) e_i(z_1) e_j(w) e_i(z_2) 
 +e_j(w) e_i(z_1) e_i(z_2) \\ 
+\, &e_i(z_2) e_i(z_1) e_j(w) - h_{ij}(\tfrac{z_2}{w},\tfrac{z_1}{w}) e_i(z_2) e_j(w) e_i(z_1) 
 +e_j(w) e_i(z_2) e_i(z_1) = 0,
\end{align*}
where $h_{ij}(x,y)$ is given in terms of the structure functions $g_{ij}(z)$ as \eqref{eq:DI:hij}:
\begin{align}\label{eq:0:hij}
 h_{ij}(x,y) \ceq \frac{(g_{ii}(x/y)+1) (g_{ij}(x) g_{ij}(y)+1)}{g_{ij}(y)+g_{ii}(x/y)g_{ij}(x)}.
\end{align}
If the structure functions are set to be $g_{ij}(z) = q^{-a_{ij}}(1-q^{a_{ij}}z)/(1-q^{-a_{ij}}z)$, then the function $h_{ij}$ is equal to the $q$-binomial coefficient $\bnm{2}{1}{q}=(q^2-q^{-2})/(q-q^{-1})$, and the above relation is nothing but the Serre relation of the quantum affine algebra $U_q(\wh{\fsl}_{l+1})$.

As mentioned in the final paragraph of \cite{DI}, the argument carries over the simply-laced types $D_l$ and $E_{6,7,8}$, and we have \emph{the Ding-Iohara quantum algebra $U_q(g,X_l)$ of simply-laced type $X_l$}. At present, this class of quantum algebras appears in various contexts, some of which will be reviewed in \cref{ss:pre:DI}.

\subsubsection*{Elliptic quantum groups}

Elliptic quantum groups are still under active investigation, and the meaning of the word ``elliptic quantum group'' is less definite than the word ``quantum group'', as far as we understand. Roughly speaking, the notion of elliptic quantum groups refers to an ``elliptic deformation'' or a ``dynamical analogue'' of quantum groups. It originates in the classification of solutions of the Yang-Baxter equation: the rational, trigonometric, and elliptic solutions, and also in the (classical or quantum) dynamical Yang-Baxter equation. We refer to \cite[\S1.1]{Ko} for the history and backgrounds of elliptic quantum groups.

At present, there are several objects in literature which are called elliptic quantum groups. This note focuses on \emph{the elliptic algebra $U_{q,p}(\wh{\frg})$ for the untwisted affine Lie algebra $\wh{\frg}$ of a finite-dimensional simple Lie algebra $\frg$}, which was first introduced by Konno \cite{Ko98} for the case $\frg=\fsl_2$, and by Jimbo-Konno-Odake-Shiraishi \cite[Appendix]{JKOS} for general $\frg$. Here $q$ denotes the quantum parameter, and $p$ denotes the elliptic nome. We will give a review of $U_{q,p}(\wh{\frg})$ in \cref{ss:pre:E}. For a comprehensive review of the relation between the elliptic algebra $U_{q,p}(\wh{\frg})$ and the other ``elliptic quantum groups'', we refer to \cite[\S1.2]{Ko}.

Here we give a brief explanation on the algebra structure of the elliptic algebra $U_{q,p}(\wh{\frg})$. See \cref{dfn:E:U} for the detail. 
Let $\frg$ be a simple Lie algebra of rank $l$, and $\wh{\frg}$ be the untwisted affine Lie algebra of $\frg$.  The algebra $U_{q,p}(\wh{\frg})$ is presented by generators and generating currents
\begin{align}\label{eq:0:Egen}
 \clM(\wtH^*), \ q^{\pm c/2}, \ d, \ K^{\pm}_i, \ 
 e_i(z), \ f_i(z), \ \psi^\pm_i(z)
\end{align}
for $i=1,\dotsc,l$, subject to the relations given in \cref{dfn:E:U}. Here $\clM(\wtH^*)$ denotes the field of meromorphic functions on the affine space $\wtH$ of dimension $2l$, given in \eqref{eq:E:wtH}. There is a distinguished basis \eqref{eq:E:wtHbas} of $\wtH$, and the elements $P_i$ among them are called \emph{the dynamical parameters}. See \cref{sss:E:dyn} for the detail. 

As mentioned in \cite[p.613]{JKOS}, if $\frg$ is of simply-laced type $X_l$, then the subalgebra of $U_q(\wh{\frg})$ generated by the currents $e_i(z),f_i(z),\psi^\pm_i(z)$ is essentially the special case of the Ding-Iohara quantum algebra $U_q(g,X_l)$ with structure function $g$ set to be certain theta functions. Note that we should take care of the Serre-type relations in the case $l \ge 2$.

Although the algebra structure of the elliptic algebra $U_{q,p}(\wh{\frg})$ has a definite sense, the situation seems to be complicated on the coalgebra structure. As far as we understand, there are two coalgebra structures in literature.
\begin{itemize}
\item 
The paper \cite{JKOS} gave the quasi-Hopf formulation of $U_{q,p}(\wh{\frg})$ based on the Drinfeld realization of the quantum affine algebra $U_q(\wh{\frg})$. Hence the quasi-Hopf algebra $U_{q,p}(\wh{\frg})$ has a comultiplication induced by the Drinfeld comultiplication on $U_q(\wh{\frg})$. We call the induced comultiplication $\Delta$ on $U_{q,p}(\wh{\frg})$ the Drinfeld-type comultiplication. Note that this comultiplication $\Delta$ is not coassociative in terms of the ordinary tensor product $\otimes$.
\item
After a decade, another coalgebra structure $\Delta'$ on $U_{q,p}(\wh{\frg})$ was introduced for the case $\frg=\fsl_2$ by Konno \cite{Ko09}, and for general $\frg$ by Farghly-Konno-Oshima \cite{FKO}. They used the notion of $\frh$-Hopf algebroids in the sense of Etingof and Varchenko \cite{EV}, which will be reviewed in \cref{ss:pre:Halgd}. Here we only remark that one uses the modified tensor product $\totimes$ in the Hopf algebroid formulation. The comultiplication $\Delta'$ can be written neatly in terms of the $L$-operators.
\end{itemize}

One of the motivations of this note is to investigate the Drinfeld-type coalgebra structure $\Delta$ of the elliptic algebra $U_{q,p}(\wh{\frg})$ in terms of the $\frh$-Hopf algebroid formulation.

\subsubsection*{The goal of this note}

At this point, it is natural to ask the following questions.
\begin{itemize}
\item
Is there an extension $U_q(g,X_l)$ of the Ding-Iohara quantum algebra for a simple Lie algebra $\frg$ of \emph{non-simply-laced type} $X_l=B_l,C_l,F_4$ or $G_2$ so that the quantum affine algebra $U_q(\wh{\frg})$ of the same type $X_l$ is a special case of $U_q(g,X_l)$ with structure functions $g$ set to be some rational functions?

\item
Is there an algebra $U_{q,p}(g,X_l)$ which extends the above $U_q(g,X_l)$, has the additional elliptic parameter $p$, and gives the elliptic algebra $U_{q,p}(\wh{\frg})$ of the same type $X_l$ if the structure functions $g$ set to be some theta functions?

\item 
Does the algebra $U_{q,p}(g,X_l)$ admit a Hopf algebroid structure with Drinfeld-type comultiplication?
\end{itemize}

The goal of this note is to introduce a family of Hopf algebroids $U_{q,p}(g,X_l)$ which gives a simultaneous answer to the above questions.

\subsection{Results and organization}

\subsubsection*{Summary of results}

As stated above, the goal of this note is to introduce a family of Hopf algebroids which unifies the elliptic algebras $U_{q,p}(\wh{\frg})$ and the Ding-Iohara quantum algebras $U_q(g,A_l)$. 
For this, we will construct a Hopf algebroid $U_{q,p}(g,X_l)$ depending on
\begin{itemize}
\item 
a complex parameter $q$ (\emph{quantum parameter}),
\item
a formal parameter $p$ (\emph{elliptic nome}),
\item
a finite root system of type $X_l$,
\item
a set $g=\{g_{ij}(z;p) \mid i,j =1,\dotsc,l\}$ of functions depending on $p$ (\emph{structure functions}),
\end{itemize}
and require the following conditions.
\begin{clist}
\item \label{0:g:1}
Specializing the structure functions $g$ appropriately, we recover the elliptic algebras $U_{q,p}\bigl(\wh{\frg}(X_l)\bigr)$ with the Drinfeld-type comultiplication $\Delta$.

\item \label{0:g:2}
In the case $X_l=A_l$, we recover the Ding-Iohara quantum algebra $U_q(\ol{g},A_l)$ by taking the limit $p \to 0$ in $U_{q,p}(g,A_l)$ and setting the structure functions to be $\ol{g} \ceq \lim_{p \to 0}g$.
\end{clist}

Let us give a summary of the results obtained in this note. We will introduce the following three objects.
\begin{itemize}
\item 
\emph{The dynamical Ding-Iohara algebroid $U_{q,p}(g,X_l)$}, a Hopf algebroid satisfying the above conditions \ref{0:g:1} and \ref{0:g:2}. The generators are \eqref{eq:0:Egen}, the same as the elliptic algebra $U_q(\wh{\frg})$:
\[
 \clM(\wtH^*), \ q^{\pm c/2}, \ d, \ K^{\pm}_i, \ 
 e_i(z), \ f_i(z), \ \psi^\pm_i(z)
\]
The precise definitions are given in the following places.
\begin{itemize}
\item 
\cref{dfn:U:ADE,lem:U:Halg,lem:U:Delta,lem:ADE:ve,thm:ADE:Halgd} for simply-laced type $X_l=A_l,D_l,E_{6,7,8}$

\item
\cref{dfn:U:BCF} and \cref{thm:BCF:Halgd} for type $X_l=B_l,C_l,F_4$.
\item
\cref{dfn:U:G} and \cref{thm:G:Halgd} for type $G_2$.
\end{itemize}

\item
A Hopf algebroid structure on the elliptic algebra $U_{q,p}(\wh{\frg})$ whose comultiplication is the Drinfeld-type comultiplication $\Delta$, given in \cref{prp:E:Halgd}.

\item
The Ding-Iohara quantum algebra $U_q(g,X_l)$ of non-simply-laced type $X_l=B_l,C_l,F_4$ or $G_2$, 
given in \cref{dfn:ns:DI} and \cref{prp:ns:DI}.
\end{itemize}
The first and third objects are definitely new, and the first one is the main object of this note. The second object might be known already, although we cannot find it in literature. The main statements of this note are the check of the conditions \ref{0:g:1} and \ref{0:g:2} for the dynamical Ding-Iohara algebroid $U_{q,p}(g,X_l)$. Note that \ref{0:g:2} makes sense for general $X_l$ after the extension of Ding-Iohara quantum algebras. The check are given in:
\begin{itemize}
\item 
\cref{cor:U:ADE}, and \cref{prp:ADE:p=0} for type $X_l=A_l,D_l,E_{6,7,8}$.
\item
\cref{prp:U:BCF} and \cref{prp:ns:p=0} for type $X_l=B_l,C_l,F_4$.
\item
\cref{prp:U:G} and \cref{prp:ns:p=0}  for type $G_2$.
\end{itemize}

Our Hopf algebroid $U_{q,p}(g,X_l)$ is an example of $H$-Hopf algebroid in the sense of Etingof-Varchenko \cite[Chap.\ 4]{EV}, where $H$ is a certain finite-dimensional commutative Lie algebra constructed from the root system of type $X_l$ (see \eqref{eq:U:H}). Let us also mention that $H$-Hopf algebroids form a subclass of \emph{dynamical Hopf algebroids} introduced in \cite[Chap.\ 5]{EV}. Thus, our $U_{q,p}(g,X_l)$ is a dynamical Hopf algebroid, as indicated in the title and abstract of this note.

There is a non-trivial point in the algebra structure of our dynamical Ding-Iohara algebroid $U_{q,p}(g,X_l)$: the Serre-type relations. To spell out the point, let us cite from \cref{dfn:U:ADE} the Serre-type relations in the case of simply-laced type $X_l=A_l,D_l,E_{6,7,8}$. For the generating currents $e_i(z), e_j(z)$ with the Cartan matrix entry $a_{ij}=-1$, the Serre-type relation is given by \eqref{eq:U:eS}:
\begin{align}\label{eq:0:eS}
\begin{split}
&\wt{g}^*_{ii}(\tfrac{z_1}{z_2})
 \wt{g}^*_{ij}(\tfrac{z_1}{z  })
 \wt{g}^*_{ij}(\tfrac{z_2}{z  }) e_i(z_1) e_i(z_2) e_j(z)
-h_{ij}
 \wt{g}^*_{ii}(\tfrac{z_1}{z_2})
 \wt{g}^*_{ij}(\tfrac{z_1}{z  }) e_i(z_1) e_j(z) e_i(z_2) \\
&
+\wt{g}^*_{ii}(\tfrac{z_1}{z_2}) e_j(z) e_i(z_1) e_i(z_2)
+\wt{g}^*_{ij}(\tfrac{z_1}{z  }) 
 \wt{g}^*_{ij}(\tfrac{z_2}{z  }) e_i(z_2) e_i(z_1) e_j(z) \\
&
-h_{ij}
 \wt{g}^*_{ij}(\tfrac{z_2}{z  }) e_i(z_2) e_j(z) e_i(z_1)
+e_j(z) e_i(z_2) e_i(z_1)=0.
\end{split}
\end{align}
Here $\wh{g}^*_{ij}(z)$ is determined from the structure function $g_{ij}(z;p)$, and 
the function $h_{ij}$ is given by \eqref{eq:U:hij}:
\begin{align*}
 h_{ij} \ceq 
 \frac{(\ol{g}_{ii}(\frac{z_1}{z_2})+1)(\ol{g}_{ij}(\frac{z_1}{z  })\ol{g}_{ij}(\frac{z_2}{z})+1)}
      { \ol{g}_{ij}(\frac{z_2}{z  })+   \ol{g}_{ii}(\frac{z_1}{z_2})\ol{g}_{ij}(\frac{z_1}{z})   },
\end{align*}
which is essentially equal to \eqref{eq:0:hij}. The Serre-type relation \eqref{eq:0:eS} is guessed from \eqref{eq:DI:hij} of the Ding-Iohara algebra $U_q(g,X_l)$ and \eqref{eq:E:eS} of the elliptic algebra $U_{q,p}(\wh{\frg})$. If we specialize the structure function $g_{ij}(z;p)$ to the theta function as \eqref{eq:U:gijE}, then $h_{ij}$ is equal to the $q$-binomial $\bnm{2}{1}{q}$, and the above Serre relation \eqref{eq:U:eS} coincides with \eqref{eq:E:eS}. Thus the dynamical Ding-Iohara algebroid under the specialization coincides with the elliptic algebra $U_{q,p}(\wh{\frg})$ (see \cref{prp:U:ADE}). 

However, in the case of non-simply-laced type $X_l=B_l,C_l,F_4,G_2$ explained in \cref{s:ns}, we need special care since there seems no definition of the Ding-Iohara algebra of type $X_l$ in literature. To find out correct Serre-type relations for non-simply-laced types, we build certain principles to find an admissible form of the function $h_{ij}$ for the Cartan matrix entry $a_{ij}=-2,-3$, and the result is summarized in the ``strange formulas'' in \cref{ss:ns:h}. Due to the behavior of the obtained formulas, we need to divide the argument into the two cases: type $B_l,C_l,F_4$ and type $G_2$. The former case is treated in \cref{ss:ns:BCF}, and the latter is in \cref{ss:ns:G}. The Serre-type relation of the generating currents $e_i(z),e_j(z)$ with the Cartan matrix entry $a_{ij}=-2$ in the type $X_l=B_l,C_l,F_4$ is given by \eqref{eq:BCF:eS2} in \cref{dfn:U:BCF}:
\begin{align*}
&\Phi_2(z;z_1,z_2,z_3)
+\wt{g}^*_{ii}(z_{12}) \Phi_2(z;z_2,z_1,z_3) \\ &
+\wt{g}^*_{ii}(z_{12}) \wt{g}^*_{ii}(z_{13}) \Phi_2(z;z_2,z_3,z_1)
+\wt{g}^*_{ii}(z_{23}) \Phi_2(z;z_1,z_3,z_2) \\ & 
+\wt{g}^*_{ii}(z_{13}) \wt{g}^*_{ii}(z_{23}) \Phi_2(z;z_3,z_1,z_2)
+\wt{g}^*_{ii}(z_{12}) \wt{g}^*_{ii}(z_{13}) \wt{g}^*_{ii}(z_{23}) \Phi_2(z;z_3,z_2,z_1) = 0.
\end{align*}
Here $\Phi_2$ is given by \eqref{eq:BCF:Phi2}:
\begin{align*}
 \Phi_2(z;z_1,z_2,z_3) \ceq \ 
&\wt{g}^*_{ij}(z_{10}) \wt{g}^*_{ij}(z_{20}) \wt{g}^*_{ij}(z_{30}) e_i(z_3)e_i(z_2)e_i(z_1)e_j(z) \\ 
&-h^3_{ij} 
 \wt{g}^*_{ij}(z_{20}) \wt{g}^*_{ij}(z_{30}) e_i(z_3)e_i(z_2)e_j(z)e_i(z_1) \\
&+h^3_{ij} 
 \wt{g}^*_{ij}(z_{30}) e_i(z_3)e_j(z)e_i(z_2)e_i(z_1) 
 -e_j(z)e_i(z_3)e_i(z_2)e_i(z_1)
\end{align*}
with the function $h^3_{ij} \ceq \bigl(\wh{g^{123}_{000}}-\wh{1}\bigr)/\bigl(\wh{g^{23}_{00}}-\wh{g^3_0}\bigr)$ given in \eqref{eq:ns:h3}. For the Serre-type relations of type $G_2$, there involve two functions $h^{4,1}_{ij}$ and $h^{4,2}_{ij}$, given in \eqref{eq:G:h4}. See \cref{dfn:U:G} for the detail.

There is one problem on these Serre-type relations which remains unclear. 
For all the type $X_l$ except $G_2$, our Serre-type relations are ``unique'' in the sense that there is no other cubic or quartic relation satisfying the following two conditions.
\begin{itemize}
\item 
It respects the Drinfeld-type comultiplication in \cref{lem:U:Delta}.
\item
It reduces to the Serre-type relation of the elliptic algebra $U_{q,p}(\wh{\frg})$ under the specialization of structure functions $g$ to the theta functions \eqref{eq:U:gijE}.
\end{itemize}
See \cref{rmk:ADE:h,rmk:BCF:h} for the detail. 
However, for the type $G_2$, the compatibility of the Drinfeld-type comultiplication only yields the vanishing condition on a certain linear combination of the functions $h^{4,1}_{ij}$ and $h^{4,2}_{ij}$, and there remains freedom of choice of these functions (the ideal $J$ in \eqref{eq:G:h4}). This problem will be mentioned in \cref{rmk:G:h} and \cref{s:cnc}.

\subsubsection*{Organization}

Let us explain the organization of this note.
The starting \cref{s:pre} is a preliminary part, consisting of three parts. The first part is \cref{ss:pre:Halgd}, where we explain the notion of $\frh$-Hopf algebroids in the sense of Etingof and Varchenko \cite{EV}. We also introduce the basic terminology of Hopf algebras. The second part is \cref{ss:pre:DI}, where we review the Ding-Iohara quantum algebras $U_q(g,A_l)$ \cite{DI}. The last part \cref{ss:pre:E} gives a review of the elliptic algebra $U_{q,p}(\wh{\frg})$. We also give in \cref{prp:E:Halgd} a Hopf algebroid structure on $U_{q,p}(\wh{\frg})$ with Drinfeld-type comultiplication $\Delta$. Some part of the verification of \cref{prp:E:Halgd} will be given in the following sections.

In \cref{s:ADE}, we introduce the dynamical Ding-Iohara algebroid $U_{q,p}(g,X_l)$ for simply-laced type $X_l=A_l,D_l$ or $E_{6,7,8}$. In \cref{ss:ADE:alg}, we give the algebra structure on $U_{q,p}(g,X_l)$ in \cref{dfn:U:ADE}, and show that the elliptic algebra $U_{q,p}(\wh{\frg})$ is a specialization of $U_{q,p}(g,X_l)$ in \cref{prp:U:ADE}. Then, in \cref{ss:ADE:Halgd}, we give the $H$-Hopf algebroid structure, and check in \cref{cor:U:ADE} that the special case $U_{q,p}(\wh{\frg})$ recovers the Drinfeld-type comultiplication $\Delta$ mentioned above. Finally, in \cref{ss:ADE:p=0}, we check that the $p \to 0$ limit of $U_{q,p}(g,X_l)$ recovers the general Ding-Iohara algebra $U_q(\ol{g},X_l)$ with structure functions $\ol{g} \ceq \lim_{p \to 0}g$.

In \cref{s:ns}, we extend our Hopf algebroid $U_{q,p}(g,X_l)$ to non-simply-laced types $X_l=B_l,C_l,F_4$ and $G_2$. After the search of possible Serre-type relations explained in \cref{ss:ns:h}, we treat the former case in \cref{ss:ns:BCF} and the latter in \cref{ss:ns:G}. The organization of each subsection is similar to \cref{ss:ADE:alg} and \cref{ss:ADE:Halgd}, and we omit the explanation. The limit $p \to 0$ is treated in \cref{ss:ns:DI}, where we first introduce the Ding-Iohara quantum algebra for non-simply-laced type $X_l$ in \cref{dfn:ns:DI} and \cref{prp:ns:DI}, and then check that the condition \ref{0:g:2} holds.

We close this note by concluding remarks in \cref{s:cnc}

\subsubsection*{Global notations}

Here are the notations and terminology used throughout in this paper. 
\begin{itemize}[nosep]
\item
We denote by $\bbN = \bbZ_{\ge 0} \ceq \{0,1,2,\ldots\}$ the set of non-negative integers.

\item
We denote by $\delta_{i,j}$ the Kronecker delta on a set $I \ni i,j$.

\item
The symbol $\diag(x_1,\dotsc,x_n)$ denotes the diagonal matrix of size $n$ with diagonal entries $x_1,\dotsc,x_n$.
\item
The word `ring' or `algebra' means a unital associative one unless otherwise stated.
The unit of a ring $R$ is denoted by $1_R$.

\item
For a commutative ring $k$, we use the word ``an algebra over $k$'' and do not use the word ``a $k$-algebra'' to distinguish it from ``an $\frh$-algebra'' for a finite-dimensional Lie algebra $\frh$  which will be explained in \cref{ss:pre:Halgd}.

\item
For a commutative ring $k$ and a variable $x$, we denote the Laurent polynomial ring by $k[x^{\pm1}]$, the formal series ring by $k\dbr{x}$, the Laurent series ring by $k\dpr{x}$, and the module of formal Laurent series by $k\dbr{x^{\pm1}}$.

\item
We use Gasper and Rahman's notation \cite{GR} for $q$-shifted factorials 
\begin{align}\label{eq:ntn:qsf}
 (x;q)_\infty \ceq \prod_{n=0}^\infty (1-x q^n), \quad 
 (x_1,\dotsc,x_n;q)_\infty \ceq \prod_{i=1}^n (x_i;q)_\infty.
\end{align}

\item
We denote by $\frS_n$ the symmetric group of degree $n$.
\end{itemize}

\section{Recollections on Ding-Iohara quantum algebras and elliptic algebras}
\label{s:pre}

\subsection{Hopf algebroid}\label{ss:pre:Halgd}

We begin with the definition of $\frh$-Hopf algebroid for a finite-dimensional Lie algebra $\frh$ in the sense of Etingof and Varchenko \cite[Chap.\ 4]{EV}. 

Let $\frh$ be a finite-dimensional commutative Lie algebra over the complex number field $\bbC$. We denote the field $\clM(\frh^*)$ of meromorphic functions on the linear dual $\frh^*$, regarded as a complex affine space, by
\[
 \bbF \ceq \clM(\frh^*).
\]

An algebra $A$ over $\bbC$ is called an \emph{$\frh$-algebra} if it is equipped with an $\frh^*$-bigrading $A = \bigoplus_{\alpha, \beta \in \frh^*}A_{\alpha, \beta}$, 
and with two algebra embeddings $\mu_l, \mu_r\colon \bbF \to A_{0, 0}$, called \emph{the left} and \emph{right moment maps}, such that 
\begin{align*}
 \mu_l(F)a = a \mu_l(T_{\alpha}F), \quad \mu_r(F)a=a\mu_{r}(T_{\beta}F)
 \quad (F \in \bbF, \, a \in A_{\alpha,\beta}), 
\end{align*}
where $T_{\alpha}$ is the shift operator 
\[
 (T_{\alpha}F)(x) \ceq F(x+\alpha) \quad (x \in \frh^*).
\]
An \emph{$\frh$-algebra morphism $A \to B$} is an algebra homomorphism consistent with the moment maps. Such a morphism automatically preserves the $\frh^*$-bigrading.

We define \emph{the modified tensor product} of two $\frh$-algebras $A$ and $B$ to be the $\frh$-algebra 
\begin{align}\label{eq:Halgd:mtp}
 \bigl(A \totimes B,\bigoplus_{\alpha,\beta \in \frh^*}(A \totimes B)_{\alpha,\beta},
  \mu_l^{A \totimes B},\mu_r^{A \totimes B}\bigr)
\end{align}
as follows. We set
\begin{align}\label{eq:Halgd:to}
 A \totimes B \ceq \bigoplus_{\alpha,\beta \in \frh^*}(A \totimes B)_{\alpha,\beta}, \quad
(A \totimes B)_{\alpha,\beta} \ceq 
 \bigoplus_{\gamma \in \frh^*}A_{\alpha,\gamma} \otimes_{\bbF} B_{\gamma,\beta},  
\end{align}
where $A_{\alpha,\gamma} \otimes_{\bbF} B_{\gamma,\beta}$ is the quotient of the standard tensor product $A_{\alpha,\gamma} \otimes_{\bbC} B_{\gamma,\beta}$ of linear spaces over $\bbC$ by the linear subspace generated by the elements
\begin{align}\label{eq:Halgd:totimes}
 \mu^A_r(F)a \otimes b - a \otimes \mu^B_l(F)b \quad (a \in A, \, b \in B, \, F \in \bbF). 
\end{align}
We denote an element of $A \totimes B$ by $a \totimes b$ for $a \in A$ and $b \in B$. The modified  tensor product $A \totimes B$ is again an $\frh$-algebra by the multiplication
\[
 (a \totimes b)(a' \totimes b') \ceq (a a') \totimes (b b')
\]
and the moment maps
\begin{align}\label{eq:Halgd:mu-totimes}
 \mu^{A \totimes B}_l(F) \ceq \mu_l^A(F) \totimes 1_B, \quad 
 \mu^{A \totimes B}_r(F) \ceq 1_A \totimes \mu_r^B(F)  \quad (F \in \bbF).
\end{align}
We also have the natural notion of \emph{the modified tensor product $f \totimes g$ of $\frh$-algebra homomorphisms $f$ and $g$}, whose detail is omitted.

The category of $\frh$-algebras has a monoidal structure whose tensor product is the above $\totimes$ and the unit object is \emph{the $\frh$-algebra $D_\frh$ of difference operators on the meromorphic function field $\bbF=\clM(\frh^*)$}. As a $\bbC$-linear space, $D_\frh$ is defined to be the subspace of $\bbC$-linear endomorphisms $\End_{\bbC}(\bbF)$ spanned by the operators of the form $f T_{\alpha}$ with $f \in \bbF$ and $\alpha \in \frh^*$. It is a subalgebra of $\End_{\bbC}(\bbF)$, and has the $\frh^*$-bigrading
\begin{align}\label{eq:Halgd:Dh}
 D_\frh = \bigoplus_{\alpha,\beta \in \frh^*} (D_\frh)_{\alpha,\beta}, \quad 
(D_\frh)_{\alpha,\beta} \ceq 
 \begin{cases} \{f T_{-\alpha} \mid f \in \bbF\} & (\alpha = \beta) \\ 0 & (\alpha \ne \beta) \end{cases}.
\end{align}
Note that $(D_\frh)_{0,0}=\bbF T_0 =\bbF \id_{\bbF}$. The moment maps are given by
\[
 \mu_l=\mu_r\colon \bbF \lto (D_\frh)_{0,0} = \bbF T_0, \quad f \lmto f T_0.
\]
These data give $D_\frh$ a structure of an $\frh$-algebra, and it is a monoidal unit with respect to the modified tensor product $\totimes$ \cite[Lemma 4.1]{EV}.

An \emph{$\frh$-bialgebroid} is a triple $(A,\Delta,\ve)$ consisting of an $\frh$-algebra $A$ and two $\frh$-algebra homomorphisms $\Delta\colon A \to A \totimes A$ and $\ve\colon A \to D_\frh$ satisfying \emph{the coassociativity and counit axioms}
\begin{align}\label{eq:Halgd:Delta-ve}
\begin{split}
 (\Delta \totimes \id_A) \Delta = (\id_A \totimes \Delta) \Delta &: A \lto 
 (A \totimes A) \totimes A \cong A \totimes (A \totimes A), 
\\
 (\ve \totimes \id_A)\Delta = \id_A = (\id_A \totimes \ve)\Delta &: A \lto 
 D_{\frh} \totimes A \cong A \cong A \totimes D_{\frh},
\end{split}
\end{align} 
where $\totimes$ denotes the modified tensor product of $\frh$-algebra homomorphisms, and $\cong$'s denote the natural isomorphisms. 
The map $\Delta$ is called \emph{the comultiplication}, and $\ve$ is called \emph{the counit} of the $\frh$-bialgebroid $A$.

Let $(A,\Delta,\ve)$ be an $\frh$-bialgebroid. A linear map $S \in \End_{\bbC}A$ is called an \emph{antipode} if the following relations hold for $f \in \bbF$ and $a \in A_{\alpha,\beta}$.
\begin{align}\label{eq:Halgd:S}
\begin{split}
 S(\mu_r(f)a) = S(a)\mu_l(f), \quad &S(a\mu_l(f)) = \mu_r(f)S(a), \\
 m(\id_A \otimes S)\Delta(a) = \mu_l(\ve(a)1), \quad 
 &m(S \otimes \id_A)\Delta(a) = \mu_r\bigl(T_{\alpha}(\ve(a)1)\bigr),
\end{split}
\end{align}
where $\ve(a)1$ denotes the action of the operator $\ve(a) \in D_{\frh}$ on the constant function $1$, and $\otimes$ in the second line denotes the ordinary tensor product of linear maps over $\bbC$. An $\frh$-bialgebroid equipped with an antipode is called an \emph{$\frh$-Hopf algebroid}.

Note that for $\frh=0$, the zero-dimensional commutative Lie algebra, we have $\bbF=\bbC$ and $\frh$-algebra is nothing but an algebra $A$ over $\bbC$ whose moment maps are the unit $\eta\colon \bbC \to A$. Then the modified tensor product $\totimes$ is equal to the standard tensor product $\otimes_{\bbC}$ and $D_\frh=\bbC$. Moreover, an $\frh$-bialgebroid $(A,\Delta,\ve)$ is a bialgebra over $\bbC$ with comultiplication $\Delta\colon A \to A \otimes A$ and counit $\ve\colon A \to \bbC$, and an $\frh$-Hopf algebroid $(A,\Delta,\ve,S)$ is a Hopf algebra with antipode $S$. Thus, an ordinary Hopf algebra is a special case of $\frh$-Hopf algebroid.

To close this \cref{ss:pre:Halgd}, let us mention the relation of $\frh$-bialgebroids and Hopf algebroids in the sense of \cite{Lu,Xu}. As far as we know, what is today commonly called a Hopf algebroid means the one introduced in loc.\ cit. Our presentation follows \cite[\S 3]{Xu}.

A Hopf algebroid is a tuple $(A,R,\alpha,\beta,m,\Delta,\ve)$ consisting of:
\begin{clist}
\item 
An algebra $(A,m)$.
\item
Another algebra $R$.
\item
An algebra homomorphism $\alpha\colon R \to A$ and an algebra anti-homomorphism $\beta\colon R \to A$ such that the images of $\alpha$ and $\beta$ commute in $A$.
\end{clist}
At this point, the algebra $A$ has an $R$-$R$-bimodule structure with $r.a \ceq \alpha(r)a$ and $a.r \ceq \beta(r)a$ for $a \in A$ and $r \in R$. In particular, we have the $R$-$R$-bimodule product $A \otimes_R A$, which admits a natural $R$-$R$-bimodule structure.
\begin{clist}[resume]
\item 
An $R$-$R$-bimodule homomorphism $\Delta\colon A \to A \otimes_R A$.

\item
An $R$-$R$-bimodule homomorphism $\ve\colon A \to R$.
\end{clist}
These should satisfy the following conditions.
\begin{enumerate}
\item
The map $\Delta$ satisfies $\Delta(1_A) = 1_A \otimes 1_A$ and $(\Delta \otimes_R \id_A)=(\id_A \otimes_R \Delta)$ under the natural isomorphism $(H \otimes_R H) \otimes_R H \cong H \otimes_R (H \otimes_R H)$.

\item
The maps $m$ and $\Delta$ are compatible in the sense that $\Delta(a)\bigl(\beta(r) \otimes 1_A - 1_A \otimes \alpha(r)\bigr)=0$ for any $r \in R$ and $a \in A$, and that $\Delta(a b)=\Delta(a)\Delta(b)$ for any $a,b \in A$.

\item
The map $\ve$ satisfies $\ve(1_A)=1_R$ and $(\ve \otimes_R \id_A)\Delta = \id_A = (\id_A \otimes_R \ve)\Delta$ under the natural isomorphisms $R \otimes_R A \cong A \cong A \otimes_R R$.
\end{enumerate}

An $\frh$-bialgebroid $(A,\Delta,\ve)$ is an Hopf algebroid in the above sense by setting
$R \ceq \bbF = \clM(\frh)$ and $\alpha \ceq \mu_l$, $\beta \ceq \mu_R$. In fact, $\bbF$ is commutative so that the commuting condition in (iii) holds. Then we find that $A \otimes_{\bbF} A \cong A \totimes A$ (see \eqref{eq:Halgd:totimes}). The conditions (1)--(3) are checked immediately.

\subsection{Ding-Iohara quantum algebras}\label{ss:pre:DI}

Here we give an overview of the quantum algebras introduced by Ding and Iohara in \cite{DI}.

For $l \in \bbZ_{>0}$, let $(a_{ij})_{i,j=1}^l$ be the Cartan matrix of type $A_l$, i.e., 
\[
 a_{ij} = \begin{cases} 2 & (i=j) \\ {-1} & (\abs{i-j}=1) \\ 0 & (\text{otherwise}) \end{cases}.
\]
We will also denote it as $(a_{ij})_{i,j \in I}$ with $I \ceq \{1,\dotsc,l\}$.

Also, we denote 
\begin{align}\label{eq:DI:clA}
 \clA \ceq 
 \{\text{analytic functions of $z$ without poles except $z=0$ or $\infty$}\},
\end{align}
regarded as a commutative algebra over $\bbC$. Let 
\[
 g = \{g_{ij}(z) \mid i,j \in I\}
\]
be a set of analytic functions satisfying the following condition.
\begin{clist}
\item
There are some $G_{ij}^{\pm}(z) \in \clA$ such that 
\begin{align}\label{eq:DI:cond0}
 g_{ij}(z) = G_{ij}^+(z)/G_{ij}^-(z).
\end{align}
\item
We have
\begin{align}\label{eq:DI:cond}
 g_{ij}(z) = g_{ji}(z^{-1})^{-1}.
\end{align}
\end{clist}
We call the set $g$ \emph{the structure functions}, and call the condition on $g$ \emph{the Ding-Iohara condition}. Note that the condition $G^{\pm}_{ij}(z) \in \clA$ implies that we can expand $G^{\pm}_{ij}(z)$ and $g_{ij}(z) = G_{ij}^+(z)/G_{ij}^-(z)$ as
\begin{align}\label{eq:DI:gexp}
 G^{\pm}_{ij}(z) = \sum_{n=-L^{\pm}}^{M^{\pm}} G^{\pm}_{ij,n} z^n \in \bbC[z^{\pm1}], \quad 
 g_{ij}(z)^{\pm1} = \sum_{n=-N^{\pm}}^\infty g_{ij,n} z^n \in \bbC\dpr{z}
\end{align}
for some $L^{\pm}, M^{\pm}, N^{\pm} \in \bbN$.

Let also $q$ be a complex number such that $\abs{q} \ne 0,1$, called \emph{the quantum parameter}. 
Using these data $A,g$ and $q$, we define the topological algebra $U_q(g,A_l)$ as follows.
See \cref{rmk:DI:top} for the meaning of the phrase ``topological''.

\begin{dfn}\label{dfn:DI:U}
Using the Cartan matrix $A=(a_{ij})_{i,j \in I}$ of type $A_l$, the structure functions $g$ and the quantum parameter $q$,  we define 
\[
 U = U_q(g,A_l), 
\]
to be the topological algebra over $\bbC$ generated by the elements 
\begin{align}\label{eq:DI:gen}
 q^{\pm c/2}, \ e_{i,n}, \ f_{i,n}, \ \psi^+_{i,n}, \ \psi^-_{i,n} \quad (i \in I, \, n \in \bbZ)
\end{align}
with the following defining relations. 
We will use the \emph{generating currents} 
\begin{align}\label{eq:DI:gc}
 e_i(z) \ceq \sum_{n \in \bbZ}e_{i,n}z^{-l}, \quad 
 f_i(z) \ceq \sum_{n \in \bbZ}f_{i,n}z^{-l}, \quad 
 \psi^\pm_i(z) \ceq \sum_{n \in \bbZ}\psi^\pm_{i,n} z^{-n},
\end{align}
where $z$ is a formal variable.
\begin{itemize}
\item 
The generators $q^{\pm c/2}$ satisfy that
\begin{align}\label{eq:DI:qc}
 \text{$q^{\pm c/2}$ are central and $q^{c/2}q^{-c/2}=1$.} 
\end{align}
Below we denote $q^{r c/2} \ceq (q^{c/2})^r$ for $r \in \bbZ$.

\item
The generating currents $\psi^{\pm}(z)$ satisfy that 
\begin{align}
 \text{$\psi^{\pm}(z)$ are invertible as series of $z$,}
\end{align}
which means that $\psi^{\pm}(z)$ are bounded above or bounded below, and the initial terms are invertible in $U$. Explicitly, in the case $\psi^+(z)$ is bounded above, we have
\begin{align*}
 \text{there is some $N \in \bbZ$ such that $\psi^+_{i,n}=0$ for $n>N$, 
       and $\psi^+_{i,N}$ is invertible in $U$.}
\end{align*}
\item 
For any $i,j \in I$, we have in terms of the formal variables $z$ and $w$ that 
\begin{align}
 \psi^+_i(z) \psi^+_j(w) = \psi^+_j(w) \psi^+_i(z), \quad 
&\psi^-_i(z) \psi^-_j(w) = \psi^-_j(w) \psi^-_i(z), 
\\
 \psi^+_i(z) \psi^-_j(w)   \psi^+_i(z)^{-1} \psi^-_j(w)^{-1} = \ 
&g_{ij}(q^{-c}z/w) g_{ij}(q^c z/w)^{-1}, 
\\
\label{eq:DI:ppe}
 \psi^+_i(z) e_j(w) = g_{ij}(q^{-c/2}z/w)^{  } e_j(w) \psi^+_i(z), \quad 
&\psi^+_i(z) f_j(w) = g_{ij}(q^{ c/2}z/w)^{-1} f_j(w) \psi^+_i(z), 
\\
\label{eq:DI:pme}
 \psi^-_i(z) e_j(w) = g_{ij}(q^{ c/2}z/w)^{  } e_j(w) \psi^-_i(z), \quad 
&\psi^-_i(z) f_j(w) = g_{ij}(q^{-c/2}z/w)^{-1} f_j(w) \psi^-_i(z), 
\\
\label{eq:DI:ee}
 G_{ij}^-(z/w) e_i(z) e_j(w) = G_{ij}^+(z/w) e_j(w) e_i(z), \quad 
&G_{ij}^+(z/w) f_i(z) f_j(w) = G_{ij}^-(z/w) f_j(w) f_i(z), 
\\
\label{eq:DI:ef}
 [e_i(z),f_j(w)] = \frac{\delta_{i,j}}{q-q^{-1}} 
 \bigl(\delta(q^{-c}z/w) &\psi^-_i(q^{c/2}w) - 
       \delta(q^{ c}z/w) \psi^+_i(q^{c/2} z) \bigr).
\end{align}
In each equality, we expand both sides as series of $z$ and $w$ using \eqref{eq:DI:gexp} and \eqref{eq:DI:gc}, and compare the coefficients of $z^r w^s$ for each $(r,s) \in \bbZ^2$ to get genuine relations. 
Here we should consider an appropriate topology for general structure functions $g$, which will be explained in \cref{rmk:DI:top}. Also, in \eqref{eq:DI:ef}, we used the commutator $[X,Y] \ceq X Y- Y X$ and the formal delta function
\begin{align}\label{eq:DI:delta}
 \delta(x) \ceq \sum_{n \in \bbZ} x^n.
\end{align}

\item
For $i,j \in I$ with $a_{ij}=0$, we have in terms of the formal variables $z$ and $w$ that 
\begin{align}\label{eq:DI:eeff}
 [e_i(z),e_j(w)] = [f_i(z),f_j(w)] = 0.
\end{align}
The meaning of these equalities are the same as the previous item.

\item
For $i,j \in I$ such that $a_{ij}={-1}$, we have in terms of the formal variables $z_1,z_2$ and $w$ \emph{the Serre-type} (or \emph{cubic}) \emph{relations} 
\begin{align}\label{eq:DI:Serre}
\begin{split}
 X_i(z_1) X_i(z_2) X_j(w) &- h_{ij}(\tfrac{z_1}{w},\tfrac{z_2}{w}) X_i(z_1) X_j(w) X_i(z_2) 
+X_j(w) X_i(z_1) X_i(z_2) \\ &+ (\text{three terms given by $z_1 \lrto z_2$}) = 0
\end{split}
\end{align}
for $X=e$ or $f$. Here we set 
\begin{align}\label{eq:DI:hij}
 h_{ij}(x,y) \ceq \frac{(g_{ii}(x/y)+1) (g_{ij}(x) g_{ij}(y)+1)}{g_{ij}(y)+g_{ii}(x/y)g_{ij}(x)}, 
\end{align}
regarded as a series belonging to $\bbC\dpr{x,y,x/y}$ by expanding $g_{ij}$ and $g_{ii}$ as in \eqref{eq:DI:gexp}.
\end{itemize}
\end{dfn}

\begin{rmk}\label{rmk:DI:top}
Let us comment on the topology of $U$. As mentioned after \eqref{eq:DI:ef}, the genuine relations are obtained by expanding the current relations. For example, the first equality of \eqref{eq:DI:ee} yields
\[
 \sum_{k=-L^-}^{M^-} G_{ij,k}^- e_{i,m} e_{j,n} = \sum_{k'=-L^+}^{M^+} G_{ij,k}^+ e_{j,n+k-k'} e_{i,m-k+k'}
\]
for each $m,n \in \bbZ$, under the expansion \eqref{eq:DI:gexp}.
This involves finite summations only, and we can apply the ordinary diamond lemma \cite{Be}. 
On the other hand, the first equality in \eqref{eq:DI:ppe} yields
\begin{align}\label{eq:DI:infsum}
 \psi_{i,m}^+ e_{j,n} = \sum_{k=-N^+}^\infty g_{ij,k} q^{-k c/2} e_{j,n-k} \psi^+_{i,m+k}
\end{align}
for each $m,n \in \bbZ$. Thus, unless $\psi^+_i(z)$ is bounded above or $e_j(z)$ is bounded on both sides, we have an infinite summation in the defining relation, and cannot obtain a genuine algebra. Largely speaking, there are two solutions of this problem.
\begin{clist}
\item \label{i:rmk:DI:top:b}
Require $\psi^+_i(z)$ to be bounded above, and $\psi^-_i(z)$ to be bounded below, so that the genuine relations contain finite summations only. This is the case of the Drinfeld realization of quantum affine algebras \cite{D}, recalled in \cref{eg:DI:qaff}.

\item \label{i:rmk:DI:top:p}
Equip the structure functions $g_{ij}(z)$ with a parameter, say $p$, and require $g_{ij}(z)=g_{ij}(z;p)$ to have a series expansion $g_{ij}(z;p)=\sum_{k=-N}^\infty g_{ij,k}(z) p^k \in \bbC[z^{\pm1}]\dpr{p}$. Now we equip $U$ with the $p$-adic topology. Then, since $g_{ij,k}(z) \in \bbC[z^{\pm1}]$, we see that each $p$-adic component of the relations like \eqref{eq:DI:infsum} contains finite summation only, Thus we can regard $U$ as a complete algebra with respect to the $p$-adic topology. We only consider this situation in the following text.
\end{clist}
In a general situation, where no parameters are available, we are not sure how to resolve the problem. One possible solution may be to use the diamond lemma adapted to the formal power series \cite{H}, which requires an appropriate choice of the norm on the formal power series ring.
\end{rmk}

\begin{rmk}
There are some typos in the quadratic relations of \cite[Definitions 2.1,2.2]{DI}. We should define the relations between the generating currents $\psi^-_i(z)$ and $e_j(w),f_j(w)$ as \eqref{eq:DI:pme}, which are consistent with those of the Drinfeld realization of quantum affine algebras \cite[\S2, Remark]{D} and of the quantum toroidal algebras of simply-laced type \cite[\S 3]{GKV}.
\end{rmk}

\begin{fct}[{\cite[Theorem 2.2]{DI}}]\label{fct:DI:U}
The topological algebra $U=U_q(g,A_l)$ has the following topological Hopf algebra structure.
\begin{itemize}
\item Formal comultiplication $\Delta$:
\begin{gather}
\label{eq:DI:Delta0}
 \Delta(q^{\pm c/2}) \ceq q^{c/2} \otimes q^{\pm c/2}, \quad 
 \\
\label{eq:DI:Delta1}
 \Delta(e_i(z)) \ceq e_i(z) \otimes 1 + \psi_i^+(q^{c_1/2} z) \otimes e_i(q^{c_1} z), \quad
 \Delta(f_i(z)) \ceq 1 \otimes f_i(z) + f_i(q^{c_2} z) \otimes \psi^-_i(q^{c_2/2}z),
 \\
\label{eq:DI:Delta2}
 \Delta(\psi^+_i(z)) \ceq \psi^+_i(q^{-c_2/2}z) \otimes \psi^+_i(q^{ c_1/2}z), \quad
 \Delta(\psi^-_i(z)) \ceq \psi^-_i(q^{ c_2/2}z) \otimes \psi^-_i(q^{-c_1/2}z), 
\end{gather}
where $q^{r c_1/2} \ceq q^{r c/2} \otimes 1$ and $q^{r c_2/2} \ceq 1 \otimes q^{r c/2}$ for $r \in \bbZ$. In each equality, we expand both sides to get the value of $\Delta$ on each generator. See \cite{DI0,DI} for the detail of such a calculation. Since infinite summations appear in this calculation, the obtained comultiplication $\Delta$ is called formal. The infinite sums make sense if we take into account the topology given by a good gradation on the generators \eqref{eq:DI:gen}.

\item Counit $\ve$:
\begin{align}
 \ve(q^c) \ceq 1, \quad \ve(\psi^{\pm}_i(z)) \ceq 1, \quad \ve(e_i(z))=\ve(f_i(z))=0.
\end{align}
Again, we expand the left hand sides to get the value of $\ve$ on each generator.

\item Formal antipode $S$:
\begin{align}
 S(q^c) \ceq q^c \otimes q^c, \quad 
&S(\psi^\pm_i(z)) \ceq \psi^\pm_i(z)^{-1} \\
\label{eq:DI:S}
 S(e_i(z)) \ceq -\psi^+_i(q^{-c/2}z)^{-1} e_i(q^{-c}z), \quad 
&S(f_i(z)) \ceq -f_i(q^{-c}z)\psi_i^-(q^{-c/2}z)^{-1}.
\end{align}
\end{itemize}
We call this topological Hopf algebra $U_q(g,A_l)$ \emph{the Ding-Iohara quantum algebra of type $A_l$},
and \emph{the Ding-Iohara algebra} for short.
\end{fct}

Let us give an example of the Ding-Iohara algebra.

\begin{eg}\label{eg:DI:qaff}
The motivational example of the algebra $U=U_q(g,A_l)$ is given in \cite[Example 2.1, Case I]{DI} by setting the structure functions $g=\{g_{ij}(z) \mid i,j \in I\}$ to be  
\begin{align*}
 g_{ij}(z) = \frac{G^+_{ij}(z)}{G^-_{ij}(z)} \ceq 
 \frac{q^{-a_{ij}} (1-q^{a_{ij}}z)}{1-q^{-a_{ij}}z}.
\end{align*}
If we quote this $U$ by the relations $\psi^{\pm}_{i,n}=0$ for $n>0$ and $\psi^+_{i,0}\psi^-_{i,0}=1$, then the resulting object is isomorphic as a topological Hopf algebra to the quantum affine algebra $U_q(\wh{\fsl}_{l+1})$ equipped with Drinfeld comultiplication. See also \cite{DI0} for the related discussion. 
\end{eg}

Let us mention that the quantum toroidal algebra $U_{q,t}(\fgl_{1,\tor})$, also called the Ding-Iohara-Miki algebra, is a version of the Ding-Iohara quantum algebra of type $A_1$ in the above meaning. This quantum algebra is nowadays under extensive studies. See \cite{FFJMM1,FHHSY,FT,M} for example. The structure function is set to be
\begin{align}\label{eq:DI:qtor}
 g(z) \ceq \frac{(1-q z)(1-z/t)(1-t z/q)}{(1-z/q)(1-t z)(1-q z/t)}.
\end{align}
The additional parameter $t$ corresponds to the (Hecke) parameter of the Macdonald symmetric polynomials $P_\lambda(x;q,t)$. 

\begin{rmk}\label{rmk:DI:ADE}
As mentioned in the last paragraph of \cite{DI}, replacing the root system $A_l$ by the simply-raced system $X_l=D_l$ or $E_{6,7,8}$, we get a topological Hopf algebra $U_q(g,X_l)$, which we call \emph{the Ding-Iohara quantum algebra of the simply-laced type $X_l$}.
\end{rmk}

\subsection{Elliptic algebras \texorpdfstring{$U_{q,p}(\wh{\frg})$}{Uqp(g)}}\label{ss:pre:E}

Here we explain the elliptic algebra $U_{q,p}(\wh{\frg})$ associated to each untwisted affine Lie algebra $\wh{\frg}$. As mentioned in \cref{s:0}, there are many objects named ``elliptic quantum groups'' or ``elliptic algebras'' in literature. The algebra $U_{q,p}(\wh{\frg})$ was introduced first in \cite{JKOS} as a quasi-Hopf twist of the quantum affine algebra $U_q(\wh{\frg})$ associated to $\wh{\frg}$ with Drinfeld comultiplication. In the paper \cite[\S2]{FKO}, the algebra $U_{q,p}(\wh{\frg})$ was reformulated as a Hopf algebroid (see \cref{ss:pre:Halgd}), and our explanation mainly follows this Hopf algebroid formulation.

\subsubsection{Preliminaries on finite root systems}\label{sss:E:root}

Let us consider the Cartan matrix $A=(a_{ij})_{i,j=1}^l$ of finite type $X_l$ in the sense of Kac \cite[Theorem 4.3, \S 4.8, TABLE Fin]{Ka}. Hence, we have $X_l=A_{l \ge 1}$, $B_{l \ge 3}$, $C_{l \ge 2}$, $D_{l \ge 4}$, $E_{6,7,8}$, $F_4$ or $G_2$ as usual. We denote the index set as 
\[
 I \ceq \{1,2,\dotsc,l\}, 
\]
so that we can write $A=(a_{ij})_{i,j \in I}$. For definiteness, we write down some examples of the Cartan matrix $A=A(X_l)$ of non-simply-laced type $X_l$. Note that, except for type $G_2$, they are transpose of those in Bourbaki \cite{B}.
\[
 A(B_4) = \begin{psm} 2 & {-1} \\ {-1} & 2 & {-1} \\ & {-1} & 2 & {-1} \\ & & {-2} & 2 \end{psm}, 
 \quad 
 A(C_4) = \begin{psm} 2 & {-1} \\ {-1} & 2 & {-1} \\ & {-1} & 2 & {-2} \\ & & {-1} & 2 \end{psm}, 
 \quad 
 A(F_4) = \begin{psm} 2 & {-1} \\ {-1} & 2 & {-1} \\ & {-2} & 2 & {-1} \\ & & {-1} & 2 \end{psm}, 
 \quad 
 A(G_2) = \begin{psm} 2 & {-1} \\ {-3} & 2 \end{psm}.
\]

Let us take a symmetrization $A = D B$ of the Cartan matrix $A$ \cite[\S2.3]{Ka}. In other words, $B$ is symmetric, $D$ is invertible and diagonal, and $B$ and $D$ have rational entries. We denote the entries of $B$ and $D$ as $B = (b_{ij})_{i,j \in I}$ and $D^{-1} = \diag(d_i)_{i \in I}$. Note that $d_i$ is defined to be the entry of the \emph{inverse} of $D$, so that we have 
\begin{align}\label{eq:E:BDA}
 b_{ij} = d_i a_{ij}, \quad b_{ij}=b_{ji} \quad (i,j \in I),
\end{align}
matching the convention in \cite[A.1]{JKOS} and \cite{FKO}. Also, recall from \cite[Remark 2.3]{Ka} that $D$ is uniquely determined by the condition $\trs{(D^{-1} A)}=D^{-1} A$ up to a constant factor, where $\trs{M}$ denotes the transpose of a matrix $M$.

Any choice of $D$ will work in the following discussion (see \cref{rmk:E:U} \ref{i:rmk:E:U:d'} for the detail). For definiteness, we make the following choice. Let $\wh{A}=(a_{ij})_{i,j \in \wh{I}}$, $\wh{I} \ceq \{0\} \cup I$ be the generalized Cartan matrix of affine type $X_l^{(1)}$ in the sense of Kac \cite[Theorem 4.3, \S 4.8, TABLE Aff1]{Ka}. We have the labels $a_0,a_1,\dotsc,a_l$ and colabels $a_0^\vee,a_1^\vee,\dotsc,a_l^\vee$ \cite[\S 4.8, \S 6.1]{Ka},  which are the sequences of positive coprime integers uniquely determined by the conditions 
\[
 \wh{A} \, \trs{(a_0,a_1,\dotsc,a_l)} = \trs{(0,0,\dotsc,0)}, \quad 
 (a_0^\vee,a_1^\vee,\dotsc,a_l^\vee) \, \wh{A} = (0,0,\dotsc,0).
\]
Then, by \cite[(6.1.2)]{Ka}, the diagonal matrix $\wh{D}^{-1} \ceq \diag(a_0^\vee/a_0,a_1^\vee/a_1,\dotsc,a_l^\vee/a_l)$ gives a symmetrization $\wh{A}=\wh{D} \wh{B}$. Since $A=\rst{\wh{A}}{I \times I}$, we obtain the symmetrization $A=D B$ with 
\[
 D^{-1} = \diag(d_1,\dotsc,d_n) = \diag(a_1^\vee/a_1,a_2^\vee/a_2,\dotsc,a_l^\vee/a_l).
\]
Explicitly, the entries $d_i$ are given as:
\begin{align}\label{eq:E:D}
\begin{split}
 A_l, D_l, E_{6,7,8}: &\quad (d_1,\dotsc,d_l)=(1,\dotsc,1), \\ 
 B_l: &\quad (d_1,\dotsc,d_l)=(1,\dotsc,1,\thf), \\
 C_l: &\quad (d_1,\dotsc,d_l)=(1,\dotsc,1,2), \\
 F_4: &\quad (d_1,\dotsc,d_4)=(1,1,\thf,\thf) \\
 G_2: &\quad (d_1,d_2)=(1,\tfrac{1}{3}). 
\end{split}
\end{align}
Note that our choices for types $A,D,E$ and $B,C$ coincide with those in \cite{FKO}.

We denote by $\frg$ the Kac-Moody Lie algebra associated to $A=(a_{ij})_{i,j \in I}$, and by $(\frh,\Pi,\Pi^\vee)$ a fixed realization of $A$ \cite[\S 1.1]{Ka}. Thus, $\frg$ is a simple finite-dimensional Lie algebra of type $X_l$, $\frh \subset \frg$ is the Cartan subalgebra with $\dim \frh=l$, and $\Pi = \{\alpha_i \mid i \in I\} \subset \frh^*$ and $\Pi^\vee = \{\alpha^\vee_i \mid i \in I\} \subset \frh$ are linearly independent indexed subsets of the dual space $\frh^*$ and of $\frh$ satisfying 
\[
 \pair{\alpha_i,\alpha^\vee_j} = a_{ji} \quad (i,j \in I).
\]
Here $\pair{\cdot,\cdot}\colon \frh^* \times \frh \to \bbC$ denotes the canonical pairing of the linear space $\frh$ and its linear dual $\frh^*$. The elements $\alpha_i$'s are called \emph{simple roots}, and $\alpha^\vee_i$'s are called \emph{simple coroots}.

We denote by $\ipd{\cdot}{\cdot}\colon \frh \times \frh \to \bbC$ the invariant form of $\frg$ with the normalization 
\begin{align}\label{eq:E:ipd}
 \ipd{\alpha^\vee_i}{\alpha^\vee_j} = b_{i j} \quad (i,j \in I).
\end{align}
Since $\ipd{\cdot}{\cdot}$ is non-degenerate, it induces a linear isomorphism 
\begin{align}\label{eq:E:nu}
 \nu\colon \frh \lsto \frh^* \quad \text{s.t.} \quad 
 \pair{\nu(h),h'} = \ipd{h}{h'} \quad (h,h' \in \frh).
\end{align} 
The induced bilinear form on $\frh$ is denoted by the same symbol $\ipd{\cdot}{\cdot}$. 
We have
\begin{align}\label{eq:E:ipd*}
 \alpha_i = \nu(\alpha^\vee_i), \quad 
 \ipd{\alpha_i}{\alpha_j} = b_{ij} \quad (i,j \in I).
\end{align}

\subsubsection{Dynamical parameters}\label{sss:E:dyn}

We keep the notation in the last \cref{sss:E:root}. In particular, $I=\{1,\dotsc,l\}$, and $A=(a_{ij})_{i,j \in I}$ denotes the Cartan matrix of finite type $X_l$. We also fix a realization $(\frh,\Pi,\Pi^\vee)$ of $A$, and denote by $\frg$ the corresponding simple Lie algebra.

Let $R \subset \frh^*$ be the root system of $\frg$, which is a reduced irreducible root system of type $X_l$ in the sense of Bourbaki \cite{B}. The set $\Pi = \{\alpha_i \mid i \in I\} \subset R$ gives the simple roots of $R$, which spans \emph{the root lattice} $\clQ \ceq \sum_{i \in I}\bbZ \alpha_i \subset \frh^*$. Similarly, the set $\Pi^\vee = \{\alpha^\vee_i \mid i \in I\} \subset \frh$ gives the simple coroots of $R$, spanning \emph{the coroot lattice} $\clQ^\vee \ceq \sum_{i \in I}\bbZ \alpha^\vee_i \subset \frh$. Then \emph{the weight lattice} $\clP \subset \frh^*$ is defined to be $\clP \ceq \{\lambda \in \frh^* \mid \forall i \in I, \, \pair{\lambda,\alpha_i^\vee} \in \bbZ\}$, where $\pair{\cdot,\cdot}\colon \frh^* \times \frh \to \bbC$ is the canonical pairing as before. We have $\clP = \sum_{i \in I}\bbZ \varpi_i \subset \frh^*$, where $\varpi_i \in \frh^*$ is \emph{the fundamental weight} determined uniquely by the condition $\pair{\varpi_i,\alpha_j^\vee}=\delta_{i,j}$ for any $j \in I$. Recall the well-known inclusions
\[
 \clQ \subset \clP \subset \clQ \otimes_{\bbZ} \bbQ \subset \frh^*.
\]

Let us consider a copy of the $\bbC$-linear space $\frh^*$, regarded as $\frh^* = \clP \otimes_{\bbZ} \bbC$, and denote it by 
\begin{align}\label{eq:E:H}
 H \ceq \{P_\mu \mid \mu \in \frh^*\},
\end{align}
following the notation in \cite{FKO,Ko}. Thus we have $P_{c_1 \mu_1 + c_2 \mu_2}=c_1 P_{\mu_1}+c_2 P_{\mu_2}$ for $c_1,c_2 \in \bbC$ and $\mu_1,\mu_2 \in \frh^*$. A basis of $H$ is given by
\begin{align}\label{eq:E:Hbas}
 P_{\alpha_1},\dotsc,P_{\alpha_l}.
\end{align}

Next, we introduce the $\bbC$-linear space $\wtH$ of dimension $2l$ spanned by the expressions 
$P_i+h_i$ and $P_i$ ($i \in I$):
\begin{align}\label{eq:E:wtH}
 \wtH \ceq \sum_{i \in I}\bbC\bigl(P_i+h_i\bigr) + \sum_{i \in I}\bbC P_i.
\end{align}
Following the notation in \cite{FKO,Ko}, for $h=\sum_{i \in I}c_i \alpha^\vee_i \in \frh$, we denote $P+h \ceq \sum_{i \in I}c_i(P_i+h_i)$. Identifying $(P_i+h_i)-P_i = h_i$ with $\alpha^\vee_i$, we have an isomorphism 
\begin{align}\label{eq:E:wtH=Hh}
 \wtH \cong H \oplus \frh.
\end{align}
Under this identification, the linear space $\wtH$ has the basis
\begin{align}\label{eq:E:wtHbas}
 \{P_i, h_i \mid i \in I\}, \quad P_i = P_{\alpha_i}, \quad h_i \cong \alpha_i^\vee,
\end{align}
which extends the basis \eqref{eq:E:Hbas} of $H$. The elements $P_i$ are called \emph{the dynamical parameters}.

Hereafter we regard $\wtH$ as a commutative Lie algebra over $\bbC$, and consider the field $\clM(\wtH^*)$ of meromorphic functions on the dual space of $\wtH$ as in \cref{ss:pre:Halgd}. We can express an element $F \in \clM(\wtH^*)$ as a function 
\begin{align}\label{eq:E:finF}
 F = F(P+h,P)
\end{align}
of variables $P+h = (P_i+h_i)_{i \in I}$ and $P = (P_i)_{i \in I}$. Then the value of $F$ at $\mu \in \wtH^*$ can be written as $F(\mu) = F\bigl(\pair{\mu,P+h},\pair{\mu,P}\bigr)$, where $\pair{\cdot,\cdot}\colon \wtH^* \times \wtH \to \bbC$ denotes the canonical pairing. In particular, we regard the elements of $\wtH$ as meromorphic functions on $\wtH^*$. We also denote by 
\begin{align}\label{eq:E:clM}
 f(P+h), \, f(P) \ \in \ \clM(\wtH^*)
\end{align}
a meromorphic function containing the variables $P+h=(P_i+h_i)_{i \in I}$ only, and one containing the variables $P=(P_i)_{i \in I}$ only.

Let us also introduce the dual spaces $H^*$ and $\wtH^*$ as 
\begin{align}\label{eq:E:H*}
\begin{split}
&H^* = \{Q_\alpha \mid \alpha \in \frh^*\} = \sum_{i \in I}\bbC Q_i, \quad 
 Q_i \ceq Q_{\alpha_i}, \\ 
&\wtH^* = \sum_{i \in I}\bbC(Q_i+\alpha_i) + \sum_{i \in I}\bbC Q_i.
\end{split}
\end{align}
Here we consider $H^*$ as a copy of $\frh^*$ regarded as $\frh^* = \clQ \otimes_{\bbZ} \bbC$.
Thus $H^*$ has a basis 
\begin{align}\label{eq:E:H*bas}
 \{Q_i, \alpha_i \mid i \in I\}, \quad Q_i \ceq Q_{\alpha_i}.
\end{align}
We call $Q_i$ \emph{the dual} (\emph{dynamical}) \emph{parameters}.
The choice of $Q_i$ and $\alpha_i$ are as follows: For this basis of $H^*$ and the basis \eqref{eq:E:wtHbas} of $H$, the canonical pairing $\pair{\cdot,\cdot}\colon H^* \times H \to \bbC$ has the values 
\begin{align}\label{eq:E:Hpair}
 \pair{Q_i,P_j} = \ipd{\alpha_i}{\alpha_j} = b_{ij}, \quad \pair{\alpha_i,h_j} = a_{ji}, 
\end{align}
and the other pairings are $0$.

We close this part with 

\begin{rmk}
\begin{enumerate}
\item
We modified the notations in \cite{FKO} in several places. In particular, the linear space $\wtH$ is denoted as $H$ in loc.\ cit., and defined to be an extension of ours by $\bbC c$.

\item 
We use the symbol $P+h$ in two meanings: the one is an element $P+h \in \wtH$ associated to $h=\sum_{i \in I}c_i \alpha^\vee_i \in \frh$, and the other is the set of variables $P+h=\bigl(P_i+\nu(h_i)\bigr)_{i \in I}$.

\item
Let us give an example of the calculation of the pairing \eqref{eq:E:Hpair}:
$\pair{Q_i+\alpha_i,P_j+h_j}=\pair{Q_i,P_j}+\pair{\alpha_i,h_j}=\pair{Q_i,P_j}+\pair{\alpha_i,h_j}=b_{ij}+a_{ji}$.
\end{enumerate}
\end{rmk}

\subsubsection{Definition}\label{sss:E:EA}

Keep the notations in the previous \cref{sss:E:root} and \cref{sss:E:dyn}.
We borrow from \cite[p.106]{FKO} the notations on $q$-numbers. 
Let $q$ be a complex number with $0<\abs{q}<1$, and $i \in I=\{1,2,\dotsc,l\}$.
For $m,n \in \bbN$ with $m \le n$, we define
\begin{gather}
\label{eq:E:qi}
 q_i \ceq q^{d_i}, \quad [n]_i \ceq \frac{q_i^n-q_i^{-n}}{q_i-q_i^{-1}}, \quad 
 [n)_i \ceq \frac{q^n-q^{-n}}{q_i-q_i^{-1}}, \\
\label{eq:E:bini}
 [n]_i! \ceq [n]_i [n-1]_i \dotsm [1]_i, \quad 
 \bnm{m}{n}{i} \ceq \frac{[m]_i!}{[n]_i! \, [m-n]_i!}.
\end{gather}
Let us also fix a formal parameter $p$. Below a \emph{topological algebra over $\bbC\dbr{p}$} means the one with respect to the $p$-adic topology.

The following definition is due to \cite[Appendix A]{JKOS} and \cite[Definition 2.1]{FKO} for general $\wh{\frg}$, and due to \cite{Ko98}, \cite[\S2.2]{Ko09} and \cite[Definition 2.3.1]{Ko} for $\wh{\frg}=\wh{\fsl}_2$.

\begin{dfn}\label{dfn:E:U}
Let $A = A(X_l) = (a_{ij})_{i,j \in I}$ be the Cartan matrix of type $X_l$.
The elliptic algebra 
\[
 U_{q,p}(\wh{\frg}) = U_{q,p}(\wh{\frg}(X_l))
\]
associated to $A$ is a topological algebra over $\bbC\dbr{p}$ generated by 
\begin{align}\label{eq:E:Ugen}
 \clM(\wtH^*), \ q^{\pm c/2}, \ d, \ K^{\pm}_i, \ e_{i,m}, \ f_{i,m}, \ \alpha^\vee_{i,n} 
 \quad (i \in I, \, m \in \bbZ, \, n \in \bbZ \setminus \{0\})
\end{align}
satisfying the following relations.
\begin{itemize}
\item
$\clM(\wtH^*)$ is a subalgebra.

\item 
$q^{\pm c/2}$ are central and $q^{c/2}q^{-c/2}=1$.

\item
$K^{\pm}_i$ are invertible.

\item
It is useful to introduce the generating currents, regarded as formal Laurent series in $z$, by 
\[
 e_i(z) \ceq \sum_{m \in \bbZ}e_{i,m}z^{-m}, \quad 
 f_i(z) \ceq \sum_{m \in \bbZ}f_{i,m}z^{-m}  \quad (i \in I).
\]
Using the functions $F(P),F(P+h) \in \clM(H^*)$ (see \eqref{eq:E:clM}) and the basis elements $Q_i=Q_{\alpha_i}$ of $H^*$ in \eqref{eq:E:H*bas}, the relations containing the elements in $\bbF$ are given as follows.
\begin{align}
\label{eq:E:U:1}
&F(P+h) \, d       = d \, F(P+h), & 
&F(P  ) \, d       = d \, F(P),  
\\
\label{eq:E:U:2}
&F(P+h) \, K_i^\pm = K^\pm_i \, F\bigl(P+h \mp \pair{Q_i,P+h}\bigr), & 
&F(P  ) \, K_i^\pm = K^\pm_i \, F\bigl(P \mp \pair{Q_i,P  }\bigr), 
\\
\label{eq:E:U:3}
&F(P+h) \, e_i(z)  = e_i(z)  \, F(P+h), & 
&F(P  ) \, e_i(z)  = e_i(z)  \, F(P-\pair{Q_i,P}), 
\\
\label{eq:E:U:4}
&F(P+h) \, f_i(z)  = f_i(z)  \, F(P+h-\pair{Q_i, P+h}), & 
&F(P  ) \, f_i(z)  = f_i(z)  \, F(P), 
\\
\label{eq:E:U:5}
&F(P+h) \, \alpha^\vee_{i,n} = \alpha^\vee_{i,n} \, F(P+h). & 
&F(P  ) \, \alpha^\vee_{i,n} = \alpha^\vee_{i,n} \, F(P  ), 
\end{align}
Similarly as the Ding-Iohara algebra in \cref{ss:pre:DI}, the equalities containing the generating currents are understood as those of formal Laurent series in $z$, and the genuine relations are obtained by comparing the coefficients of $z^k$ ($k \in \bbZ$) in both sides.

\item
The remaining relations containing $d$ are given by 
\begin{align}
\label{eq:E:dK}
 [d,K^\pm_i] = 0, \quad 
 [d,e_{i,m}] = m e_{i,m}, \quad 
 [d,f_{i,m}] = m f_{i,m}, \quad
 [d,\alpha^\vee_{j,n}] = n \alpha^\vee_{i,n}.
\end{align}

\item
The remaining relations containing $K^\pm_i$ are given by 
\begin{align}
\label{eq:E:Ke}
 K^\pm_i e_j(z)  = q_i^{\mp a_{ij}} e_j(z) K^\pm_i, \quad 
 K^\pm_i f_j(z)  = q_i^{\pm a_{ij}} f_j(z) K^\pm_i, \quad 
[K^\pm_i, \alpha^\vee_{j,n}]=0.
\end{align}

\item
Using \eqref{eq:E:qi}, the remaining relations containing $\alpha^\vee_{i,n}$ are given by 
\begin{align}
\label{eq:E:2.14}
&[\alpha^\vee_{i,m},\alpha^\vee_{j,n}] = 
 \delta_{m+n,0} \frac{[a_{ij} m]_i [c m)_j}{m} \frac{1-p^m}{1-p^{* m}} q^{-cm}, 
\\
\label{eq:E:2.15}
&[\alpha^\vee_{i,m},e_j(z)] = 
 \frac{[a_{ij}m]_i}{m} \frac{1-p^m}{1-p^{* m}} q^{-cm}z^m e_j(z),
\\
\label{eq:E:2.16}
&[\alpha^\vee_{i,m},f_j(z)] = -\frac{[a_{ij}m]_i}{m} z^m f_j(z).
\end{align}
Here and hereafter, we use 
\begin{align}\label{eq:E:p*}
 p^* \ceq p q^{-2c},
\end{align}
and the term $\frac{1}{1-p^{*m}}$ is understood as the formal series $\sum_{n \ge 0} p^{* m n}$.

\item
For the other relations, it is useful to introduce additional generating currents 
\begin{gather*}
 \psi^+_i(q^{-c/2}z) \ceq K^+_i 
 \exp\Bigl(-(q_i-q_i^{-1})\sum_{n>0} \frac{   \alpha^\vee_{i,-n}}{1-p^n}z^{ n}\Bigr)
 \exp\Bigl( (q_i-q_i^{-1})\sum_{n>0} \frac{p^n\alpha^\vee_{i, n}}{1-p^n}z^{-n}\Bigr), \\
 \psi^-_i(q^{ c/2}z) \ceq K^-_i 
 \exp\Bigl(-(q_i-q_i^{-1})\sum_{n>0} \frac{p^n\alpha^\vee_{i,-n}}{1-p^n}z^{ n}\Bigr)
 \exp\Bigl( (q_i-q_i^{-1})\sum_{n>0} \frac{   \alpha^\vee_{i, n}}{1-p^n}z^{-n}\Bigr),
\end{gather*}
where $\exp(x) \ceq \sum_{n \in \bbN}\frac{1}{n!}x^n$ denotes the formal exponential, and the expression $\frac{1}{1-p^n}$ is understood as $\sum_{k \ge 0} p^{k n}$. One can check that these generating currents have well-defined coefficients in the $p$-adic topology (c.f.\ \cref{rmk:DI:top} \ref{i:rmk:DI:top:p}). Then the remaining relations are given as follows.
\begin{itemize}
\item 
For any $i,j \in I$, the quadratic relations of $e_{i,m}$'s and $f_{i,m}$'s are given by 
\begin{align}
\label{eq:E:ee}
 z \, \gamma(\tfrac{w}{z};q^{ b_{ij}},p^*) e_i(z)e_j(w) &= 
-w \, \gamma(\tfrac{z}{w};q^{ b_{ij}},p^*) e_j(w)e_i(z), \\
\label{eq:E:ff} 
 z \, \gamma(\tfrac{w}{z};q^{-b_{ij}},p  ) f_i(z)f_j(w) &= 
-w \, \gamma(\tfrac{z}{w};q^{-b_{ij}},p  ) f_j(w)f_i(z),
\end{align}
where the structure function $G$ is given by 
\begin{align*}
 \gamma(x;q,p) \ceq \frac{(q x;p)_\infty}{(p q^{-1}x;p)_\infty}.
\end{align*}
See \eqref{eq:ntn:qsf} for the shifted factorial $(x;p)_\infty$. We regard $G(x;q,p) \in \bbC[x^{\pm1}]\dbr{p}$ by expanding the denominator using $\frac{1}{1-x p^k}=\sum_{n \ge 0}(x p^k)^n$. The quantity $b_{ij}$ is the entry of the symmetrization $A=D B$ of the Cartan matrix $A$ with the diagonal matrix $D^{-1}=(d_1,\dotsc,d_l)$ given by \eqref{eq:E:D}. As before, the genuine relations are obtained by expanding both sides into formal Laurent series of $z,w$ and comparing the coefficients of $z^m w^n$ ($m,n \in \bbZ$) in both sides.

\item
For $i,j \in I$ with $i \ne j$, assume $a \ceq 1-a_{ij} \ge 2$. We define 
\begin{align}\label{eq:E:wtgam}
 \wtgm(x;q,p) \ceq \frac{(p q x;p)_\infty}{(p q^{-1} x;p)_\infty},
\end{align}
regarded as a formal series of $p$.
Then, using \eqref{eq:E:bini}, the Serre-type relation of $e_i(z),e_j(z)$ is
\begin{align}\label{eq:E:eS}
\begin{split}
 \sum_{\sigma \in \frS_a} 
&\prod_{1 \le m<n \le a} \wtgm(\tfrac{z_{\sigma(n)}}{z_{\sigma(m)}};q^2,p^*) \\
&\cdot
 \sum_{s=0}^a (-1)^s \bnm{a}{s}{i} 
 \prod_{m=  1}^s \wtgm(\tfrac{w}{z_{\sigma(m)}};q^{b_{ij}},p^*)
 \prod_{m=s+1}^a \wtgm(\tfrac{z_{\sigma(m)}}{w};q^{b_{ij}},p^*) \\
&\cdot
 e_i(z_{\sigma(1)})   \dotsm e_i(z_{\sigma(s)}) e_j(w) 
 e_i(z_{\sigma(s+1)}) \dotsm e_i(z_{\sigma(a)}) = 0,
\end{split}
\end{align}
and that of $f_i(z),f_j(z)$ is
\begin{align}
\label{eq:E:fS}
\begin{split}
 \sum_{\sigma \in \frS_a} 
&\prod_{1 \le m<n \le a} \wtgm(\tfrac{z_{\sigma(n)}}{z_{\sigma(m)}};q^{-2},p) \\
&\cdot
 \sum_{s=0}^a (-1)^s \bnm{a}{s}{i} 
 \prod_{m=  1}^s \wtgm(\tfrac{w}{z_{\sigma(m)}};q^{-b_{ij}},p)
 \prod_{m=s+1}^a \wtgm(\tfrac{z_{\sigma(m)}}{w};q^{-b_{ij}},p) \\
&\cdot
 f_i(z_{\sigma(1)})   \dotsm f_i(z_{\sigma(s)}) f_j(w) 
 f_i(z_{\sigma(s+1)}) \dotsm f_i(z_{\sigma(a)}) = 0.
\end{split}
\end{align}
Here $\frS_a$ denotes the symmetric group of degree $a$.

\item
Finally, the relations of $e_{i,m}$ and $f_{i,m}$ are given by 
\begin{align}
\label{eq:E:ef}
 [e_i(z),f_j(w)] = \frac{\delta_{i,j}}{q-q^{-1}}
 \bigl(\delta(q^{-c}\tfrac{z}{w}) \psi^-_i(q^{ c/2}w)
      -\delta(q^{ c}\tfrac{z}{w}) \psi^+_i(q^{-c/2}z)\bigr),
\end{align}
where $\delta$ denotes the formal delta function 
$\delta(x) \ceq \sum_{n \in \bbZ} x^n$ in \eqref{eq:DI:delta}.
\end{itemize}
\end{itemize}
\end{dfn}

\begin{rmk}\label{rmk:E:U}
Some comments on \cref{dfn:E:U} are in order.
\begin{enumerate}
\item
If we take $p$ to be a complex number such that $\abs{p}<1$, then $U_{p,q}(\wh{g})$ is a genuine algebra over $\bbC$.

\item
The definition in \cite[Definition 2.1]{FKO} and \cite[Definition 2.3.1]{Ko} has a typo. There $U_{q,p}(\wh{\frg})$ is defined to be a topological algebra over $\bbF\dbr{p}$ with $\bbF \ceq \clM(P(\frh)^*)$, but then the functions $F(P+h),F(P) \in \bbF$ would commute with the other generators. A correct definition of the $\wh{\fsl}_2$ case is given in \cite[\S2.2]{Ko09}.

\item 
The relations $[d,K^\pm_i]=0$ are missed in \cite[Definition 2.1]{FKO}.

\item
The relations \eqref{eq:E:ee} and \eqref{eq:E:ff} can be rewritten as 
\begin{align}
\label{eq:E:ee''}
 z \, \theta^*(q^{ b_{ij}}\tfrac{w}{z}) e_i(z) e_j(w) &= 
-w \, \theta^*(q^{ b_{ij}}\tfrac{z}{w}) e_j(w) e_i(z), \\
\label{eq:E:ff''}
 z \, \theta  (q^{-b_{ij}}\tfrac{w}{z}) f_i(z) f_j(w) &= 
-w \, \theta  (q^{-b_{ij}}\tfrac{z}{w}) f_j(w) f_i(z).
\end{align}
where we used the theta series 
\begin{align}\label{eq:E:theta}
 \theta(x) \ceq \theta(x;p), \quad \theta^*(x) \ceq \theta(x;p^*), \quad 
 \theta(x;p) \ceq (x,p x^{-1};p)_{\infty} \in \bbC[x^{\pm1}]\dbr{p}.
\end{align}
See \eqref{eq:ntn:qsf} for the symbol of infinite product.
These relations are given in \cite[(A.17), (A.18)]{JKOS}.

\item \label{i:rmk:E:U:d'}
The definition of $U=U_{q,p}$ is independent of the choice of symmetrization $A=D B$, $D^{-1}=(d_1,\dotsc,d_l)$, of the Cartan matrix $A$ up to isomorphism. Indeed, if we take another symmetrization $A=D' B'$ with $D'^{-1}=(d'_1,\dotsc,d'_l)$, then by \cite[Remark 2.4]{Ka}, there is an $r \in \bbQ \setminus \{0\}$ such that $d'_i = r d_i$ for any $i \in I$. Thus, replacing 
\[
 q \lmto q' \ceq q^{1/r}, \quad c \lmto c' \ceq r c,
\]
we have $q_i=q^{d_i}=(q')^{d'_i}=q'_i$, $q^{-2c}=(q')^{-2c'}$ and $q^{b_{ij}}=q^{d_i a_{ij}}=(q')^{d'_i a_{ij}}=(q')^{b'_{ij}}$, so that the defining relations \eqref{eq:E:Ke}--\eqref{eq:E:2.16}, \eqref{eq:E:ee}--\eqref{eq:E:ef} are preserved. Thus, we have an isomorphism $U_{q,p} \cong U_{q',p}$ of topological algebras.
\end{enumerate}
\end{rmk}

For the later purpose, let us rewrite some of the defining relations.

\begin{lem}\label{lem:E:pef}
The relations \eqref{eq:E:2.15} and \eqref{eq:E:2.16} are equivalent (under \eqref{eq:E:Ke}) to the following current relations.
\begin{align}
\label{eq:E:ppe'}
 \psi_i^+(z)e_j(w) &= g_{ij}^*(q^{-c/2}\tfrac{z}{w})e_j(w)\psi_i^+(z), &
 \psi_i^+(z)f_j(w) &= g_{ij}(q^{c/2}\tfrac{z}{w})^{-1}f_j(w)\psi_i^+(z), 
\\
\label{eq:E:pme'}
 \psi_i^-(z)e_j(w) &= g_{ij}^*(q^{c/2}\tfrac{z}{w})e_j(w)\psi_i^-(z), &
 \psi_i^-(z)f_j(w) &= g_{ij}(q^{-c/2}\tfrac{z}{w})^{-1}f_j(w)\psi_i^-(z).
\end{align}
Here the formal series $g_{ij}(x)=g_{ij}(x;p)$, $g_{ij}^*(x)=g_{ij}(x;p^*)$ are the expansion of 
\begin{align}\label{eq:E:gij}
 g_{ij}(x;p) = \frac{G^+_{ij}(x;p)}{G^-_{ij}(x;p)} \ceq 
 \frac{q^{-b_{ij}} \theta(q^{b_{ij}}x;p)}{\theta(q^{-b_{ij}}x;p)}
\end{align}
with respect to $p$, where $\theta (x;p)$ is given in \eqref{eq:E:theta}.
\end{lem}

\begin{proof}
We only derive $\psi_i^+(z)e_j(w) = g_{ij}^*(q^{-c/2}\tfrac{z}{w})e_j(w)\psi_i^+(z)$ from \eqref{eq:E:2.15} and \eqref{eq:E:Ke}. Using the relation $\exp(A) \exp(B) = \exp([A,B])\exp(B) \exp(A)$ in case $[A,B]$ commutes with $A$ and $B$, we have $\psi^+_i(z) e_j(w) = C(\frac{z}{w}) e_j(w) \psi^+_i(z)$ with 
\begin{align*}
 C(x) \ceq q_i^{-a_{ij}} 
 \exp\Bigl[-\sum_{n >0} \frac{1}{n} \frac{q_i^{a_{ij} n}-q_i^{-a_{ij}n}}{1-p^{*n}} 
            (q^{-c/2} x)^n\Bigr]
 \exp\Bigl[ \sum_{n >0} \frac{1}{n} \frac{q_i^{a_{ij} n}-q_i^{-a_{ij}n}}{1-p^{*n}}
            (p^* q^{ c/2}x^{-1})^n\Bigr].
\end{align*}
Then, using $\exp(-\sum_{n>0}\frac{1}{n}x^n)=1-x$ and $q_i^{a_{ij}}=q^{d_i a_{ij}}=q^{b_{ij}}$, we have
\begin{align*}
 C(x) = q^{-b_{ij}} 
 \frac{(    q^{ b_{ij}}q^{-c/2}x;p^*)_\infty}
      {(    q^{-b_{ij}}q^{-c/2}x;p^*)_\infty}
 \frac{(p^* q^{-b_{ij}}q^{ c/2}x^{-1};p^*)_\infty}
      {(p^* q^{ b_{ij}}q^{ c/2}x^{-1};p^*)_\infty} 
    = q^{-b_{ij}} \frac{\theta(q^{ b_{ij}}q^{-c/2}x;p^*)}{\theta(q^{-b_{ij}}q^{-c/2}x;p^*)}
    = g^*_{ij}(q^{-c/2}x).
\end{align*}
Thus we have the desired equality.
\end{proof}

\begin{lem}\label{lem:E:pp}
The relations \eqref{eq:E:2.14} are equivalent (under \eqref{eq:E:Ke}) to the current relations
\begin{align}\label{eq:E:pp'}
 \psi_i^{\pm}(z)\psi_j^{\pm}(w) = 
 \frac{g^*_{ij}(\frac{z}{w})}{g_{ij}(\frac{z}{w})}\psi_j^{\pm}(w)\psi_i^{\pm}(z), \quad 
 \psi_i^{+}(z)\psi_j^{-}(w) = 
 \frac{g^*_{ij}(q^{-c}\frac{z}{w})}{g_{ij}(q^c\frac{z}{w})}\psi_j^{-}(w)\psi_i^{+}(z)
\end{align}
with $g_{ij}(x)=g_{ij}(x;p)$ and $g_{ij}^*(x)=g_{ij}(x;p^*)$ given in \eqref{eq:E:gij}.
\end{lem}

\begin{proof}
Let us derive $\psi_i^{\pm}(z)\psi_j^{\pm}(w) = \frac{g^*_{ij}(\frac{z}{w})}{g_{ij}(\frac{z}{w})}\psi_j^{\pm}(w)\psi_i^{\pm}(z)$. The same argument as in \cref{lem:E:pef} yields 
\[
 \psi_i^{\pm}(z) \psi_j^{\pm}(w) = 
 C_{ji}(\tfrac{z}{w}) C_{ij}(\tfrac{w}{z})^{-1} \psi_j^{\pm}(w) \psi_i^{\pm}(z),
\]
where, using $(x;p,p^*)_\infty \ceq \prod_{n \in \bbN}(x p^{*n};p)_\infty = \prod_{m,n\in\bbN}(1-xp^m p^{*n})$, we set 
\[
 C_{ji}(x) \ceq \exp\Bigl[\sum_{n>0} \frac{1}{n} 
  \frac{q_j^{a_{ji}n}-q_j^{-a_{ji}n}}{1-p^n} \frac{1-q^{-2cn}}{1-p^{*n}} (p x)^n\Bigr]. 
\]
By $\exp(-\sum_{n>0}\frac{1}{n}x^n)=1-x$ and $q_j^{a_{ji}}=q^{d_j a_{ji}}=q^{b_{ji}}$, we have
\[
 C_{ji}(x) = \frac{(q_j^{-a_{ji}}x p  ,q_i^{a_{ji}}x p^*;p,p^*)_\infty}
                  {(q_j^{-a_{ji}}x p^*,q_j^{a_{ji}}x p  ;p,p^*)_\infty}
           = \frac{(q^{-b_{ji}}x p  ,q^{b_{ji}}x p^*;p,p^*)_\infty}
                  {(q^{-b_{ji}}x p^*,q^{b_{ji}}x p  ;p,p^*)_\infty}.
\]
Then, using 
\[
 \frac{(x p;p,p^*)_\infty}{(x q;p,p^*)_\infty} = 
 \frac{\prod_{m,n \in \bbN}(1-x p^{m+1}p^{*n})}{\prod_{m,n \in \bbN}(1-x p^m (p^*)^{n+1})} = 
 \frac{\prod_{m \in \bbN}(1-x p^{m+1})}{\prod_{n \in \bbN}(1-x (p^*)^{n+1})} = 
 \frac{(x p;p)_\infty}{(x p^*;p^*)_\infty}
\]
and $b_{ij}=b_{ji}$, we have
\begin{align*}
  C_{ji}(x) C_{ij}(x^{-1})^{-1} 
&=\frac{(q^{-b_{ji}}x     p  ;p  )_\infty}{(q^{-b_{ji}}x p^*,p^*)_\infty}
  \frac{(q^{ b_{ji}}x     p^*;p^*)_\infty}{(q^{ b_{ji}}x p  ;p  )_\infty} \cdot 
  \frac{(q^{ b_{ij}}x^{-1}p^*;p^*)_\infty}{(q^{ b_{ij}}x^{-1}p  ;p  )_\infty}
  \frac{(q^{ b_{ij}}x^{-1}p  ;p  )_\infty}{(q^{ b_{ij}}x^{-1}p^*;p^*)_\infty} \\
&=\frac{(q^{-b_{ij}}x,p q^{b_{ij}}/x;p)_\infty}{(q^{b_{ij}}x,p q^{-b_{ij}}/x;p)_\infty}
  \frac{1-q^{b_{ij}}x}{1-q^{-b_{ij}}x} \cdot 
  \frac{(q^{b_{ij}}x,p^* q^{-b_{ij}}/x;p^*)_\infty}{(q^{-b_{ij}}x,p^* q^{b_{ij}}/x;p^*)_\infty}
  \frac{1-q^{-b_{ij}}x}{1-q^{b_{ij}}x} \\
&=\frac{\theta(q^{ b_{ij}}x;p^*)}{\theta(q^{-b_{ij}}x;p^*)} 
  \frac{\theta(q^{-b_{ij}}x;p  )}{\theta(q^{ b_{ij}}x;p)}
 =\frac{g_{ij}^*(x)}{g_{ij}(x)}.
\end{align*}
\end{proof}

\begin{lem}\label{lem:E:ee}
The relations \eqref{eq:E:ee} and \eqref{eq:E:ff} can be rewritten as 
\begin{align}
\label{eq:E:ee'}
 G_{ij}^{-*}(\tfrac{z}{w}) e_i(z) e_j(w)  = 
 G_{ij}^{+*}(\tfrac{z}{w}) e_j(w) e_i(z), \quad
 G_{ij}^{+ }(\tfrac{z}{w}) f_i(z) f_j(w) = 
 G_{ij}^{- }(\tfrac{z}{w}) f_j(w) f_i(z),
\end{align}
respectively, where we used $G_{ij}^{\pm}(x) \ceq G_{ij}^{\pm}(x;p)$ and $G_{ij}^{\pm *}(x) \ceq G_{ij}^{\pm}(x;p^*)$ with $G_{ij}^{+}(x;p) \ceq q^{-b_{ij}} \theta(q^{b_{ij}}x;p)$, $G_{ij}^-(x;p) \ceq \theta(q^{-b_{ij}}x;p)$ as introduced in \eqref{eq:E:gij}.
\end{lem}

\begin{proof}
Recall the rewritten form \eqref{eq:E:ee''} of \eqref{eq:E:ee}. 
Then, \eqref{eq:E:ee''} can be rewritten as 
\[
 z \, G_{ij}^{+*}(\tfrac{w}{z}) e_i(z) e_j(w) = -w \, G_{ij}^{+*}(\tfrac{z}{w}) e_j(w) e_i(z),
\]
which is equivalent to the first half of \eqref{eq:E:ee'} by the equality $x^{-1} G_{ij}^{+*}(x^{-1}) = G_{ij}^{-*}(x)$. We can similarly show the equivalence between \eqref{eq:E:ff} and the second half of \eqref{eq:E:ee'}.
\end{proof}

\begin{rmk}
Note the similarity between the fundamental relations of the Ding-Iohara quantum algebra $U_q(g,A_l)$ and those of the elliptic algebra $U_{q,p}(\wh{\frg})$ of type $A_l$: \eqref{eq:DI:ppe} and \eqref{eq:E:ppe'}, \eqref{eq:DI:pme} and \eqref{eq:E:pme'}, \eqref{eq:DI:ee} and \eqref{eq:E:ee'}.
\end{rmk}

\subsubsection{Hopf algebroid structure of $U_{q,p}(\wh{\frg})$}

Now we explain a Hopf algebroid structure on the elliptic algebra $U_{q,p}(\wh{\frg})$ of \cref{dfn:E:U}. 

Let us first remark that the definitions on Hopf algebroids in \cref{ss:pre:Halgd} are given for algebras over $\bbC$, and since $U_{q,p}(\wh{\frg})$ is not an algebra over $\bbC$ but a topological algebra over $\bbC\dbr{p}$, we should modify the definitions in \cref{ss:pre:Halgd} appropriately. The obtained notions will be called topological $H$-algebras, topological $H$-bialgebroids, and topological $H$-Hopf algebroids. We leave it for the reader to write down the precise definitions.

Let us keep the notation in \S\S \ref{sss:E:root}--\ref{sss:E:EA}, and fix the Cartan matrix $A=(a_{ij})_{i,j \in I}$ of finite type $X_l$. For simplicity, we often denote the elliptic algebra associated to $A$ by 
\[
 U \ceq U_{q,p}(\wh{\frg}).
\]

Recall the linear space $H$ in \eqref{eq:E:H} and the basis $\{P_i \mid i \in I\}$ in \eqref{eq:E:wtHbas}. We denote the field of meromorphic functions on the dual space $H^*$ by 
\[
 \bbF = \clM(H^*).
\]
Let us consider the extension $\bbF\dbr{p} \ceq \bbF \otimes_{\bbC}\dbr{p}$ since we will work over $\bbC\dbr{p}$. Following the notation \eqref{eq:E:finF}, we express an element $F \in \bbF\dbr{p}$ as 
\begin{align}\label{eq:E:fPp}
 F = F(P,p) = \sum_{n \in \bbN}p^n F_n(P), 
\end{align}
where $F_n(P)$ is an element of $\bbF$ for each $n \in \bbN$.

Let us also recall from \eqref{eq:E:H} that the linear space $\wtH$ spanned by $P_i+h_i$, $P_i$ ($i \in I$). By \cref{dfn:E:U}, the extension $\clM(H^*)\dbr{p}$ of the meromorphic function field on $H^*$ is a subalgebra of $U$ over $\bbC\dbr{p}$.

Now, let us regard $H$ as a finite-dimensional commutative Lie algebra, so that we have the notion of $H$-algebra (see \cref{ss:pre:Halgd}). In order to define the $H^*$-bigrading on $U$, we consider the actions of $q^{P+h},q^P \in \clM(\wtH^*)$ on $U$, where we regard $q^{P+h} = \exp((P+h) \log q)$, $q^P = \exp(P \log q)$ and apply \eqref{eq:E:U:1}--\eqref{eq:E:U:5}.

\begin{lem}\label{lem:E:Halg}
The elliptic algebra $U$ has the following topological $H$-algebra structure. 
\begin{clist}
\item 
The $H^*$-bigrading $U = \bigoplus_{\alpha,\beta \in H^*} U_{\alpha,\beta}$ is given by
\begin{align*}
 U_{\alpha,\beta} \ceq \{a \in U \mid q^{P} a q^{-P} = q^{\pair{\alpha,P}}a, \ 
 q^{P+h} a q^{-(P+h)} = q^{\pair{\beta,P+h}}a \quad (\forall P,P+h \in \wtH) \},
\end{align*}
Here $\pair{\cdot,\cdot}$ denotes the the canonical pairing on $\wtH^* \times \wtH$ (see \eqref{eq:E:Hpair}) restricted to $H^* \times \wtH$ with respect to the embedding $H =\sum_i \bbC Q_i \inj \wtH=\sum_i \bbC(Q_i+\alpha_i)+\sum_i \bbC Q_i$, $Q_i \mto Q_i$.

\item
The moment maps $\mu_l,\mu_r\colon \bbF\dbr{p} \to U_{0,0}$ are defined by
\begin{align}\label{eq:E:mu}
 \mu_l(F) \ceq F(P,p^*), \quad \mu_r(F) \ceq F(P+h,p)
\end{align}
for $F = F(P,p) \in \bbF\dbr{p}$.  
\end{clist}
Moreover, the following statements hold.
\begin{enumerate}
\item \label{i:lem:E:gen}
The images $\mu_l(\bbF)$ and $\mu_r(\bbF)$ generate $\clM(\wtH^*)$.

\item \label{i:lem:E:Halg:Q}
The generators have the following bigrading: 
\begin{align*}
 e_i(z) \in U_{-Q_i,0}, \quad f_i(z) \in U_{0,-Q_i}, \quad
 \psi^\pm_i(z) \in U_{-Q_i,-Q_i}, \quad  K^\pm_i \in U_{-Q_i,-Q_i},
\end{align*}
and the others belong to $U_{0,0}$, where, for a current $x(z)=\sum_{m \in \bbZ} x_m z^{-m}$, the expression $x(z) \in U'_{\alpha,\beta}$ means $x_m \in U_{\alpha,\beta}$ for each $m \in \bbZ$. In particular, the bigrading values in the lattice $\clQ$:  
\[
 U = \bigoplus_{\alpha,\beta \in \clQ} U_{\alpha,\beta}, \quad 
 \clQ \ceq \bigoplus_{i \in I} \bbZ Q_i \subsetneq H^*.
\]
\end{enumerate}
\end{lem}

\begin{proof}
This is a special case of \cref{lem:U:Halg}. 
\end{proof}

Note that $\clQ$ is isomorphic to the root lattice of the Lie algebra $\frg$. We also have $\clM(\wtH^*)\dbr{p} \subset U_{0,0}$, so that the values \eqref{eq:E:mu} of the moment maps do belong to $U_{0,0}$.

\begin{prp}\label{prp:E:Halgd}
The topological $H$-algebra $U$ in \cref{lem:E:Halg} has the following topological $H$-Hopf algebroid structure $(\Delta,\ve,S)$.
\begin{enumerate}
\item \label{i:E:Halgd:Delta}
The comultiplication $\Delta\colon U \to U \totimes U$ is given by the Drinfeld-type formulas.
For $F=F(P,p) \in \bbF\dbr{p}$, we have
\begin{align*}
 \Delta(\mu_l(F)) \ceq \mu_l(F) \totimes 1, \quad 
 \Delta(\mu_r(F)) \ceq 1 \totimes \mu_r(F).
\end{align*}
By \cref{lem:E:Halg} \ref{i:lem:E:gen}, these define $\Delta$ on the whose $\clM(\wtH^*)$.
The formulas for the generators $q^{\pm c/2}$, $d$ and $K^\pm_i$ are 
\[
 \Delta(q^{\pm c/2}) \ceq q^{\pm c/2} \totimes q^{\pm c/2}, \quad 
 \Delta(d)       \ceq d \totimes 1 + 1 \totimes d, \quad 
 \Delta(K^\pm_i) \ceq K^\pm_i \totimes K^\pm_i. 
\]
The remaining formulas are given in terms of the generating currents:
\begin{align} 
\label{eq:E:De}
&\Delta(e_i(z)) \ceq e_i(z) \totimes 1 + \psi^+_i(q^{c_1/2}z) \totimes e_i(q^{c_1}z), \\
\label{eq:E:Df}
&\Delta(f_i(z)) \ceq 1 \totimes f_i(z) + f_i(q^{c_2}z) \totimes \psi^-_i(q^{c_2/2}z), \\
\label{eq:E:Dp+}
&\Delta(\psi_i^+(z)) \ceq \psi_i^+(q^{ c_2/2}z) \totimes \psi_i^+(q^{-c_1/2}z), \\
\label{eq:E:Dp-}
&\Delta(\psi_i^-(z)) \ceq \psi_i^-(q^{-c_2/2}z) \totimes \psi_i^-(q^{ c_1/2}z).
\end{align}
For the symbols $q^{\pm c_1/2}$ and $q^{\pm c_2/2}$, see the comment after \eqref{eq:DI:Delta2}.

\item \label{i:E:Halgd:ve}
The counit $\ve\colon U \to D_H\dbr{p}$ is given by 
\begin{gather*}
 \ve(\mu_l(F)) = \ve(\mu_r(F)) \ceq F T_0, \quad 
 \ve(q^{\pm c/2}) \ceq 1, \quad \ve(d) \ceq 0, \quad \ve(K^\pm_i) \ceq T_{Q_i}, \\
 \ve(e_i(z))  = \ve(f_i(z)) \ceq 0, \quad \ve(\psi_i^{\pm}(z)) \ceq T_{Q_i}.
\end{gather*}
Here $D_H$ denotes the $H$-algebra of difference operators on $H^*$ (see the paragraph of \eqref{eq:Halgd:Dh}), and $D_H\dbr{p} \ceq D_H \otimes_{\bbC} \bbC\dbr{p}$ is its $\bbC\dbr{p}$-extension with the naturally induced topological $H$-algebra structure.

\item \label{i:E:Halgd:S}
The antipode $S\colon U \to U$ is given by 
\begin{gather*}
 S(\mu_l(F)) = \mu_r(F), \quad S(\mu_r(F)) \ceq \mu_l(F), \quad
 S(d) \ceq d, \quad S(q^{\pm c/2}) \ceq q^{\mp c/2}, \\
 S(e_i(z)) \ceq -\psi^+_i(q^{-c/2}z)^{-1}e_i(q^{-c}z), \quad
 S(f_i(z)) \ceq -f_i(q^{-c}z)\psi^-_i(q^{-c/2}z), \\
 S(K^\pm_i) \ceq (K^\pm_i)^{-1}, \quad 
 S(\psi_i^{\pm}(z)) \ceq \psi_i^{\pm}(z)^{-1}.
\end{gather*}
\end{enumerate}
\end{prp}

\begin{proof}
The simply-laced types will be proved in \cref{s:ADE}, \cref{prp:U:ADE} and \cref{cor:U:ADE}. 
The non-simply types will be treated in \cref{s:ns}. The $BCF$ types will be proved in \cref{prp:U:BCF}, and the type $G_2$ will be proved in \cref{prp:U:G}.
\end{proof}

The operations in \cref{prp:E:Halgd} are collected and inspected from those in \cite[Appendix B]{JKOS}, and the topological $H$-algebra structure in \cref{lem:E:Halg} is designed to fit with these operations.

\subsubsection{Opposite Hopf algebroid structure}

In this part, we will explain another $H$-algebra structure on the elliptic algebra $U_{q,p}(\wh{\frg})$ appearing in the literature \cite[\S3]{Ko09}, \cite[Proposition 2.4]{FKO}\footnote{There are some typos in \cite{FKO}, Proposition 2.4 and after, on the choice of the finite-dimensional commutative Lie algebra $H$. The right choice should be our $H$, denoted by $P_{\ol{\frh}}$ in loc.\ cit.} and \cite[Proposition 3.1.1]{Ko}. 

We start with a general argument on $\frh$-algebras. As in \cref{ss:pre:Halgd}, let $\frh$ be a finite-dimensional commutative Lie algebra over $\bbC$, $\bbF \ceq \clM(\frh^*)$, and $A=\bigl(A,\bigoplus_{\alpha,\beta \in \frh^*}A_{\alpha,\beta},\mu_l,\mu_r\bigr)$ be an $\frh$-algebra with moment maps $\mu_l,\mu_r\colon \bbF \to A_{0,0}$. Then, exchanging the bigrading and the moment maps, we have another $\frh$-algebra $A^\dagger=\bigl(A^\dagger,\bigoplus_{\alpha,\beta \in \frh^*}A^\dagger_{\alpha,\beta},\mu^\dagger_l,\mu^\dagger_r\bigr)$ with
\begin{align}\label{eq:E:Adag}
 A^\dagger \ceq A, \quad A^\dagger_{\alpha,\beta} \ceq A_{\beta,\alpha}, \quad 
 \mu^\dagger_l \ceq \mu_r, \quad \mu^\dagger_r \ceq \mu_l.
\end{align}
We now want to introduce an $\frh$-Hopf algebroid structure on the $\frh$-algebra $A^\dagger$. 
Recall the modified tensor product $\totimes$ of $\frh$-algebras \eqref{eq:Halgd:mtp}.

\begin{lem}\label{lem:E:wtau}
We have an $\frh$-algebra isomorphism 
\begin{align}
 \wt{\tau}\colon A \totimes A \lto (A^\dagger \totimes A^\dagger)^\dagger, \quad 
 a \totimes b \lmto b \totimes a.
\end{align}
\end{lem}

\begin{proof}
We prove that $\wt{\tau}$ is well-defined as an $\frh$-algebra homomorphism. Note that, by definition, $(A^\dagger \totimes A^\dagger)^\dagger=A^\dagger \totimes A^\dagger$ as an algebra, and the bigrading and the moment maps are given by
\begin{align*}
&(A^\dagger \totimes A^\dagger)^\dagger_{\alpha,\beta} \ceq (A^\dagger \totimes A^\dagger)_{\beta,\alpha} 
 = \bigoplus_{\gamma \in \frh^*} A^\dagger_{\beta,\gamma} \otimes_{\bbF} A^\dagger_{\gamma,\alpha}
 = \bigoplus_{\gamma \in \frh^*} A_{\gamma,\beta} \otimes_{\bbF} A_{\alpha,\gamma}, \\
&\mu_l^{\totimes^\dagger} \ceq 1 \totimes \mu^\dagger_r = 1 \totimes \mu_l, \quad
 \mu_r^{\totimes^\dagger} \ceq \mu^\dagger_l \totimes 1 = \mu_r \totimes 1.
\end{align*} 
Now, let $a,a' \in A_{\alpha,\gamma}$ and $b,b' \in A_{\gamma,\beta}$, and assume $a \totimes b = a' \totimes b'$ in $A \totimes A$. Then, by the definition \eqref{eq:Halgd:totimes} of $\totimes$, we have
\begin{align*}
 a \otimes b - a' \otimes b' 
=\sum_i(\mu_r(f_i)a_i \otimes b_i-a_i \otimes \mu_l(f_i)b_i)
\end{align*}
for some $f_i \in \bbF$ and $a_i,b_i \in A$, where $\otimes$ denotes the ordinary tensor product over $\bbC$ of $A_{\alpha, \gamma}$ and $A_{\gamma, \beta}$. 
The definition of the moment maps of $A'$ implies
\begin{align*}
  b \otimes a -b' \otimes a' 
= \sum_i(b_i \otimes \mu_r (f_i)a_i -\mu_l (f_i)b_i \otimes a_i) 
= \sum_i(b_i \otimes \mu'_l(f_i)a_i -\mu'_r(f_i)b_i \otimes a_i),
\end{align*}
which yields $\wt{\tau}(a \totimes b) = b \totimes a = b' \totimes a' = \wt{\tau}(a' \totimes b')$ in $(A^\dagger \totimes A^\dagger)^\dagger$. Thus $\wt{\tau}$ is well-defined as an algebra homomorphism. 

Next, if $a \in A_{\alpha,\gamma}$ and $b \in A_{\gamma,\beta}$, then $a \totimes b \in (A \totimes A)_{\alpha,\beta}$ and $b \totimes a \in (A^\dagger \totimes A^\dagger)^\dagger_{\alpha,\beta}$. Also, for $f \in \bbF$, we have 
\begin{align*}
&\wt{\tau}(\mu^{\totimes}_l(f)) = \wt{\tau}(\mu_l(f) \totimes 1) =
 1 \totimes \mu_l(f) = \mu^{\totimes'}_l(f), \\
&\wt{\tau}(\mu^{\totimes}_r(f)) = \wt{\tau}(1 \totimes \mu_r(f)) =
 \mu_r(f) \totimes 1 = \mu^{\totimes'}_r(f).
\end{align*} 
Thus $\wt{\tau}$ is an $\frh$-algebra homomorphism. The bijectivity is clear. 
\end{proof}

Recall (see \cite[Corollary III.3.5]{Kas} for example) that, for a Hopf algebra $H=(H,m,\Delta,S)$ with multiplication $m$, comultiplication $\Delta$ and antipode $S$, we can associate the following new Hopf algebras.
\[
 H^{\op}      \ceq (H,m^{\op},\Delta,      S^{-1}), \quad 
 H^{\cop}     \ceq (H,m,      \Delta^{\op},S^{-1}), \quad 
 H^{\op,\cop} \ceq (H,m^{\op},\Delta^{\op},S),
\]
where the opposite operations are defined as $m^{\op} \ceq m \circ \tau$ and $\Delta^{\op} \ceq \tau \circ \Delta$, using the flip $\tau\colon H \otimes H \to H \otimes H$, $a \otimes b \mto b \otimes a$. Moreover, the antipode $S$ induces Hopf algebra isomorphisms $H \sto H^{\op,\cop}$ and $H^{\op} \sto H^{\cop}$.

Using the $\frh$-algebra isomorphism $\wt{\tau}$ in \cref{lem:E:wtau}, we have a natural analogue of the above opposite theory for $\frh$-Hopf algebroids. Besides $A^\dagger$ in \eqref{eq:E:Adag}, we introduce another $\frh$-algebra $\dot{A}=(A,\bigoplus_{\alpha,\beta}\dot{A}_{\alpha,\beta},\mu_l,\mu_r)$ with 
\[
 \dot{A}_{\alpha,\beta} \ceq A_{-\alpha,-\beta}.
\]

\begin{prp}\label{prp:E:Aopcop}
Let $A$ be an $\frh$-algebra, and $(m,\Delta,S)$ be an $\frh$-Hopf algebroid structure on $A$ with multiplication $m\colon A \otimes A \to A$, comultiplication $\Delta\colon A \to A \totimes A$ and antipode $S\colon A \to A$. Then we have the following new $\frh$-Hopf algebroids.
\begin{align}\label{eq:E:Hop}
 A^{\op}      \ceq (\dot{A},  m^{\op},\Delta,      S^{-1}), \quad 
 A^{\cop}     \ceq (A^\dagger,m,      \Delta^{\op},S^{-1}), \quad 
 A^{\op,\cop} \ceq (\dot{A}^\dagger,m^{\op},\Delta^{\op},S).
\end{align}
Here we defined the opposite operations as $m^{\op} \ceq m \circ \tau$ and $\Delta^{\op} \ceq \wt{\tau} \circ \Delta$, using the flip $\tau\colon A \otimes A \to A \otimes A$, $a \otimes b \mto b \otimes a$ and the map $\wt{\tau}\colon A \totimes A \to (A^\dagger \totimes A^\dagger)^\dagger$ in \cref{lem:E:wtau}. Moreover, the antipode $S$ induces $\frh$-Hopf algebroid isomorphisms $A \sto A^{\op,\cop}$ and $A^{\op} \sto A^{\cop}$.
\end{prp}

The proof is straightforward, and we omit it. Note that in the case $\frh=0$, we have $\totimes=\otimes$ and $\wt{\tau}=\tau$, and recover the opposite Hopf algebras in \eqref{eq:E:Hop}.

Now we go back to the elliptic algebra $U=U_{q,p}(\wh{\frg})$. We equip it with the $H$-algebra structure in \cref{lem:E:Halg}. Then the $H$-algebra $U^\dagger$ in \eqref{eq:E:Adag} has the $H^*$-bigrading $U^\dagger = \bigoplus_{\alpha,\beta \in H^*} U^\dagger_{\alpha,\beta}$ with
\[
 U^\dagger_{\alpha,\beta} \ceq \{a \in U \mid 
  q^{P+h} a q^{-(P+h)} = q^{\pair{\alpha,P+h}}a, \, 
  q^P a q^{-P} = q^{\pair{\beta,P}} a \ (\forall \, P+h,P \in \wtH)\},
\] 
and the moment maps are $\mu^\dagger_l(F) \ceq F(P+h,p^*)$ and $\mu^\dagger_r(F) \ceq F(P,p)$ for $F=F(P,p) \in \bbF\dbr{p}$. 

\begin{rmk}
The $H$-algebra structure on $U^\dagger$ is used in \cite[\S3.1]{Ko09} and \cite[\S3.2]{Ko} to study the comultiplication given in terms of the $L$-operators. Also, it is used in \cite{FKO} to analyze the quantum $Z$-algebra. The tensor product $(U^\dagger \totimes U^\dagger)^\dagger$ is nothing but the opposite tensor product of $U$ used in \cite{Ko09,Ko}. Let us denote 
\[
 U \totimes^{\op} U \ceq U^\dagger \totimes U^\dagger.
\]
Then, the opposite comultiplication $\Delta^{\op}$ in \cref{prp:E:Aopcop} can be written as
\begin{align*}
&\Delta^{\op}(e_i(z)) = 1 \totimes^{\op} e_i(z) + e_i(q^{c_1}z) \totimes^{\op} \psi^+_i(q^{c_1/2}z), \\
&\Delta^{\op}(f_i(z)) = f_i(z) \totimes^{\op} 1 + \psi^-_i(q^{c_2/2}z) \totimes^{\op} f_i(q^{c_2}z), \\
&\Delta^{\op}(\psi_i^+(z)) = \psi_i^+(q^{-c_1/2}z) \totimes^{\op} \psi_i^+(q^{ c_2/2}z), \\
&\Delta^{\op}(\psi_i^-(z)) = \psi_i^-(q^{ c_1/2}z) \totimes^{\op} \psi_i^-(q^{-c_2/2}z).
\end{align*}
These formulas coincide with the Drinfeld-type comultiplication in \cite[Appendix B]{JKOS} up to $q$-factors. See also \cite{Fa,KO} where similar comultiplication appears.
\end{rmk}

\section{Dynamical Ding-Iohara algebroids of simply-laced root systems}\label{s:ADE}

In this section, we introduce an family of Hopf algebroids which includes the elliptic algebras $U_{q,p}(\wh{g})$ (see \cref{ss:pre:E}) of simply-laced types, and whose Hopf algebra degeneration is the Ding-Iohara quantum algebra (see \cref{ss:pre:DI}). As in the previous \cref{ss:pre:E}, we use the word ``a topological algebra'' to mean the algebra over $\bbC\dbr{p}$ with $p$-adic topology. 

\subsection{Algebra structure}\label{ss:ADE:alg}

We use the notation of finite root systems in \cref{sss:E:root}. Thus, let 
\[
 A = A(X_l) = (a_{ij})_{i,j \in I}, \quad I=\{1,\dotsc,l\}
\]
be the Cartan matrix of finite type $X_l$ in the sense of Kac \cite[Theorem 4.3, \S 4.8, TABLE Fin]{Ka}. At this moment, we do not restrict $X_l$ to be simply-laced. We choose the symmetrization \eqref{eq:E:BDA}: 
\[
 A = D B, \quad B = (b_{ij})_{i,j \in I}, \quad D^{-1} = \diag(d_i)_{i \in I}
\]
with $(d_1,\dotsc,d_l)$ in \eqref{eq:E:D}, and fix a realization $(\frh,\Pi,\Pi^\vee)$ of $A$, consisting of the Cartan subalgebra $\frh \subset \frg$, the simple roots $\Pi = \{\alpha_i \mid i \in I\} \subset \frh^*$ and the simple coroots $\Pi^\vee = \{\alpha^\vee_i \mid i \in I\} \subset \frh$.

We also use the notation of dynamical parameters in \cref{sss:E:dyn}.
Thus, we have the linear spaces \eqref{eq:E:H}, \eqref{eq:E:wtH}, \eqref{eq:E:H*}
\begin{align}\label{eq:U:H}
\begin{split}
    H   = \sum_{i \in I} \bbC P_i \ \subset \ 
&\wtH   = \sum_{i \in I} \bbC(P_i+h_i) + \sum_{i \in I} \bbC P_i, \\
    H^* = \sum_{i \in I} \bbC Q_i \ \subset \ 
&\wtH^* = \sum_{i \in I} \bbC(Q_i+\alpha_i) + \sum_{i \in I} \bbC Q_i
\end{split}
\end{align}
equipped with the bilinear form \eqref{eq:E:Hpair}:
\[
 \pair{Q_i,P_j} = b_{ij}, \quad \pair{\alpha_i,h_j} = a_{ji}, \quad 
 \pair{\text{others}} = 0.
\]
We will regard $H$ as a commutative Lie algebra, and apply the formalism of $H$-Hopf algebroids in \cref{ss:pre:Halgd}.

Finally, we introduce the structure functions satisfying a modified Ding-Iohara condition.

\begin{dfn}\label{dfn:U:DI}
Let $p$ be a formal parameter, and $\{G^\pm_{ij}(z;p) \mid i,j \in I\}$ be a set of functions satisfying either of the following \emph{Ding-Iohara condition}.
\begin{clist}
\item
$G^\pm_{ij}(z;p) \in \bbC[z^{\pm1}]\dbr{p}$, and is invertible in $\bbC\dpr{z}\dbr{p}$.
\item
The series 
\[
 g_{ij}(z;p) \ceq \frac{G^+_{ij}(z;p)}{G^-_{ij}(z;p)} \in \bbC\dpr{z}\dbr{p}
\]
satisfies 
\[
 g_{ij}(z^{-1};p) = g_{ji}(z;p)^{-1}.
\]
\end{clist}
We call $g \ceq \{g_{ij}(z;p) \mid i,j \in I\}$ \emph{the structure functions}.
\end{dfn}

Note that we have definite limits 
\begin{align}\label{eq:U:olG}
 \ol{G}^\pm_{ij}(z) \ceq \lim_{p \to 0}G^\pm_{ij}(z;p)
\end{align}
which belong to $\bbC[z^{\pm1}]$, and these limits are invertible in the field $\bbC\dpr{z}$.
In particular, we have a well-defined function
\begin{align}\label{eq:U:olg}
 \ol{g}_{ij}(z) \ceq \frac{\ol{G}^+_{ij}(z)}{\ol{G}^-_{ij}(z)} \ \in \bbC\dpr{z},
\end{align}
which satisfy the condition \eqref{eq:DI:cond}:
\begin{align}\label{eq:U:olg-cond}
 \ol{g}_{ij}(z^{-1}) = \ol{g}_{ji}(z)^{-1}.
\end{align}

Recall that we are considering the $p$-adic topology. Now let us introduce:

\begin{dfn}\label{dfn:U:ADE}
Let $A=A(X_l)$ be the Cartan matrix of type $X_l=A_l,D_l$ or $E_{6,7,8}$, and take the symmetrization $A=D B$ with $D^{-1}=(d_1,\dotsc,d_l)$ in \eqref{eq:E:D}, a complex parameter $q$ with $0<\abs{q}<1$, a formal parameter $p$, and the structure functions $g=\{g_{i,j} \mid i,j \in I\}$ as in \cref{dfn:U:DI}. 
We define 
\[
 U_{q,p}(g,X_l)
\]
to be the topological algebra over $\bbC\dbr{p}$ generated by 
\begin{align}\label{eq:ADE:gen}
 \clM(\wtH^*), \ q^{\pm c/2}, \ d, \ K^{\pm}_i, \ 
 e_{i,m}, \ f_{i,m}, \ \psi^{\pm}_{i,m} \quad (i \in I, \, m  \in \bbZ)
\end{align} 
with the following relations. Below we use the generating currents
\[
        e_i(z) \ceq \sum_{m \in \bbZ}        e_{i,m} z^{-m}, \quad 
        f_i(z) \ceq \sum_{m \in \bbZ}        f_{i,m} z^{-m}, \quad 
 \psi^\pm_i(z) \ceq \sum_{m \in \bbZ} \psi^\pm_{i,m} z^{-m},
\]
and the genuine relations are obtained by expansion with respect to $p$, as in \cref{dfn:E:U} of the elliptic algebra $U_{q,p}(\wh{\frg})$.
\begin{itemize}
\item
The generators $\clM(\wtH^*),q^{\pm c/2},\psi_i^{\pm}(z)$ and $K^\pm_i$ should satisfy:
\begin{align}\label{eq:U:qc}
\begin{split}
&\text{The field $\clM(\wtH^*)$ is a subalgebra.} \\
&\text{$q^{\pm c/2}$ are central and $q^{c/2}q^{-c/2}=1$.} \\
&\text{$\psi_i^{\pm}(z)$ and $K^\pm_i$ are invertible.} 
\end{split}
\end{align}

\item
For $F(P),F(P+h) \in \clM(\wtH^*)$, we have 
\begin{align}
\label{eq:U:Fd}
&F(P+h) d      = d       F(P+h), & 
&F(P  ) d      = d       F(P), \\
\label{eq:U:FK}
&F(P+h) K^\pm_i = K^\pm_i F(P+h \mp \pair{Q_i,P+h}), & 
&F(P  ) K^\pm_i = K^\pm_i F(P   \mp \pair{Q_i,P  }), \\
\label{eq:U:Fe}
&F(P+h) e_i(z)  = e_j(z) F(P+h), & 
&F(P  ) e_i(z)  = e_i(z) F(P  -\pair{Q_i,P}), \\
\label{eq:U:Ff}
&F(P+h) f_i(z)  = f_i(z) F(P+h-\pair{Q_i,P+h}), &
&F(P  ) f_i(z)  = f_i(z) F(P), \\
\label{eq:U:Fp}
&F(P+h) \psi^\pm_i(z) = \psi^\pm_i(z) F(P+h-\pair{Q_i,P+h}), &
&F(P  ) \psi^\pm_i(z) = \psi^\pm_i(z) F(P  -\pair{Q_i,P  }).
\end{align}

\item
The remaining relations with $d$ are 
\begin{align}
 [d,K^\pm_i] = 0, \quad 
 [d,e_{i,m}] = m e_{i,m}, \quad 
 [d,f_{i,m}] = m f_{i,m}, \quad
 [d,\psi^\pm_i] = m \psi^\pm_{i,m}.
\end{align}

\item
The remaining relations containing $K^\pm_i$ are given by 
\begin{align}\label{eq:U:Ke}
 K^\pm_i e_j(z)  = q_i^{\mp a_{ij}} e_j(z) K^\pm_i, \quad 
 K^\pm_i f_j(z)  = q_i^{\pm a_{ij}} f_j(z) K^\pm_i, \quad 
[K^\pm_i, \psi^\pm_{i,m}]=0.
\end{align}

\item
Hereafter we use
\[
 g^{ }_{ij}(z) \ceq g_{ij}(z;p  ), \quad 
 g^{*}_{ij}(z) \ceq g_{ij}(z;p^*).
\]
Then the remaining relations containing $\psi^\pm_i$ are given by 
\begin{align}
\label{eq:U:pp}
&\psi^\pm_i(z) \psi^\pm_j(w) =
 \frac{g^*_{ij}(      \frac{z}{w})}{g_{ij}(   \frac{z}{w})} \psi^\pm_j(w) \psi^\pm_i(z), &  
&\psi^{+}_i(z) \psi^{-}_j(w) =
 \frac{g^*_{ij}(q^{-c}\frac{z}{w})}{g_{ij}(q^c\frac{z}{w})} \psi^{-}_j(w) \psi^{+}_i(z), 
\\
\label{eq:U:ppe}
&\psi^+_i(z) e_j(w) = g_{ij}^*(q^{-c/2}\tfrac{z}{w})     e_j(w) \psi^+_i(z), &
&\psi^+_i(z) f_j(w) = g_{ij}  (q^{ c/2}\tfrac{z}{w})^{-1}f_j(w) \psi^+_i(z), 
\\
\label{eq:U:pme}
&\psi^-_i(z) e_j(w) = g_{ij}^*(q^{c/2}\tfrac{z}{w})e_j(w) \psi^-_i(z), &
&\psi^-_i(z) f_j(w) = g_{ij}  (q^{-c/2}\tfrac{z}{w})^{-1}     f_j(w) \psi^-_i(z).
\end{align}

\item
Hereafter we use
\[
 G^{\pm  }_{ij}(z) \ceq G^{\pm}_{ij}(z;p  ), \quad 
 G^{\pm *}_{ij}(z) \ceq G^{\pm}_{ij}(z;p^*).
\]
Then the remaining relations consist of the quadratic relations 
\begin{gather}
\label{eq:U:ee}
 G^{-*}_{ij}(\tfrac{z}{w}) e_i(z) e_j(w) = 
 G^{+*}_{ij}(\tfrac{z}{w}) e_j(w) e_i(z), \quad 
 G^{+ }_{ij}(\tfrac{z}{w}) f_i(z) f_j(w) = 
 G^{- }_{ij}(\tfrac{z}{w}) f_j(w) f_i(z), 
\\
\label{eq:U:ef}
 [e_i(z), f_j(w)] = \frac{\delta_{i,j}}{q-q^{-1}}
 \bigl(\delta(q^{-c}\tfrac{z}{w})\psi^-_i(q^{c/2}w)
      -\delta(q^{ c}\tfrac{z}{w})\psi^+_i(q^{c/2}z)\bigr)
\end{gather}
for any $i,j \in I$, and the following Serre-type relations \eqref{eq:U:eS} and \eqref{eq:U:fS} for $i,j \in I$ such that $a_{ij}=-1$. Using $\wt{g}^*_{ij}(z) \ceq \ol{g}_{ij}(z)/g^*_{ij}(z)$ with $\ol{g}_{ij}(z)$ given in \eqref{eq:U:olg} and 
\begin{align}\label{eq:U:hij}
 h_{ij} \ceq 
 \frac{(\ol{g}_{ii}(\frac{z_1}{z_2})+1)(\ol{g}_{ij}(\frac{z_1}{z  })\ol{g}_{ij}(\frac{z_2}{z})+1)}
      { \ol{g}_{ij}(\frac{z_2}{z  })+   \ol{g}_{ii}(\frac{z_1}{z_2})\ol{g}_{ij}(\frac{z_1}{z})   },
\end{align}
the Serre-type relation for $e_i(z),e_j(z)$ is given by 
\begin{align}\label{eq:U:eS}
\begin{split}
&\wt{g}^*_{ii}(\tfrac{z_1}{z_2})
 \wt{g}^*_{ij}(\tfrac{z_1}{z  })
 \wt{g}^*_{ij}(\tfrac{z_2}{z  }) e_i(z_1) e_i(z_2) e_j(z)
-h_{ij}
 \wt{g}^*_{ii}(\tfrac{z_1}{z_2})
 \wt{g}^*_{ij}(\tfrac{z_1}{z  }) e_i(z_1) e_j(z) e_i(z_2) \\
&
+\wt{g}^*_{ii}(\tfrac{z_1}{z_2}) e_j(z) e_i(z_1) e_i(z_2)
+\wt{g}^*_{ij}(\tfrac{z_1}{z  }) 
 \wt{g}^*_{ij}(\tfrac{z_2}{z  }) e_i(z_2) e_i(z_1) e_j(z) \\
&
-h_{ij}
 \wt{g}^*_{ij}(\tfrac{z_2}{z  }) e_i(z_2) e_j(z) e_i(z_1)
+e_j(z) e_i(z_2) e_i(z_1)=0.
\end{split}
\end{align}
Similarly, using $\wt{g}_{ij}(z) \ceq g_{ij}(z)/\ol{g}_{ij}(z)$, the Serre-type relation for $f_i(z),f_j(z)$ is given by
\begin{align}\label{eq:U:fS}
\begin{split}
&\wt{g}_{ii}(\tfrac{z_1}{z_2})
 \wt{g}_{ij}(\tfrac{z_1}{z  })
 \wt{g}_{ij}(\tfrac{z_2}{z  }) f_i(z_1)f_i(z_2)f_j(z)
-h_{ij}
 \wt{g}_{ii}(\tfrac{z_1}{z_2})
 \wt{g}_{ij}(\tfrac{z_1}{z  }) f_i(z_1)f_j(z)f_i(z_2) \\
&
+\wt{g}_{ii}(\tfrac{z_1}{z_2}) f_j(z)f_i(z_1)f_i(z_2)
+\wt{g}_{ij}(\tfrac{z_1}{z  })
 \wt{g}_{ij}(\tfrac{z_2}{z  }) f_i(z_2)f_i(z_1)f_j(z) \\
&
-h_{ij}
 \wt{g}_{ij}(\tfrac{z_2}{z}) f_i(z_2)f_j(z)f_i(z_1)
+f_j(z)f_i(z_2)f_i(z_1)=0.
\end{split}
\end{align}
Note that the function $h_{ij}$ of \eqref{eq:U:hij} coincides with \eqref{eq:DI:hij} appearing in the Serre-type relations for the Ding-Iohara quantum algebra with $g_{ij}(z)$ replaced by $\ol{g}_{ij}(z)$.
\end{itemize}
\end{dfn}

\begin{rmk}
By the same discussion of \cref{rmk:E:U} \ref{i:rmk:E:U:d'}, the topological algebra $U_{q,p}(g,X_l)$ is independent, up to isomorphism, of the choice of a symmetrization of the Cartan matrix $A$.
\end{rmk}

\begin{rmk}\label{rmk:ADE:top}
Instead of \cref{dfn:U:DI}, we can use broader class of the structure functions $g=\{g_{ij}(z;p) \mid i,j \in I\}$ with some care. Let us explain two such classes.
\begin{enumerate}
\item \label{i:rmk:ADE:top:p'}
Let $p$ and $p'$ be formal parameters, and $\{G^\pm_{ij}(z;p) \mid i,j \in I\}$ be a set of series satisfying the following two conditions.
\begin{clist}
\item
$G^\pm_{ij}(z;p) \in \bbC[z^{\pm1}]\dbr{p'}\dbr{p}$, and is invertible in $\bbC\dpr{z}\dbr{p'}\dbr{p}$.

\item
The series $g_{ij}(z;p) \ceq G^+_{ij}(z;p)/G^-_{ij}(z;p) \in \bbC\dpr{z}\dbr{p'}\dbr{p}$
satisfies $g_{ij}(z^{-1};p) = g_{ji}(z;p)^{-1}$.
\end{clist}
Then we have the definite limits $\ol{g}_{ij}(z) \in \bbC\dbr{z}\dbr{p'}$, and considering the $p$- and $p'$-adic topology, we obtain the topological algebra $U$.

\item
Let $\{G^\pm_{ij}(z;p) \mid i,j \in I\}$ satisfy the following two conditions.
\begin{clist}
\item
$G^\pm_{ij}(z;p) \in \clA \dbr{p}$, and is invertible in $(\Frac \clA)\dbr{p}$.
Here we used the function ring $\clA$ in  \eqref{eq:DI:clA}.

\item
The series $g_{ij}(z;p) \ceq G^+_{ij}(z;p)/G^-_{ij}(z;p) \in (\Frac \clA)\dbr{p}$
satisfies $g_{ij}(z^{-1};p) = g_{ji}(z;p)^{-1}$.
\end{clist}
Then we have the definite limits $\ol{g}_{ij}(z) \in \Frac \clA$, and obtain the genuine relations of $U$ by expanding generating currents and series first in terms of $p$, and second in terms of the current/series variable. However, this process yields relations containing infinite summations in general, and we should take some solution such as \ref{i:rmk:DI:top:b} in \cref{rmk:DI:top}. 
\end{enumerate}
\end{rmk}

The Serre-type relations \eqref{eq:U:eS} and \eqref{eq:U:fS} look rather complicated, but the other defining relations are the same as those of the elliptic algebra $U_{q,p}(\wh{g})$ of type $X_l=A_l,D_l$ or $E_{6,7,8}$ in \cref{dfn:E:U}. Indeed, recalling \eqref{eq:E:gij}, we have:

\begin{prp}\label{prp:U:ADE}
In \cref{dfn:U:ADE} of $U_{q,p}(g,X_l)$ for type $X_l=A_l,D_l,E_{6,7,8}$, 
set the structure functions $g=\{g_{ij}(z;p) \mid i,j \in I\}$ as
\begin{align}\label{eq:U:gijE}
 g_{ij}(x;p) = \frac{G_{ij}^+(x;p)}{G_{ij}^-(x;p)} \ceq 
 \frac{q^{-b_{ij}} \theta(q^{b_{ij}}x;p)}{\theta(q^{-b_{ij}}x;p)},
\end{align}
where $B=(b_{ij})_{i,j \in I}$ is the symmetrization $A=D B$ of the Cartan matrix $A=A(X_l)$. Note that these satisfy the Ding-Iohara condition (\cref{dfn:U:DI}). Then we have an isomorphism of topological algebras
\[
 U_{q,p}(g,X_l) \cong U_{q,p}(\wh{g}).
\]
\end{prp}

\begin{proof}
The generators of both topological algebra coincide. Also, using \cref{lem:E:pef,lem:E:pp,lem:E:ee} and comparing the fundamental relations of $U_{q,p}(\wh{g})$ in \cref{dfn:E:U} and those of $U_{q,p}(g,X_l)$ in \cref{dfn:U:ADE}, we immediately find that they coincide except the Serre-type relations \eqref{eq:E:eS} and \eqref{eq:E:fS}. So it remains to check that the specialization \eqref{eq:U:gijE} turns the relations \eqref{eq:U:eS} and \eqref{eq:U:fS} into \eqref{eq:E:eS} and \eqref{eq:E:fS}, respectively. First, since we are considering the simply-laced types, the structure functions $g^*_{ij}(z)=g_{ij}(z;p^*)$ and their limits $\ol{g}_{ij}(z)$ are given by 
\begin{align*}
 g^*_{ii}(z) = q^{-2} \frac{\theta^*(q^{ 2}z)}{\theta^*(q^{-2}z)}, \quad 
 g^*_{ij}(z) = q^{  } \frac{\theta^*(q^{-1}z)}{\theta^*(q^{  }z)}, \quad
 \ol{g}_{ii}(z) = q^{-2} \frac{1-q^{ 2}z}{1-q^{-2}z}, \quad 
 \ol{g}_{ij}(z) = q^{  } \frac{1-q^{-1}z}{1-q^{  }z}
\end{align*}
with $\theta^*(x) \ceq \theta(x;p^*) = (x,p^*/x;p^*)_\infty$. 
Then, using \eqref{eq:E:bini}, the function $h_{ij}$ in \eqref{eq:U:hij} is equal to 
\begin{align}\label{eq:ADE:h=2q}
 h_{ij} = q+q^{-1} = \bnm{2}{1}{i}
\end{align}
for any $i \ne j$. (See \cref{ss:ns:h} for a generalization of this equality.) Also, denoting $z_{a0} \ceq z_a/z_0$ and $z_{0a} \ceq z_0/z_a$ for $a=1,2$, we have
\begin{align}\label{eq:ADE:olg/g*}
 \frac{\ol{g}_{ij}(z_{a0})}{g^*_{ij}(z_{a0})} = 
 \frac{(p^* q z_{a0},p^*/(q z_{a0});p^*)}{(p^* q^{-1}z_{a0},p^*/(q^{-1} z_{a0});p^*)} =
 \frac{\wtgm^*(z_{0a};q^{-1})}{\wtgm^*(z_{a0};q^{-1})},
\end{align}
where $\wtgm^*(z;q) \ceq \wtgm(z;q,p^*)=(p^* q x;p^*)_\infty/(p^* q^{-1} z;p^*)_\infty$ using the function $\wtgm$ in \eqref{eq:E:wtgam}. 
Now, we can rewrite the last second term of \eqref{eq:U:eS} as
\begin{align}\label{eq:ADE:eS5}
-h_{ij} \frac{\ol{g}_{ij}(z_{20})}{g^*_{ij}(z_{20})} e_i(z_2) e_j(z) e_i(z_1) = 
-\bnm{2}{1}{i}
 \frac{\wtgm^*(z_{12};q^{ 2})}{\wtgm^*(z_{12};q^{ 2})}
 \frac{\wtgm^*(z_{02};q^{-1})}{\wtgm^*(z_{20};q^{-1})}
 \frac{\wtgm^*(z_{10};q^{-1})}{\wtgm^*(z_{10};q^{-1})} e_i(z_2) e_j(z) e_i(z_1),
\end{align}
where we denoted $z_{12} \ceq z_1/z_2$. 
Note that we inserted two trivial fractions in the right hand side. 
We can similarly rewrite the fourth term of \eqref{eq:U:eS} as
\begin{align}\label{eq:ADE:eS4}
 \frac{\ol{g}_{ij}(z_{10})}{g^*_{ij}(z_{10})}
 \frac{\ol{g}_{ij}(z_{20})}{g^*_{ij}(z_{20})} e_i(z_2) e_i(z_1) e_j(z) = 
 \frac{\wtgm^*(z_{12};q^{ 2})}{\wtgm^*(z_{12};q^{ 2})}
 \frac{\wtgm^*(z_{01};q^{-1})}{\wtgm^*(z_{10};q^{-1})}
 \frac{\wtgm^*(z_{02};q^{-1})}{\wtgm^*(z_{20};q^{-1})} e_i(z_2) e_i(z_1) e_j(z).
\end{align}
Next, using 
\begin{align}\label{eq:ADE:olg/g}
 \frac{\ol{g}_{ii}(z_{12})}{g^*_{ii}(z_{12})} = 
 \frac{(p^* q^{-2} z_{12},p^*/(q^{-2} z_{12});p^*)}{(p^* q^2 z_{12},p^*/(q^2 z_{12});p^*)} =
 \frac{\wtgm^*(z_{21};q^2)}{\wtgm^*(z_{12};q^2)}
\end{align}
with $z_{21} \ceq z_2/z_1$, we can rewrite the third term of \eqref{eq:U:eS} as
\begin{align}\label{eq:ADE:eS3}
 \frac{\ol{g}_{ii}(z_{12})}{g^*_{ii}(z_{12})} e_j(z) e_i(z_1) e_i(z_2) = 
 \frac{\wtgm^*(z_{21};q^{ 2})}{\wtgm^*(z_{12};q^{ 2})}
 \frac{\wtgm^*(z_{10};q^{-1})}{\wtgm^*(z_{10};q^{-1})}
 \frac{\wtgm^*(z_{20};q^{-1})}{\wtgm^*(z_{20};q^{-1})} e_j(z) e_i(z_1) e_i(z_2),
\end{align}
By a similar calculation, we can rewrite the second term as 
\begin{align}
\label{eq:ADE:eS2}
\begin{split}
&-h_{ij}
 \frac{\ol{g}_{ii}(z_{12})}{g^*_{ii}(z_{12})}
 \frac{\ol{g}_{ij}(z_{10})}{g^*_{ij}(z_{10})} e_i(z_1) e_j(z) e_i(z_2) \\
&=-\bnm{2}{1}{i}
 \frac{\wtgm^*(z_{21};q^{ 2})}{\wtgm^*(z_{12};q^{ 2})}
 \frac{\wtgm^*(z_{01};q^{-1})}{\wtgm^*(z_{10};q^{-1})}
 \frac{\wtgm^*(z_{20};q^{-1})}{\wtgm^*(z_{20};q^{-1})} e_i(z_1) e_j(z) e_i(z_2),
\end{split}
\end{align}
and the first terms as
\begin{align}
\label{eq:ADE:eS1}
\begin{split}
&\frac{\ol{g}_{ii}(z_{12})}{g^*_{ii}(z_{12})}
 \frac{\ol{g}_{ij}(z_{10})}{g^*_{ij}(z_{10})}
 \frac{\ol{g}_{ij}(z_{20})}{g^*_{ij}(z_{20})} e_i(z_1) e_i(z_2) e_j(z) \\
&=
 \frac{\wtgm^*(z_{21};q^{ 2})}{\wtgm^*(z_{12};q^{ 2})}
 \frac{\wtgm^*(z_{01};q^{-1})}{\wtgm^*(z_{10};q^{-1})}
 \frac{\wtgm^*(z_{02};q^{-1})}{\wtgm^*(z_{20};q^{-1})} e_i(z_1) e_i(z_2) e_j(z).
\end{split}
\end{align}
Finally, combining \eqref{eq:ADE:eS5}--\eqref{eq:ADE:eS1}, reordering, and multiplying by 
$\wtgm^*(z_{21};q^2) \wtgm^*(z_{20};q^{-1}) \wtgm^*(z_{10};q^{-1})$,
we can rewrite \eqref{eq:U:eS} as
\begin{align*}
0=
&\wtgm^*(z_{21};q^2)\Bigl(
 \wtgm^*(z_{10};q^{-1}) \wtgm^*(z_{20};q^{-1}) e_j(z) e_i(z_1) e_i(z_2)
-\bnm{2}{1}{i}
 \wtgm^*(z_{01};q^{-1}) \wtgm^*(z_{20};q^{-1}) e_i(z_1) e_j(z) e_i(z_2) \\
&\hspace{4em}
+\wtgm^*(z_{01};q^{-1}) \wtgm^*(z_{02};q^{-1}) e_i(z_1) e_i(z_2) e_j(z)\Bigr) 
\\
+
&\wtgm^*(z_{12};q^2)\Bigl(
 \wtgm^*(z_{20};q^{-1}) \wtgm^*(z_{10};q^{-1}) e_j(z) e_i(z_2) e_i(z_1)
-\bnm{2}{1}{i}
 \wtgm^*(z_{02};q^{-1}) \wtgm^*(z_{10};q^{-1}) e_i(z_2) e_j(z) e_i(z_1) \\
&\hspace{4em}
+\wtgm^*(z_{02};q^{-1}) \wtgm^*(z_{01};q^{-1}) e_i(z_2) e_i(z_1) e_j(z)\Bigr),
\end{align*}
which is nothing but \eqref{eq:E:eS} for simply-laced types.
In the same way, we can check that \eqref{eq:U:fS} is equivalent to \eqref{eq:E:fS} under the assumption \eqref{eq:U:gijE}. Hence we have the statement.
\end{proof}

\begin{rmk}
We discovered the combination $\ol{g}_{ij}(z_a/z_b)/g_{ij}(z_a/z_b)$ and $\ol{g}_{ij}(z_a/z_b)/g^*_{ij}(z_a/z_b)$ in the Serre-type relations \eqref{eq:U:eS} and \eqref{eq:U:fS} to establish the equalities \eqref{eq:ADE:olg/g*} and \eqref{eq:ADE:olg/g} with the function $\wt{\gamma}$. The function $h_{ij}$ is discovered by the well-definedness of the Drinfeld-type comultiplication, which will be commented in \cref{rmk:ADE:h}.
\end{rmk}

\subsection{\texorpdfstring{$H$}{H}-Hopf algebroid structure}\label{ss:ADE:Halgd}

Hereafter in this \cref{ss:ADE:Halgd}, we denote
\[
 U \ceq U_{q,p}(g,X_l), \quad \bbF \ceq \clM(H^*)
\]
for simplicity. Let us introduce an $H$-Hopf algebroid structure on $U$ with $H=\sum_{i \in I}\bbC P_i$ in \eqref{eq:U:H}. We begin with:

\begin{lem}\label{lem:U:Halg}
$U=U_{q,p}(g,X_l)$ is a topological $H$-algebra with the $H^*$-bigrading $U=\bigoplus_{\alpha,\beta \in H^*} U_{\alpha,\beta}$ given by
\begin{align*}
 U_{\alpha,\beta} \ceq \{a \in U \mid 
  q^{P  } a q^{-P  } = q^{\pair{\alpha,P }}a, \ 
  q^{P+h} a q^{-P-h} = q^{\pair{\beta,P+h}}a \quad (P,P+h \in \wtH) \},
\end{align*}
and the moment maps $\mu_l,\mu_r\colon \bbF\dbr{p} \to U_{0,0}$ defined by
\begin{align*}
 \mu_l(F) \ceq F(P,p^*), \quad \mu_r(F) \ceq F(P+h,p)
\end{align*}
for $F=F(P,p) \in \bbF\dbr{p}$. 
Moreover, the bigrading values in the lattice $\clQ \ceq \bigoplus_{i \in I} \bbZ Q_i \subsetneq H^*$.
\end{lem}

\begin{proof}
The conditions of the moment maps can be checked directly. For instance, 
\begin{align*}
&\mu_l(F)(e_i(z)) = F(P,p^*) e_i(z) = e_i(z) F(P-\pair{Q_i,P},p) =
 e_i(z) \mu_l\bigl(T_{-Q_i}(F)\bigr), \\
&\mu_r(F)(e_i(z)) = F(P+h,p^*)e_i(z) = e_i(z) F(P+h,p) = 
 e_i(z) \mu_r\bigl(T_0(F)\bigr).
\end{align*}
The last statement on the bigrading holds since 
\begin{align*}
 e_i(z) \in U_{-Q_i,0}, \quad f_i(z) \in U_{0,-Q_i}, \quad
 \psi^\pm_i(z) \in U_{-Q_i,-Q_i}, \quad  K^\pm_i \in U_{-Q_i,-Q_i}
\end{align*}
(see \cref{lem:E:Halg} \ref{i:lem:E:gen} for the meaning)
and the other generators belong to $U_{0,0}$ 
\end{proof}

For later use, let us write down the isomorphisms $D_H \totimes U \cong U \cong U \totimes D_H$ in \eqref{eq:Halgd:Delta-ve}. By \eqref{eq:Halgd:to} and \eqref{eq:Halgd:Dh}, we have
\begin{align*}
& (D_H \totimes U)_{\alpha,\beta} 
= (D_H)_{\alpha,\alpha} \otimes_{\bbF} U_{\alpha,\beta}
= \bbF T_{-\alpha} \otimes_{\bbF} U_{\alpha,\beta}, \\
& (U \totimes D_H)_{\alpha,\beta} 
= U_{\alpha,\beta} \otimes_{\bbF} (D_H)_{\beta,\beta}
= U_{\alpha,\beta} \otimes_{\bbF} \bbF T_{-\beta} 
\end{align*}
for $\alpha,\beta \in \clQ \subset H^*$. Then, by \eqref{eq:Halgd:totimes}, we find that the isomorphisms are given by
\begin{align}\label{eq:ADE:DU=U}
 (D_H \totimes U)_{\alpha,\beta} = \bbF T_{-\alpha} \otimes_{\bbF} U_{\alpha,\beta} 
 \lsto U_{\alpha,\beta}, \quad 
 F T_{-\alpha} \totimes a  = T_{-\alpha} \totimes \mu_l(F) a \lmto \mu_l(F) a
\end{align}
and
\begin{align}\label{eq:ADE:UD=U}
 (U \totimes D_H)_{\alpha,\beta} = U_{\alpha,\beta} \otimes_{\bbF} \bbF T_{-\beta} 
 \lsto U_{\alpha,\beta}, \quad 
 a \totimes F T_{-\beta} = \mu_r(F) a \otimes T_{-\beta} \lmto a \mu_r(F)
\end{align}
for $a \in U_{\alpha,\beta}$ and $F \in \bbF\dbr{p}$.

Next, we introduce a Drinfeld-type comultiplication on $U$.

\begin{lem}\label{lem:U:Delta}
We define a homomorphism $\Delta \colon U \to U \totimes U$ of topological algebras by 
\begin{align}\label{eq:U:Delta0}
 \Delta(\mu_l(F)) \ceq \mu_l(F) \totimes 1, \quad 
 \Delta(\mu_r(F)) \ceq 1 \totimes \mu_r(F) 
\end{align}
for $F=F(P,p) \in \bbF\dbr{p}$, 
\begin{align}
 \Delta(q^{\pm c/2}) \ceq q^{\pm c/2} \totimes q^{\pm c/2}, \quad 
 \Delta(d)       \ceq d \totimes 1 + 1 \totimes d, \quad 
 \Delta(K^\pm_i) \ceq K^\pm_i \totimes K^\pm_i, 
\end{align}
and 
\begin{align} 
 \Delta(e_i(z)) \ceq e_i(z) \totimes 1+ \psi^+_i(q^{c_1/2}z) \totimes e_i(q^{c_1}z), \qquad 
&\Delta(f_i(z)) \ceq 1 \totimes f_i(z)+ f_i(q^{c_2}z) \totimes \psi^-_i(q^{c_2/2}z), 
\\
\label{eq:U:Delta1}
 \Delta(\psi_i^+(z)) \ceq \psi_i^+(q^{-c_2/2}z) \totimes \psi_i^+(q^{c_1/2}z), \qquad 
&\Delta(\psi_i^-(z)) \ceq \psi_i^-(q^{c_2/2}z) \totimes \psi_i^-(q^{- c_1/2}z)
\end{align}
for the generating currents. Then $\Delta$ is a topological $H$-algebra homomorphism. 
\end{lem}


\begin{proof}
First, we show that $\Delta$ is well-defined as a homomorphism of topological algebras. We will only write down the check for a part of the defining relations of $U$. We can check the block \eqref{eq:U:FK}--\eqref{eq:U:Fp} as
\begin{align*}
  \Delta(F(P))\Delta(e_i(z))
&=\bigl(1 \totimes F(P)\bigr)
  \bigl(1 \totimes e_i(z)+ e_i(q^{c_2}z)\totimes \psi_i^+(q^{c_2/2}z)\bigr) \\
&=1 \totimes F(P)e_i(z)+ e_i(q^{c_2}z)\totimes F(P)\psi_i^+(q^{c_2/2}z) \\
&=1 \totimes e_i(z)F(P-\pair{Q_i,P}) 
  + e_i(q^{c_2}z) \totimes \psi_i^+(q^{c_2/2}z)F(P-\pair{Q_i,P}) \\
&=\Delta(e_i(z))\Delta(F(P-\pair{Q_i, P})).
\end{align*}
The relations \eqref{eq:U:Ke} can be checked as
\begin{align*}
  \Delta(K^\pm_i)\Delta(\psi^+_j(w))
&=\bigl(K^\pm_i \totimes K^\pm_i\bigr)
  \bigl( \psi^+_j(q^{c_2/2}w) \totimes \psi^+_j(q^{-c_1/2}w)\bigr) 
 =K^\pm_i\psi^+_j(q^{c_2/2}w) \totimes K^\pm_i \psi^+_j(q^{-c_1/2}w) \\
&=\psi_j^+(q^{c_2/2}w) K^\pm_i \totimes \psi^+_j(q^{-c_1/2}w) K^\pm_i 
 =\Delta(\psi^+_j(w)) \Delta(K^\pm_i).
\end{align*}
The block \eqref{eq:U:pp}--\eqref{eq:U:pme} can be checked as
\begin{align*}
  \Delta(\psi_i^+(z))\Delta(\psi_j^+(w))
&=\psi_i^+(q^{ c_2/2}z)\psi_j^+(q^{ c_2/2}w) \totimes 
  \psi_i^+(q^{-c_1/2}z)\psi_j^+(q^{-c_1/2}w) \\
&=\frac{g_{ij}^*(\frac{z}{w})}{g_{ij}(\frac{z}{w})}
  \psi_j^+(q^{ c_2/2}w)\psi_i^+(q^{c_2/2}z) \totimes 
  \frac{g_{ij}^*(\frac{z}{w})}{g_{ij}(\frac{z}{w})} 
  \psi_j^+(q^{-c_1/2}w) \psi_i^+(q^{-c_1/2}z) \\
&=\Delta\left(\frac{g_{ij}^*(\frac{z}{w})}{g_{ij}(\frac{z}{w})}\right)
  \Delta(\psi_j^+(w))\Delta(\psi_i^+(z)),
\\
  \Delta(\psi_i^+(z))\Delta(\psi_j^-(w))
&=\psi_i^+(q^{ c_2/2}z) \psi_j^-(q^{-c_2/2}w) \totimes 
  \psi_i^+(q^{-c_1/2}z) \psi_j^-(q^{ c_1/2}w) \\
&=\frac{g_{ij}^*(q^{-c_1+c_2}\frac{z}{w})}{g_{ij}(q^{ c_1+c_2}\frac{z}{w})}
  \psi_j^-(q^{-c_2/2}w)\psi_i^+(q^{ c_2/2}z) \totimes 
  \frac{g_{ij}^*(q^{-c_1-c_2}\frac{z}{w})}{g_{ij}(q^{-c_1+c_2}\frac{z}{w})}
  \psi_j^-(q^{-c_2/2}w)\psi_i^+(q^{ c_2/2}z) \\
&=\Delta\left(\frac{g_{ij}^*(q^{-c}\frac{z}{w})}{g_{ij}(q^{c}\frac{z}{w})}\right)
  \Delta(\psi_j^-(w))\Delta(\psi_i^+(z)),  
\\
  \Delta(\psi_i^+(z))\Delta(e_j(w))
&=\bigl(\psi_i^+(q^{-c_2/2}z) \totimes \psi_i^+(q^{c_1/2}z)\bigr)
  \bigl(e_j(w) \totimes 1+\psi^+_j(q^{c_1/2}w) \totimes e_j(q^{c_1}w)\bigr) \\
&=\psi_i^+(q^{-c_2/2}z)e_j(w) \totimes \psi_i^+(q^{c_1/2}z)
 +\psi_i^+(q^{-c_2/2}z) \psi^+_j(q^{c_1/2}w) \totimes
  \psi_i^+(q^{c_1/2}z) e_j(q^{c_1}w) \\
&=g^*_{ij}(q^{\frac{-c_1-c_2}{2}}\tfrac{z}{w})e_j(w)\psi_i^+(q^{-c_2/2}z) \totimes \psi_i^+(q^{c_1/2}z)
   \\
&\quad 
 +\frac{g_{ij}^*(q^{\frac{-c_1-c_2}{2}}\frac{z}{w})}{g_{ij}(q^{\frac{-c_1-c_2}{2}}\frac{z}{w})} 
  \psi^+_j(q^{c_1/2}w)\psi^+_i(q^{-c_2/2}z) \totimes 
 g_{ij}^*(q^{\frac{-c_1-c_2}{2}}\tfrac{z}{w})
 e_j(q^{c_1}w)\psi^+_i(q^{c_1/2}z) \\
&=\Delta\bigl(g_{ij}(q^{-c/2}\tfrac{z}{w})\bigr) \Delta(e_j(w)) \Delta(\psi_i^+(z)).
\end{align*}
As for the quadratic relation \eqref{eq:U:ee}, we have

\begin{align*}
& [\Delta(e_i(z)), \Delta(f_i(w))] \\
&=\bigl(e_i(z) \totimes 1+\psi^+_i(q^{c_1/2}z) \totimes e_i(q^{c_1}z)\bigr)
  \bigl(1 \totimes f_i(w) + f_i(q^{c_2}w) \totimes \psi^-_i(q^{\frac{c_2}{2}}w)\bigr) \\
&\quad 
 -\bigl(1 \totimes f_i(w) + f_i(q^{c_2}w) \totimes \psi^-_i(q^{\frac{c_2}{2}}w)\bigr)
  \bigl(e_i(z) \totimes 1+\psi^+_i(q^{c_1/2}z) \totimes e_i(q^{c_1}z)\bigr) \\
&=\psi_i^+(q^{c_1/2}z) \totimes [e_i(q^{c_1}z),f_i(w)] + [e_i(z),f_i(q^{c_2}w)] \totimes 
  \psi_i^+(q^{c_2/2}w) \\
&\quad 
+\bigl(\psi^+_i(q^{c_1/2}z)f_i(q^{c_2}w) \totimes e_i(q^{c_1}z)\psi^-_i(q^{c_2/2}w) 
-f_i(q^{c_2}w)\psi^+_i(q^{c_1/2}z) \totimes \psi^-_i(q^{c_2/2}w)e_i(q^{c_1}z)\bigr).
\end{align*}

The defining relations imply that the last term vanishes. The first and second terms are 
\begin{align*}
& \psi_i^+(q^{c_1/2}z) \totimes [e_i(q^{c_1}z),f_i(w)] + [e_i(z),f_i(q^{c_2}w)] \totimes \psi_i^+(q^{c_2/2}w) \\
&=\frac{1}{q-q^{-1}}\Bigl(
  \psi_i^+(q^{c_1/2}z) \totimes 
  \bigl(\delta(q^{c_1-c_2}\tfrac{z}{w}) \psi_i^-(q^{c_2/2}w) - 
        \delta(q^{c_1+c_2}\tfrac{z}{w}) \psi_i^+(q^{c_1+c_2/2}z) \bigr) \\
&\hspace{5em}
 +\bigl(\delta(q^{-c_1-c_2}\tfrac{z}{w}) \psi_i^-(q^{c_1/2+c_2}w) - 
        \delta(q^{ c_1-c_2}\tfrac{z}{w}) \psi_i^+(q^{c_1/2}z) \bigr)
  \totimes \psi_i^-(q^{c_2/2}w) \Bigr) \\
&=\frac{1}{q-q^{-1}}\bigl(   
  \delta(q^{-c_1-c_2}\tfrac{z}{w}) \psi_i^-(q^{c_1/2+c_2}w) \totimes \psi_i^+(q^{c_2/2}w) \totimes - 
  \psi^+_i(q^{c_1/2}z) \totimes 
  \delta(q^{ c_1+c_2}\tfrac{z}{w}) \psi_i^+(q^{c_1+c_2/2}z)  \\
&=\frac{1}{q-q^{-1}}\Bigl(
  \Delta\bigl(\delta(q^{-c}\tfrac{z}{w})\bigr) \Delta\bigl(\psi_i^-(q^{c/2}w)\bigr) - 
  \Delta\bigl(\delta(q^{ c}\tfrac{z}{w})\bigr) \Delta\bigl(\psi_i^+(q^{c/2}z)\bigr)\Bigr),
\end{align*}
by which we have checked \eqref{eq:U:ee}.
As for the Serre-type relation \eqref{eq:U:eS}, we have 
\begin{align}
\label{eq:ADE:DS}
\begin{split}
& \Delta\Bigl(
   \frac{\ol{g}_{ii}(\frac{z_1}{z_2})}{g^*_{ii}(\frac{z_1}{z_2})}
   \frac{\ol{g}_{ij}(\frac{z_1}{z})}{g^*_{ij}(\frac{z_1}{z})}
   \frac{\ol{g}_{ij}(\frac{z_2}{z})}{g^*_{ij}(\frac{z_2}{z})}\Bigr)
  \Delta(e_i(z_1)) \Delta(e_i(z_2)) \Delta(e_j(z)) \\
&-\Delta\Bigl(h_{ij} 
   \frac{\ol{g}_{ii}(\frac{z_1}{z_2})}{g^*_{ii}(\frac{z_1}{z_2})}
   \frac{\ol{g}_{ij}(\frac{z_1}{z})}{g^*_{ij}(\frac{z_1}{z})}\Bigr)
  \Delta(e_i(z_1)) \Delta(e_j(z)) \Delta(e_i(z_2)) \\
&+\Delta\Bigl(\frac{\ol{g}_{ii}(\frac{z_1}{z_2})}{g^*_{ii}(\frac{z_1}{z_2})}\Bigr)
  \Delta(e_j(z)) \Delta(e_i(z_1)) \Delta(e_i(z_2))
 +\Delta\Bigl(
   \frac{\ol{g}_{ij}(\frac{z_1}{z})}{g^*_{ij}(\frac{z_1}{z})}
   \frac{\ol{g}_{ij}(\frac{z_2}{z})}{g^*_{ij}(\frac{z_2}{z})}\Bigr)
  \Delta(e_i(z_2)) \Delta(e_i(z_1)) \Delta(e_j(z)) \\
&-\Delta\Bigl(h_{ij} \frac{\ol{g}_{ij}(\frac{z_2}{z})}{g^*_{ij}(\frac{z_2}{z})}\Bigr)
  \Delta(e_i(z_2)) \Delta(e_j(z)) \Delta(e_i(z_1))
 +\Delta(e_j(z)) \Delta(e_i(z_2)) \Delta(e_i(z_1))
 = 0.
\end{split}
\end{align}
The remaining relations can be checked similarly. 

Next we prove that $\Delta$ is a homomorphism of topological $H$-algebras. It is clear that $\Delta$ preserves the $H^*$-bigradings. The compatibility with the moment maps can be checked as 
\begin{align*}
 \Delta(\mu_l(F)) = \mu_l(F) \totimes 1 = \mu_l^{\totimes}(F), \quad 
 \Delta(\mu_r(F)) = 1 \totimes \mu_r(F) = \mu_r^{\totimes}(F)
\end{align*}
for $F=F(P,p) \in \bbF\dbr{p}$, where $\mu_l^{\totimes}$ and $\mu_r^{\totimes}$ denote the moment maps of the tensor product $U \totimes U$ (see \eqref{eq:Halgd:mu-totimes}).
\end{proof}

\begin{rmk}\label{rmk:ADE:h}
The condition that \eqref{eq:ADE:DS} vanishes uniquely determines the function $h_{ij}$ in \eqref{eq:U:hij}.
\end{rmk}

Next, we introduce the counit of $U$. Recall the topological $H$-algebra $D_H\dbr{p}$ explained in \cref{prp:E:Halgd} \ref{i:E:Halgd:ve}.

\begin{lem}\label{lem:ADE:ve}
Let $\ve\colon U \to D_H\dbr{p}$ be the homomorphism of topological algebra with
\begin{align}\label{eq:U:ve0}
 \ve(\mu_l(F)) = \ve(\mu_r(F)) \ceq F T_0, \quad 
 \ve(q^{\pm c/2}) \ceq 1, \quad \ve(d) \ceq 0
\end{align}
for $F=F(P,p) \in \bbF\dbr{p}$ and
\begin{align}\label{eq:U:ve1}
 \ve(K^\pm_i) \ceq T_{Q_i}, \quad 
 \ve(e_i(z))  = \ve(f_i(z)) \ceq 0, \quad \ve(\psi_i^{\pm}(z)) \ceq T_{Q_i}.
\end{align}
Then $\ve$ is a topological $H$-algebra homomorphism.
\end{lem}

\begin{proof}
We can check that $\ve$ is well-defined as a topological algebra morphism by direct calculations as in \cref{lem:U:Delta} (the calculation is much simpler). On the check of $\ve$ being a topological $H$-algebra homomorphism, it is clear that $\ve$ preserves the $H^*$-bigradings. 
The compatibility with the moment maps can be checked as 
\begin{align*}
&\ve(\mu_l(F)) = \ve(F(P,p^*)) = F(P,p) = F(P,p) T_0 = \mu_l^{D_H}(F), \\
&\ve(\mu_r(F)) = \ve(F(P+h,p)) = F(P,p) = F(P,p) T_0 = \mu_r^{D_H}(F)
\end{align*}
for $F=F(P,p) \in \bbF\dbr{p}$, where $\mu_l^{D_H}$ and $\mu_r^{D_H}$ are the moment maps of $D_H$.
\end{proof}

To state the $H$-bialgebroid structure, let us write down the natural isomorphisms $D_H \totimes U \cong U \cong U \totimes D_H$ in \eqref{eq:Halgd:Delta-ve}. By the argument of \eqref{eq:ADE:DU=U} and \eqref{eq:ADE:UD=U}, we have
\begin{align}\label{eq:U:DU=U}
 (D_H \totimes U)_{\alpha,\beta} = \bbF T_{-\alpha} \otimes_{\bbF} U_{\alpha,\beta} 
 \lsto U_{\alpha,\beta}, \quad 
 F T_{-\alpha} \totimes a  = T_{-\alpha} \totimes \mu_l(F) a \lmto \mu_l(F) a
\end{align}
and
\begin{align}\label{eq:U:UD=U}
 (U \totimes D_H)_{\alpha,\beta} = U_{\alpha,\beta} \otimes_{\bbF} \bbF T_{-\beta} 
 \lsto U_{\alpha,\beta}, \quad 
 a \totimes F T_{-\beta} = \mu_r(F) a \totimes T_{-\beta} \lmto a \mu_r(F)
\end{align}
for $a \in U_{\alpha,\beta}$ and $F \in \bbF\dbr{p}$.
Now we can show:

\begin{prp}\label{prp:ADE:Halgd}
$(U,\Delta,\ve)$ is a topological $H$-bialgebroid. 
That is, the following relations hold under the isomorphisms $(U \totimes U) \totimes U \cong U \totimes (U \totimes U)$ and $D_H \totimes U \cong U \cong U \totimes D_H$.
\begin{align*}
 (\Delta \totimes \id)\Delta = (\id \totimes \Delta)\Delta, \quad 
 (\ve \totimes \id)\Delta = \id = (\id \totimes \ve)\Delta.
\end{align*}
\end{prp}

\begin{proof}
We check the relations for each generator.
By direct calculation, we have 
\begin{align*}
& (\Delta \totimes 1)\Delta(e_i(z))
 =(\Delta \totimes 1)
  \bigl(e_i(z) \totimes 1 + \psi^+_i(q^{c_1/2}z) \totimes e_i(q^{c_1}z)\bigr) \\
&=e_i(z) \totimes 1 \totimes 1 + \psi^+_i(q^{\frac{c_1}{2}}z) \totimes e_i(q^{c_1}z) \totimes 1 
 +\psi^+_i(q^{\frac{c_1}{2}+\frac{c_2}{2}-\frac{c_2}{2}}z) \totimes 
  \psi^+_i(q^{\frac{c_1}{2}+\frac{c_2}{2}+\frac{c_1}{2}}z) \totimes e_i(q^{c_1+c_2}z) \\
&=e_i(z) \totimes 1 \totimes 1 + 
  \psi^+_i(q^{\frac{c_1}{2}}z) \totimes e_i(q^{c_1}z) \totimes 1 + 
  \psi^+_i(q^{\frac{c_1}{2}}z) \totimes 
  \psi^+_i(q^{c_1+\frac{c_2}{2}}z) \totimes e_i(q^{c_1+c_2}z), 
\\
& (1 \totimes \Delta)\Delta(e_i(z))
 =(1 \totimes \Delta)\bigl(e_i(z) \totimes 1 + 
  \psi^+_i(q^{c_1/2}z) \totimes e_i(q^{c_1}z)\bigr) \\
&=e_i(z) \totimes 1 \totimes 1 + 
  \psi^+_i(q^{\frac{c_1}{2}}z) \totimes e_i(q^{c_1}z) \totimes 1 + 
  \psi^+_i(q^{\frac{c_1}{2}}z) \totimes 
  \psi^+_i(q^{c_1+\frac{c_2}{2}}z) \totimes e_i(q^{c_1+c_2}z),
\end{align*}
which yields $(\Delta \totimes 1)\Delta(e_i(z))=(1 \totimes \Delta)\Delta(e_i(z))$. We also have 
\begin{align*}
  (\Delta \totimes 1)\Delta(\psi_i^+(z))
&=(\Delta \totimes 1)(\psi_i^+(zq^{\frac{c_2}{2}}) \totimes \psi_i^+(zq^{-\frac{c_1}{2}})) \\
&=\psi_i^+(q^{\frac{c_2}{2}+\frac{c_3}{2}}z) \totimes 
  \psi_i^+(\psi_i^+(q^{-\frac{c_1}{2}+\frac{c_3}{2}}z)) \totimes 
  \psi_i^+(q^{-\frac{c_1}{2}-\frac{c_2}{2}}z), \\
  (1 \totimes \Delta)\Delta(\psi_i^+(z))
&=(1 \totimes \Delta)(\psi_i^+(zq^{\frac{c_2}{2}}) \totimes \psi_i^+(zq^{-\frac{c_1}{2}})) \\
&=\psi_i^+(q^{\frac{c_2}{2}+\frac{c_3}{2}}z) \totimes 
  \psi_i^+(\psi_i^+(q^{-\frac{c_1}{2}+\frac{c_3}{2}}z)) \totimes 
  \psi_i^+(q^{-\frac{c_1}{2}-\frac{c_2}{2}}z), 
\end{align*}
which yields $(\Delta \totimes 1)\Delta(\psi^+_i(z))=(1 \totimes \Delta)\Delta(\psi^+_i(z))$. 
By similar calculations, we can check $(\Delta \totimes 1)\Delta=(1 \totimes \Delta)\Delta$ for the other generators. 

On the relations $(\ve \totimes \id)\Delta = \id = (\id \totimes \ve)\Delta$, we have 
\begin{align*}
  (\ve \totimes 1)\Delta(e_i(z))
&=(\ve \totimes 1)\bigl(e_i(z) \totimes 1 + 
   \psi^+_i(q^{c_1/2}z) \totimes e_i(q^{c_1}z)\bigr)
 = T_{Q_i} \totimes e_i(z) = e_i(z), \\
  (1 \totimes \ve)\Delta(e_i(z))
&=(1 \totimes \ve)\bigl(e_i(z) \totimes 1 + 
  \psi^+_i(q^{c_1/2}z) \totimes e_i(q^{c_1}z)\bigr) 
 =e_i(z) \totimes 1 = e_i(z); 
\\
  (\ve \totimes 1)\Delta(\psi_i^+(z))
&=(\ve \totimes 1)\bigl(\psi_i^+(q^{c_2/2}z) \totimes \psi_i^+(q^{-c_1/2}z)\bigr)
 =T_{Q_i} \totimes \psi_i^+(z) = \psi_i^+(z), \\
  (1 \totimes \ve)\Delta(\psi_i^+(z))
&=(1 \totimes \ve)\bigl(\psi_i^+(q^{c_2/2}z) \totimes \psi_i^+(q^{-c_1/2}z)\bigr) 
 =\psi_i^+(z) \totimes T_{Q_i} = \psi_i^+(z).
\end{align*}
Here we used \eqref{eq:U:DU=U} and \eqref{eq:U:UD=U}.
By similar calculations, we can check $(\ve \totimes \id)\Delta = \id = (\id \totimes \ve)\Delta$ for the other generators.
\end{proof}

Finally, we introduce the antipode of $U$.

\begin{lem}\label{lem:U:S}
We introduce a $\bbC$-linear map $S\colon U \to U$ by 
\begin{align}\label{eq:U:S}
 S(\mu_l(F)) \ceq \mu_r(F), \quad S(\mu_r(F)) \ceq \mu_l(F)
\end{align}
for $F=F(P,p) \in \bbF\dbr{p}$ and 
\begin{align}
&S(d) \ceq d, \quad S(q^{\pm c/2}) \ceq q^{\mp c/2}, \quad 
 S(K^\pm_i) \ceq (K^\pm_i)^{-1}, \quad S(\psi^\pm_i(z)) \ceq \psi^\pm_i(z)^{-1}, \\
\label{eq:U:S1}
&S(e_i(z)) \ceq -\psi_i^+(q^{-c/2}z)^{-1}e_i(q^{-c}z) , \quad 
 S(f_i(z)) \ceq - f_i(q^{-c}z)\psi_i^-(q^{-c/2}z)^{-1}.
\end{align}
Then $S$ is well-defined as an anti-homomorphism of topological algebras.
\end{lem}

\begin{proof}
We will only write down the check for a part of the defining relations of $U$. The block \eqref{eq:U:FK}--\eqref{eq:U:Fp} can checked as 
\begin{align*}
 S(K^\pm_i) S\bigl(\mu_r(F)\bigr)
&=(K^\pm_i)^{-1} F(P+h,p) 
 =F(P+h-\pair{Q_i,P+h},p) (K^\pm_i)^{-1} 
 =S\bigl(\mu_r(F')\bigr) S(K), 
\\
  S(e_i(z)) S(\mu_r(F))
&=-\psi_i^+(q^{-c/2}z)^{-1} e_i(q^{-c}z) F(P+h,p) \\ 
&=-\psi_i^+(q^{-c/2}z)^{-1} F(P+h, p) e_i(q^{-c}z) \\
&=-F(P+h-\pair{Q_i,P+h},p) \psi_i^+(q^{-c/2}z)^{-1} e_i(q^{-c}z)
 =S\bigl(\mu_r(F')\bigr) S(e_i(z))
\end{align*}
for $F=F(P,p), F'=F(P-\pair{Q_i,P},p) \in \bbF\dbr{p}$. 
The relations \eqref{eq:U:Ke} can be checked as
\begin{align*}
  S(\psi_j^+(w)) S(K^\pm_i)
&=\psi_j^+(w)^{-1} (K^\pm_i)^{-1} = (K^\pm_i)^{-1} \psi_j^+(w)^{-1} = S(K^\pm_i) S(\psi_j^+(w)).
\end{align*}
The block \eqref{eq:U:pp}--\eqref{eq:U:pme} can checked as
\begin{align*}
 S(e_j(w)) S(\psi_i^+(z))
= \psi_j^+(q^{-c/2}w)^{-1} e_j(q^{-c}w) \psi_i^+(z)^{-1}
=-\psi_j^+(q^{-c/2}w)^{-1} g^*_{ij}(q^{c/2}\tfrac{z}{w})
  \psi_i^+(z)^{-1} e_j(q^{-c}w) \\
=-\frac{g_{ij}(q^{c/2}\frac{z}{w})}{g^*_{ij}(q^{c/2}\frac{z}{w})}g^*_{ij}(q^{c/2}\tfrac{z}{w})
  \psi_i^+(z)^{-1}  \psi_j^+(q^{-c/2}w)^{-1} e_j(q^{-c}w)
=S\bigl(g_{ij}^*(q^{-c/2}\tfrac{z}{w})\bigr) S(\psi_i^+(z)) S(e_j(w)).
\end{align*}
As for the relation \eqref{eq:U:ef}, we have 
\begin{align*}
& [S(f_i(w)), S(e_i(z))]
 =[-f_i(q^{-c}w)\psi_i^-(q^{-c/2}w)^{-1}, - \psi_i^+(q^{-c/2}z)^{-1}e_i(q^{-c}z)] \\
&=f_i(q^{-c}w)\psi_i^-(q^{-c/2}w)^{-1}\psi_i^+(q^{-c/2}z)^{-1}e_i(q^{-c}z)
 -\psi_i^+(q^{-c/2}z)^{-1}e_i(q^{-c}z)f_i(q^{-c}w)\psi_i^-(q^{-c/2}w)^{-1},
\end{align*}
and the first term in the last line can be calculated as
\begin{align*}
&\text{(1st term)}
 =\frac{g_{ii}(q^{c}\frac{z}{w})}{g^*_{ii}(q^{-c}\frac{z}{w})} f_i(q^{-c}w) 
  \psi_i^+(q^{-c/2}z)^{-1} \psi_i^-(q^{-c/2}w)^{-1} e_i(q^{-c}z) \\
&=\frac{g_{ii}(q^{c}\frac{z}{w})}{g^*_{ii}(q^{-c}\frac{z}{w})} g_{ii}(q^c\tfrac{z}{w})^{-1}
   \psi_i^+(q^{-c/2}z)^{-1} f_i(q^{-c}w) \psi_i^-(q^{-c/2}w)^{-1} e_i(q^{-c}z) \\
&=g_{ii}^*(q^{-c}\tfrac{z}{w})^{-1} g_{ii}^*(q^{-c}\tfrac{z}{w}) 
  \psi_i^+(q^{-c/2}z)^{-1} f_i(q^{-c}w) e_i(q^{-c}z) \psi_i^-(q^{-c/2}w)^{-1} \\
&=\text{(2nd term)} - \tfrac{1}{q-q^{-1}} \psi_i^+(q^{-c/2}z)^{-1}
  \bigl(\delta(q^{-c}\tfrac{z}{w}) \psi_i^-(q^{-c/2}w)
       -\delta(q^{ c}\tfrac{z}{w}) \psi_i^+(q^{-c/2}z)\bigr) 
  \psi_i^-(q^{-c/2}w)^{-1} \\
&=\text{(2nd term)} - \tfrac{1}{q-q^{-1}}
  \bigl(\delta(q^{-c}\tfrac{z}{w}) \psi_i^+(q^{-c/2}z)^{-1}
       -\delta(q^{ c}\tfrac{z}{w}) \psi_i^-(q^{-c/2}w)^{-1}\bigr).
\end{align*}
Thus we have
\begin{align*}
  [S(f_i(w)), S(e_i(z))]
&=\tfrac{1}{q-q^{-1}}
  \bigl(\delta(q^{ c}\tfrac{z}{w}) \psi_i^-(q^{-c/2}w)^{-1}
       -\delta(q^{-c}\tfrac{z}{w}) \psi_i^+(q^{-c/2}z)^{-1}\bigr) \\
&=\tfrac{1}{q-q^{-1}}
  \bigl(S(\delta(q^{-c}\tfrac{z}{w})) S(\psi_i^-(q^{c/2}w))
       -S(\delta(q^{ c}\tfrac{z}{w})) S(\psi_i^+(q^{c/2}z))\bigr).
\end{align*}
We can also check the Serre relations as follows.
Using \eqref{eq:U:pp} and \eqref{eq:U:ppe}, we have 
\begin{align*}
&S(e_j(z)) S(e_i(z_2)) S(e_i(z_1))\\
&=-\psi^+_j(q^{-c/2}z)^{-1} e_j(q^{-c}z) \psi^+_i(q^{-c/2}z_2)^{-1}
   e_i(q^{-c}z_2) \psi^+_j(q^{-c/2}z_1)^{-1} e_j(q^{-c}z_1) \\
&=-g^*_{ij}\bigl(\tfrac{z_1}{z  }\bigr) g^*_{ij}\bigl(\tfrac{z_2}{z}\bigr)
   g^*_{ii}\bigl(\tfrac{z_1}{z_2}\bigr) 
   \psi^+_j(q^{-c/2}z)^{-1} \psi^+_i(q^{-c/2}z_2)^{-1} \psi^+_j(q^{-c/2}z_1)^{-1}
   e_j(q^{-c}z) e_i(q^{-c}z_2) e_i(q^{-c}z_1) \\
&=-g^*_{ij}\bigl(\tfrac{z_1}{z  }\bigr) g^*_{ij}\bigl(\tfrac{z_2}{z  }\bigr)
   g^*_{ii}\bigl(\tfrac{z_1}{z_2}\bigr) g^*_{ji}\bigl(\tfrac{z  }{z_1}\bigr)
   g^*_{ji}\bigl(\tfrac{z  }{z_2}\bigr) g^*_{ii}\bigl(\tfrac{z_2}{z_1}\bigr), \\
&\qquad \cdot 
   \psi^+_j(q^{-c/2}z)^{-1} \psi^+_i(q^{-c/2}z_2)^{-1} \psi^+_j(q^{-c/2}z_1)^{-1}
   e_j(q^{-c}z) e_i(q^{-c}z_1) e_i(q^{-c}z_2) e_j(q^{-c}z) \\
&=-\psi^+_j(q^{-c/2}z)^{-1} \psi^+_i(q^{-c/2}z_2)^{-1} \psi^+_j(q^{-c/2}z_1)^{-1}
   e_j(q^{-c}z) e_i(q^{-c}z_1) e_i(q^{-c}z_2) e_j(q^{-c}z), 
\end{align*}
and a similar calculation yields
\begin{align*}
&S(e_i(z_2)) S(e_j(z)) S(e_i(z_1)) \\
&=-\frac{g^*_{ij}(\tfrac{z_2}{z})}{g_{ij}(\tfrac{z_2}{z})}
   \psi^+_j(q^{-c/2}z)^{-1} \psi^+_i(q^{-c/2}z_2)^{-1} \psi^+_j(q^{-c/2}z_1)^{-1}
   e_i(q^{-c}z_1) e_j(q^{-c}z) e_i(q^{-c}z_2), 
\\
&S(e_i(z_2)) S(e_i(z_1)) S(e_j(z)) \\
&=-\frac{g^*_{ij}(\tfrac{z_1}{z})}{g_{ij}(\tfrac{z_1}{z})} 
   \frac{g^*_{ij}(\tfrac{z_2}{z})}{g_{ij}(\tfrac{z_2}{z})}
   \psi^+_j(q^{-c/2}z)^{-1} \psi^+_i(q^{-c/2}z_2)^{-1} \psi^+_j(q^{-c/2}z_1)^{-1}
   e_j(q^{-c}z) e_i(q^{-c}z_1) e_i(q^{-c}z_2),
\\
&S(e_j(z)) S(e_i(z_1)) S(e_i(z_2)) \\
&=-\frac{g^*_{ii}(\tfrac{z_1}{z_2})}{g_{ii}(\tfrac{z_1}{z_2})}
   \psi^+_j(q^{-c/2}z)^{-1} \psi^+_i(q^{-c/2}z_2)^{-1} \psi^+_j(q^{-c/2}z_1)^{-1}
   e_i(q^{-c}z_2) e_i(q^{-c}z_1) e_j(q^{-c}z), 
\\
&S(e_i(z_1)) S(e_j(z)) S(e_i(z_2)) \\
&=-\frac{g^*_{ij}(\tfrac{z_1}{z  })}{g_{ij}(\tfrac{z_1}{z  })}
   \frac{g^*_{ii}(\tfrac{z_1}{z_2})}{g_{ii}(\tfrac{z_1}{z_2})}
   \psi^+_j(q^{-c/2}z)^{-1} \psi^+_i(q^{-c/2}z_2)^{-1} \psi^+_j(q^{-c/2}z_1)^{-1}
   e_i(q^{-c}z_2) e_j(q^{-c}z) e_i(q^{-c}z_1), 
\\
&S(e_i(z_1)) S(e_i(z_2)) S(e_j(z)) \\
&=-\frac{g^*_{ij}(\tfrac{z_1}{z  })}{g_{ij}(\tfrac{z_1}{z  })}
   \frac{g^*_{ij}(\tfrac{z_2}{z  })}{g_{ij}(\tfrac{z_2}{z  })}
   \frac{g^*_{ii}(\tfrac{z_1}{z_2})}{g_{ii}(\tfrac{z_1}{z_2})}
   \psi^+_j(q^{-c/2}z)^{-1} \psi^+_i(q^{-c/2}z_2)^{-1} \psi^+_j(q^{-c/2}z_1)^{-1}
   e_j(q^{-c}z) e_i(q^{-c}z_2) e_i(q^{-c}z_1).
\end{align*}
Comparing with \eqref{eq:U:eS}, we conclude that the antipode $S$ is compatible with the Serre relation for $e_i(z)$. Similarly we can check the Serre relation \eqref{eq:U:fS} for $f_i(z)$.
\end{proof}

Now we have the main statement.

\begin{thm}\label{thm:ADE:Halgd}
The tuple $(\Delta,\ve,S)$ is a topological $H$-Hopf algebroid structure on the topological $H$-algebra $U=U_{q,p}(g,X_l)$. That is, the relations
\begin{align*}
 S(\mu_r(F)a) = S(a) \mu_l(F), \quad 
&S(a\mu_l(F)) = \mu_r(F)S(a), \\
 m(\id_U \totimes S)\Delta(a) = \mu_l(\ve(a)1), \quad 
&m(S \totimes \id_U)\Delta(a) = \mu_r\bigl(T_{\alpha}(\ve(a)1)\bigr)
\end{align*}
hold for any $a \in U_{\alpha,\beta}$ and $F \in \bbF\dbr{p}$, 
where $m$ is the multiplication of the topological algebra $U$. 
\end{thm}

\begin{proof}
The first line of the statement is built in the definition \eqref{eq:U:S} of $S$.
The second line can be checked on the generators as 
\begin{align*}
  m(\id \totimes S)\Delta(\mu_l(F)) 
&=m(\id \totimes S)(\mu_l(F) \totimes 1) = \mu_l(F), \quad
 \mu_l(\ve(\mu_l(F)) 1) = \mu_l(F T_0 1)=\mu_l(F); 
\\
  m(\id \totimes S)\Delta(e_i(z))
&=m(\id \totimes S)\bigl(e_i(z) \totimes 1 + 
   \psi^+_i(q^{c_1/2}z) \totimes e_i(q^{c_1}z)\bigr)\\
&=e_i(z)-\psi^+_i(q^{c/2}z)\psi^+_i(q^{c/2}z)^{-1}e_i(z) = 0, \\
  \mu_l(\ve(e_i(z)) 1) &= \mu_l(0) = 0; 
\\
  m(\id \totimes S)\Delta(\psi_i^+(z))
&=m(\id \totimes S)\psi_i^+(zq^{-c_2/2})\totimes \psi_i^+(zq^{c_1/2})
 =\psi_i^+(zq^{c/2})\cdot \psi_i^+(zq^{c/2})^{-1} = 1, \\
  \mu_l(\ve(\psi_i^+(z)) 1) 
&=\mu_l(T_{Q_i} 1) = \mu_l(1) = 1. 
\end{align*}
By a similar calculation on the other generators,
we get $m(1 \totimes S)\Delta(a) = \mu_l(\ve(a)\cdot 1)$. 
Similarly, we have 
\begin{align*}
  m(S \totimes \id)\Delta(\mu_l(F)) 
&=m(S \totimes \id)(\mu_l(F) \totimes 1) = \mu_r(F), \\
 \mu_r\bigl(T_0(\ve(\mu_l(F)) 1)\bigr) 
&= \mu_r\bigl(T_0(F)\bigr)=\mu_r(F); 
\\
  m(S \totimes \id)\Delta(e_i(z))
&=m(S \totimes \id)
 \bigl(e_i(z) \totimes 1 + 
   \psi^+_i(q^{c_1/2}z) \totimes e_i(q^{c_1}z)\bigr) \\
&=- \psi^+_i(q^{-c/2}z)^{-1} e_i(q^{-c}z) \cdot 1 - \psi^+_i(q^{-c/2}z)^{-1}e_i(q^{-c}z)  =0, \\
  \mu_r\bigl(T_{-Q_i}(\ve(e_i(z)) 1)\bigr) &=\mu_r(T_{-Q_i}(0)) = \mu_r(0) = 0;
\\
  m(S \totimes \id)\Delta(\psi_i^+(z))
&=m(S \totimes \id)
  \bigl(\psi_i^+(z q^{-c_2/2}) \totimes \psi_i^+(z q^{c_1/2})\bigr) 
 =\psi_i^+(q^{-c/2}z)^{-1} \cdot \psi_i^+(q^{-c/2}z) = 1, \\
  \mu_r\bigl(T_{-Q_i}(\ve(\psi_i^+(z)) 1)\bigr)
&=\mu_r\bigl(T_{-Q_i}(T_{Q_i} 1)\bigr)
 =\mu_r(T_{-Q_i}(1)) = \mu_r(1) = 1. 
\end{align*}
The other generators can be checked similarly, and we have $m(S \totimes \id)\Delta(a)=\mu_r\bigl(T_{\alpha}(\ve(a) 1)\bigr)$ for any $a \in U_{\alpha, \beta}$.
\end{proof}

\begin{dfn}\label{dfn:ADE:U}
The obtained $H$-Hopf algebroid $U_{q,p}(g,X_l)$ is called \emph{the dynamical Ding-Iohara algebroid of type $X_l$} for $X=ADE$.
\end{dfn}

We close this subsection by the relation of $U_{q,p}(g,X_l)$ and the elliptic algebra $U_{q,p}(\wh{\frg}(X_l))$. By \cref{prp:U:ADE}, they are isomorphic as topological algebras. By comparing the $H$-Hopf algebroid structures, we have:

\begin{cor}\label{cor:U:ADE}
Let $A=A(X_l)$ be the Cartan matrix of type $X_l=A_l,D_l$ or $E_{6,7,8}$.
In \cref{dfn:U:ADE}, set the structure functions $g=\{g_{ij}(z;p) \mid i,j \in I\}$ to be
\[
 g_{ij}(x;p) = \frac{G_{ij}^+(x;p)}{G_{ij}^-(x;p)} \ceq 
 \frac{q^{-b_{ij}} \theta(q^{b_{ij}}x;p)}{\theta(q^{-b_{ij}}x;p)},
\]
where $B=(b_{ij})_{i,j \in I}$ is the symmetrization $A=D B$.
Then we have an isomorphism of topological $H$-Hopf algebroids
\[
 U_{q,p}(g,X_l) \cong U_{q,p}(\wh{\frg}(X_l)), 
\]
mapping each generator in the dynamical Ding-Iohara algebroid $U_{q,p}(g,X_l)$ to the corresponding one in the elliptic algebra $U_{q,p}(\wh{\frg}(X_l))$.
\end{cor}

\subsection{Hopf algebra limit}\label{ss:ADE:p=0}

Here we discuss the limit $p \to 0$.

\begin{prp}\label{prp:ADE:p=0}
Let $X_l=A_l,D_l$ or $E_{6,7,8}$, and denote by $U''_{q,p}(g,X_l) \subset U_{q,p}(g,X_l)$ the topological $H$-Hopf sub-algebroid of the dynamical Ding-Iohara algebroid generated by 
\[
 q^{\pm c/2}, \ e_{i,m}, \ f_{i,m}, \ \psi^{\pm}_{i,m} \quad (i \in I, \, m  \in \bbZ)
\]
i.e., generators in \eqref{eq:ADE:gen} except $\clM(\wtH^*)$, $d$, and $K_i^{\pm}$.
Then, taking the limit $p \to 0$ of $U''_{q,p}(g,X_l)$, we have a topological Hopf algebra isomorphism 
\[
 \lim_{p \to 0} U''_{q,p}(g,X_l) \cong U_q(\ol{g},X_l), 
\]
where $\ol{g} \ceq \{\ol{g}_{ij}(z) \mid i,j \in I\}$ denotes the set of the limit structure functions in \eqref{eq:U:olg}, and $U_q(\ol{g},X_l)$ denotes the Ding-Iohara quantum algebra of the simply-laced type $X_l$ with structure functions $\ol{g}$ (see \cref{dfn:DI:U}, \cref{fct:DI:U} and \cref{rmk:DI:ADE})
\end{prp}

\begin{proof}
For the algebra structure, compare the defining relations \eqref{eq:DI:qc}--\eqref{eq:DI:Serre} and \eqref{eq:U:qc}--\eqref{eq:U:fS}, noting that the functions $\wt{g}_{ij}(z)$ and $\wt{g}^*_{ij}(z)$ in the Serre-type relations \eqref{eq:U:eS} and \eqref{eq:U:fS} reduce to $1$ in the limit $p \to 0$. For the Hopf algebra structure, note that the modified tensor product $\totimes$ reduces to the ordinary tensor product $\otimes$ in the limit $p \to 0$, and compare the formulas \eqref{eq:DI:Delta0}--\eqref{eq:DI:S} and \eqref{eq:U:Delta0}--\eqref{eq:U:Delta1}, \eqref{eq:U:ve0}, \eqref{eq:U:ve1}, \eqref{eq:U:S}--\eqref{eq:U:S1}.
\end{proof}

\begin{rmk}
The limit Hopf algebra $U_q(\ol{g},X_l)$ is independent of the parameter $p$. To keep a parameter, say $p'$, in the limit, then use \cref{rmk:ADE:top} \ref{i:rmk:ADE:top:p'} for the class of the structure functions $g$.
\end{rmk}

Thus we reached the goal of this note for simply-laced root systems.

\section{Dynamical Ding-Iohara algebroids of non-simply-laced root systems}\label{s:ns}

In this section, we extend the dynamical Ding-Iohara algebroid associated to the simply-laced simple root systems in \cref{dfn:ADE:U} to the non-simply-laced systems.

\subsection{Strange formulas}\label{ss:ns:h}

Let us start with the recollection of \cref{prp:U:ADE}, where we recovered the elliptic algebra $U_{q,p}(\wh{\frg}(X_l))$ from our dynamical Ding-Iohara algebroid $U_{q,p}(g,X_l)$ of simply-laced root systems $X_l$ by setting the structure functions as \eqref{eq:U:gijE}. There we used a strange equality \eqref{eq:ADE:h=2q}, which claims that the term
\[
 h_{ij} \ceq 
 \frac{(\ol{g}_{ii}(\frac{z_1}{z_2})+1)(\ol{g}_{ij}(\frac{z_1}{z})\ol{g}_{ij}(\frac{z_2}{z})+1)}
      { \ol{g}_{ij}(\frac{z_2}{z})+\ol{g}_{ii}(\frac{z_1}{z_2})\ol{g}_{ij}(\frac{z_1}{z})}.
\] 
in the Serre-type relations \eqref{eq:U:eS} and \eqref{eq:U:fS} of $U_{q,p}(g,X_l)$ reduces under the setting
\begin{align}\label{eq:ns:olg-sp}
 \ol{g}_{ii}(z) = \frac{1-q^2 z}{q^2-z}, \quad 
 \ol{g}_{ij}(z) = \frac{q-    z}{1-q z}. 
\end{align}
to the $q$-integers $[2]_i$ (see \eqref{eq:E:qi}) in those \eqref{eq:E:eS} and \eqref{eq:E:fS} of $U_{q,p}(\wh{\frg}(X_l))$. As mentioned in \cref{rmk:ADE:h}, we found the term $h_{ij}$ by the compatibility with the Drinfeld-type comultiplication $\Delta$ in \eqref{lem:U:Delta}.

For the discussion of non-simply-laced types, let us give a slight modification of the equality \eqref{eq:ADE:h=2q}.

\begin{lem}\label{lem:ns:sr2}
For a parameter $q^d$ and variables $z_{12} \ceq z_1/z_2$, $z_{10} \ceq z_1/z$ and $z_{20} \ceq z_2/z$ and the rational functions 
\begin{align}\label{eq:ns:ADEolg}
 g^1_2 \ceq \frac{1-q^{2d} z_{12}}{q^{2d}-z_{12}}, \quad  
 g^m_0 \ceq \frac{q^d-z_{m0}}{1-q^d z_{m0}} \quad (m=1,2),
\end{align}
the following equality holds 
\begin{align}\label{eq:ns:2q}
 [2]_{q^d}  \bigl(\ceq \tfrac{q^d-q^{-d}}{q-q^{-1}}\bigr) = 
 \frac{(1+g^1_2)+(g^1_0 g^2_0+g^1_2 g^1_0 g^2_0)}{g^2_0+g^1_2 g^1_0}. 
\end{align}
\end{lem}

\begin{proof}
Direct calculation. 
\end{proof}

Let us give some observations of the formula \eqref{eq:ns:2q}.
\begin{itemize}
\item 
The variables $z_1,z_2$ and $z$ come from the Serre-type relations \eqref{eq:E:eS} and \eqref{eq:E:fS} of the elliptic algebra $U_{q,p}(\wh{\frg})$, and the number $2=\abs{\{1,2\}}$ corresponds to the quantity $a=1-a_{ij}=2$ determined by the off-diagonal entry $a_{ij}=-1$ of the Cartan matrix $A=A(X_l)$ (see the line before \eqref{eq:E:wtgam}).

\item
Set $z_{21} \ceq z_2/z_1$ and $g^2_1 \ceq (1-q^{2d}z_{21})/(q^{2d}-z_{21})$. Then we have 
$g^2_1 = (g^1_2)^{-1}$, which is a rewritten form of the Ding-Iohara condition \eqref{eq:DI:cond} or \eqref{eq:U:olg-cond}.

\item
Let us focus on the invariance of \eqref{eq:ns:2q} under the permutation of variables $z_1$ and $z_2$. Then, the denominator and the numerator of the right hand side of \eqref{eq:ns:2q} can be regarded as the $\frS_2$-orbit summations in the following sense. Consider the set of rational functions
\begin{align}\label{eq:ns:R2}
 R_2 \ceq \{(g^1_2)^{p_{12}} (g^1_0)^{p_1} (g^2_0)^{p_2} \mid p_{12},p_1,p_2 = 0,1\}.
\end{align}
The symmetric group $\frS_2$ acts on $R_2$ by letting the generator $s_1 \ceq (1,2) \in \frS_2$ act as 
\begin{align}\label{eq:ns:s1-R2}
 s_1\bigl((g^1_2)^{p_{12}} (g^1_0)^{p_1} (g^2_0)^{p_2}\bigr) \ceq 
 g^1_2 (g^2_1)^{p_{12}} (g^2_0)^{p_1} (g^1_0)^{p_2} = 
 (g^1_2)^{1-p_{12}} (g^1_0)^{p_2} (g^2_0)^{p_1}.
\end{align}
In the second equality, we used $g^2_1 = (g^1_2)^{-1}$. By the condition on the power indices $p_{12},p_1,p_2$, we find that this action is well-defined. For $f \in R_2$, we denote its orbit sum by $\wh{f} \ceq \sum_{\sigma \in \frS_2} \sigma(f)$. Then, we have
\[
 \wh{1} = 1+g^1_2, \quad \wh{g^1_0 g^2_0} = g^1_0 g^2_0 + g^1_2 g^1_0 g^2_0, \quad 
 \wh{g^2_0} = g^2_0 + g^1_2 g^1_0,
\]
so that \eqref{eq:ns:2q} can be rewritten as 
\begin{align}\label{eq:ns:hij2}
 [2]_{q^d} = \frac{\wh{1}+\wh{g^1_0 g^2_0}}{\wh{g^2_0}}.
\end{align}
\end{itemize}

Now we turn to the non-simply-laced type $X_l$, where the $q$-binomials $\bnm{3}{s}{i}$ and $\bnm{4}{s}{i}$ appear in the Serre-type relations \eqref{eq:E:eS} and \eqref{eq:E:fS} of the elliptic algebra $U_{q,p}\bigl(\wh{\frg}(X_l)\bigr)$. We want to find a correct term $h_{ij}$ which yields Serre-type relations of the would-be Ding-Iohara algebroids of types $B,C,F$ under the following three principles.
\begin{clist}
\item 
The compatibility with the Drinfeld-type comultiplication $\Delta$ in \eqref{lem:U:Delta}.
\item
The invariance under the permutation of variables $z_i$'s.
\item
Recovering the $q$-binomials under the specialization \eqref{eq:ns:olg-sp}.
\end{clist}
In this \cref{ss:ns:h}, we give the consequence and check the second and third principle. The compatibility of these functions with the Drinfeld-type comultiplication will be shown in the following subsections. 

First, we consider the $q$-integer $[3]_{i}$ appearing in the types $B,C$ and $F$. Following the above observations, let us consider the variables
\begin{align}\label{eq:ns:BCF-z}
 z_{mn} \ceq z_m/z_n \quad (1\le m < n \le 3), \quad z_{m0} \ceq z_m/z \quad (1 \le m \le 3)
\end{align}
and the rational functions
\begin{align}\label{eq:ns:BCF-sp}
 g^m_n \ceq \frac{1-q^{2d} z_{mn}}{q^{2d}-z_{mn}} \quad (1 \le m < n \le 3), \quad 
 g^m_0 \ceq \frac{q^{2d}-z_{m0}}{1-q^{2d} z_{m0}} \quad (1 \le m \le 3). 
\end{align}
Then, following \eqref{eq:ns:R2}, we define a set $R_3$ of rational functions by 
\begin{align}\label{eq:ns:R3}
 R_3 \ceq \bigl\{\tprd_{1 \le m < n \le 3}(g^m_n)^{p_{mn}} \tprd_{1 \le m \le 3}(g^m_0)^{p_m}
                 \mid p_{mn},p_m = 0,1 \bigr\}. 
\end{align}
We also define an action of the symmetric group $\frS_3$ on $R_3$ by letting the generators $s_1 \ceq (1,2)$ and $s_2 \ceq (2,3)$ act as 
\begin{align}
\label{eq:ns:s1-R3}
&s_1\Bigl(\prod_{1 \le m < n \le 3}(g^m_n)^{p_{mn}} \prod_{1 \le m \le 3}(g^m_0)^{p_m}\Bigr)
\ceq g^1_2\prod_{1 \le m < n \le 3}\bigl(g^{s_1(m)}_{s_1(n)}\bigr)^{p_{mn}}
          \prod_{1 \le m     \le 3}\bigl(g^{s_1(m)}_{0}\bigr)^{p_{m}}, \\
\label{eq:ns:s2-R3}
&s_2\Bigl(\prod_{1 \le m < n \le 3}(g^m_n)^{p_{mn}} \prod_{1 \le m \le 3}(g^m_0)^{p_m}\Bigr)
\ceq g^2_3\prod_{1 \le m < n \le 3}\bigl(g^{s_2(m)}_{s_2(n)}\bigr)^{p_{mn}}
          \prod_{1 \le m     \le 3}\bigl(g^{s_2(m)}_{0     }\bigr)^{p_m}.
\end{align}
We can check $s_1(R_3) \subset F$ using $s_1(g^m_n),s_1(g^m_0)\in R_3$ except for $g^1_2$ and $g^n_m=(g^m_n)^{-1}$. We can similarly show $s_2(R_3) \subset R_3$, and hence the $\frS_3$-action on $R_3$ is well-defined. Note that this well-definedness only uses $g^n_m=(g^m_n)^{-1}$, which is a rewritten form of the Ding-Iohara condition \eqref{eq:DI:cond} or \eqref{eq:U:olg-cond}.

We seek an analogue of \eqref{eq:ns:hij2} for the $q$-integer $[3]_{q^d}$, i.e., want to express $[3]_{q^d}$ as a ratio of the $\frS_3$-orbit summations in $R_3$. We abbreviate the product of the functions by $g^{11}_{23} \ceq g^1_2 g^1_3$ and so on, and denote the orbit sum of $f \in R_3$ by $\wh{f} \ceq \sum_{\sigma \in \frS_3}\sigma(f)$. For example, we have 
\[
 \wh{1} \ceq \sum_{\sigma \in \frS_3} \sigma(1) = 
 1+g^1_2+g^2_3+g^{11}_{23}+g^{12}_{33}+g^{112}_{233}. 
\]
Under these notation, we can obtain the following formula by tedious but direct calculation. 
\begin{align}\label{eq:ns:hij3}
 [3]_{q^d} \bigl(\ceq \tfrac{q^{3d}-q^{-3d}}{q^d-q^{-d}}\bigr) =
 \frac{\wh{g^{123}_{000}}-\wh{1}}{\wh{g^{23}_{00}}-\wh{g^3_0}}
\end{align}
Thus, we found the desired term $h_{ij}$ for the types $B,C$ and $F$. For later use, we rename it as $h^3_{ij}$, and summarize the result in:

\begin{cor}\label{cor:ns:BCF}
Let $X_l=B_l,C_l$ or $F_4$, and $A=A(X_l)=(a_{rs})_{r,s \in I}$ be the Cartan matrix in \cref{sss:E:root}. Let $(i,j) \in I^2 = \{1,\dotsc,l\}^2$ be the (unique) pair such that the entry $a_{ij}$ is equal to $-2$. Also, let $g=\{g_{rs}(z;p) \mid r,s \in I\}$ be the structure functions (\cref{dfn:U:DI}), and $\ol{g}_{mn}(z)$ be the limit $\lim_{p \to 0}g_{ij}(z;p)$.
\begin{enumerate}
\item 
For the variables \eqref{eq:ns:BCF-z}, we set the rational functions 
\[
 g^m_n \ceq \ol{g}_{ii}(z_{mn}) \quad (1 \le m < n \le 3), \quad 
 g^m_0 \ceq \ol{g}_{ij}(z_{m0}) \quad (1 \le m \le 3). 
\]
Then \eqref{eq:ns:s1-R3} and \eqref{eq:ns:s2-R3} define an $\frS_3$-action on the set $R_3$ given by \eqref{eq:ns:R3}.

\item
Using the abbreviation of the product $g^{11}_{23} \ceq g^1_2 g^1_3$ and so on, and of the orbit sum $\wh{f} \ceq \sum_{\sigma \in \frS_3}\sigma(f)$, we define
\begin{align}\label{eq:ns:h3}
 h^3_{ij} \ceq \frac{\wh{g^{123}_{000}}-\wh{1}}{\wh{g^{23}_{00}}-\wh{g^3_0}}.
\end{align}
We also specialize the structure functions as 
\[
 g_{rs}(z;p) = q^{-b_{rs}} \frac{\theta(q^{b_{rs}}z;p)}{\theta(q^{-b_{rs}}z;p)} 
 \quad (r,s \in I)
\]
(see \eqref{eq:U:gijE}) with the following symmetrization $A=D B$, $B=(b_{rs})_{r,s \in I}$.
\begin{clist}
\item 
For the type $B_l$, we take $D(B_l)^{-1} \ceq (d_1,\dotsc,d_l)=(1,\dotsc,1,\thf)$, 
so that we have
\begin{align*}
 B(B_l) = \begin{psm} 2 & {-1} \\ {-1} & 2 & {-1} \\ & \ddots & \ddots & \ddots \\ 
                     &  & {-1} & 2 & {-1} \\ & & & {-1} & 1 \end{psm}.
\end{align*}
We also have $(i,j)=(l,l-1)$, $q_i=q^{d_i}=q^{1/2}$ (recall \eqref{eq:E:qi}) and 
\begin{align}\label{eq:B:olg}
\begin{split}
&\ol{g}_{rr}(z) = \frac{1-q^{2} z}{q^{2}-z} \quad (1 \le r < l=i),\quad  
 \ol{g}_{ii}(z) = \frac{1-q^{ } z}{q^{ }-z},\\
&\ol{g}_{rs}(z) = \frac{  q^{ }-z}{1  -q z} \quad (1 \le r,s \le l, \, \abs{r-s}=1).
\end{split}
\end{align}
Then, using the notation \eqref{eq:E:bini}, we have 
\begin{align}\label{eq:B:hij3}
 h^3_{ij} = q+1+q^{-1} = \frac{q^{3/2}-q^{-3/2}}{q^{1/2}-q^{-1/2}} = 
\bnm{3}{1}{i} = \bnm{3}{2}{i}.
\end{align}
Also, for $(r,s) \in I^2$ such that $\abs{r-s}=1$ and $(r,s) \ne (i,j)$, we have 
\begin{align}\label{eq:B:hij2}
  h_{rs} \ceq \frac{\wh{1}+\wh{g^{12}_{00}}}{\wh{g^2_0}} = \bnm{2}{1}{r},
\end{align}
where $q_r=q$ and $g^m_n$ is given by \eqref{eq:ns:ADEolg} with $d \ceq d_r=1$.

\item
For the type $C_l$, we take $D(C_l)^{-1}=(d_1,\dotsc,d_l)=(1,\dotsc,1,2)$, so that
\begin{align*}
 B(C_l) = \begin{psm} 2 & {-1} \\ {-1} & 2 & {-1} \\ & \ddots & \ddots & \ddots \\ 
                     &  & {-1} & 2 & {-2} \\ & & & {-2} & 4 \end{psm}, 
\end{align*}
$(i,j)=(l-1,l)$ and 
\begin{align}\label{eq:C:olg}
\begin{aligned}
&\ol{g}_{rr}(z) = \frac{1-q^2 z}{q^2-z} \quad (1 \le r <l=j), & 
&\ol{g}_{jj}(z) = \frac{1-q^4 z}{q^4-z},\\
&\ol{g}_{rs}(z) = \frac{q-z}{1-q z} \quad (1 \le r,s<l, \ \abs{r-s}=1), & 
&\ol{g}_{ij}(z) = \ol{g}_{ji}(z) = \frac{q^2-z}{1-q^2 z}.
\end{aligned}
\end{align}
Then we have  $q_i=q$ and
\begin{align}\label{eq:C:hij3}
 h^3_{ij} = \frac{q^3-q^{-3}}{q-q^{-1}} = \bnm{3}{1}{i} = \bnm{3}{2}{i}.
\end{align}
For $(r,s) \in I^2$ such that $\abs{r-s}=1$ and $(r,s) \ne (i,j)$, we have 
\begin{align}\label{eq:C:hij2}
 h_{rs} \ceq \frac{\wh{1}+\wh{g^{12}_{00}}}{\wh{g^2_0}} = \bnm{2}{1}{r}, 
\end{align}
where $q_r=q$ ($r<l-1$) or $q_r=q^2$ ($r=l$), and $g^m_n$ is given by \eqref{eq:ns:ADEolg} with $d=d_r$.

\item
For the type $F_4$, we take $D(F_4)^{-1}=(d_1,\dotsc,d_4)=(2,2,1,1)$, so that 
\begin{align*}
 B(F_4) = \begin{psm} 4 & {-2} \\ {-2} & 4 & {-2} \\  
                        & {-2} & 2 & {-1} \\ & & {-1} & 2 \end{psm}, 
\end{align*}
$(i,j)=(3,2)$ and 
\begin{align}\label{eq:F:olg}
\begin{aligned}
&\ol{g}_{11}(z) = \ol{g}_{22}(z) = \frac{1-q^4 z}{q^4-z}, &
&\ol{g}_{33}(z) = \ol{g}_{44}(z) = \frac{1-q^2 z}{q^2-z}, \\
&\ol{g}_{rs}(z) = \frac{q^2-z}{1-q^2 z} \quad (\{r,s\}=\{1,2\},\{2,3\}), & 
&\ol{g}_{rs}(z) = \frac{q-z}{1-qz} \quad (\{r,s\}=\{3,4\}).
\end{aligned}
\end{align}
Then we have $q_i=q$ and 
\begin{align}\label{eq:F:hij3}
 h^3_{ij} = \frac{q^3-q^{-3}}{q-q^{-1}} = \bnm{3}{1}{i} = \bnm{3}{2}{i}.
\end{align}
For $(r,s) \in I^2$ such that $\abs{r-s}=1$ and $(r,s) \ne (i,j)$, we have
\begin{align}\label{eq:F:hij2}
 h_{rs} \ceq \frac{\wh{1}+\wh{g^{12}_{00}}}{\wh{g^2_0}} = \bnm{2}{1}{r}, 
\end{align}
where $q_r=q^2$ ($r=1,2$) or $q_r=q$ ($r=3,4$), and $g^m_n$ is given by \eqref{eq:ns:ADEolg} with $d=d_r$.
\end{clist}
\end{enumerate}
\end{cor}

\begin{proof}
For the type $B_l$, the formula \eqref{eq:B:olg} of $h^3_{ij}$ is obtained from \eqref{eq:ns:hij3} since $\ol{g}_{ii}$ and $\ol{g}_{ij}$ in \eqref{eq:B:olg} are given by \eqref{eq:ns:BCF-sp} with $d=d_i=\hf$. The formula \eqref{eq:B:hij2} of $h_{rs}$ is the consequence of \eqref{eq:ns:hij2} since we have $a_{rs}=-1$. The formulas for the other types can be obtained by similar arguments using \eqref{eq:ns:hij3} and \eqref{eq:ns:hij2}.
\end{proof}

We can similarly discuss the $q$-binomials $\bnm{4}{s}{i}$ appearing in the type $G_2$. Define the variables
\begin{align}\label{eq:ns:G-z}
 z_{m,n} \ceq z_m/z_n \quad (1 \le m < n \le 4), \quad 
 z_{m,0} \ceq z_m/z   \quad (1 \le m \le 4),
\end{align}
the rational functions
\begin{align}\label{eq:ns:G-sp}
 g^m_n \ceq \frac{1-q^{2d}z_{m,n}}{q^{2d}-z_{m,n}} \quad (1 \le m < n \le 4), \quad 
 g^m_0 \ceq \frac{q^{3d}-z_{m,0}}{1-q^{3d}z_{m,0}} \quad (1 \le m     \le 4),
\end{align}
and the set of rational functions
\begin{align}\label{eq:ns:R4}
 R_4 \ceq \bigl\{\tprd_{1 \le m < n \le 4}(g^m_n)^{p_{mn}} \tprd_{1 \le m \le 4}(g^m_0)^{p_m}
                \mid p_{mn},p_m = 0,1 \bigr\}. 
\end{align}
Similarly as the $\frS_2$-action \eqref{eq:ns:s1-R2} on $R_2$ and the $\frS_3$-action \eqref{eq:ns:s1-R3} and \eqref{eq:ns:s2-R3} on $R_3$, we can define the $\frS_4$-action on $R_4$ by letting the generators $s_r=(r,r+1) \in \frS_4$ ($r=1,2,3$) act as 
\begin{align}\label{eq:ns:sr-R4}
 s_r\Bigl(\prod_{1 \le m < n \le 4}(g^m_n)^{p_{mn}} \prod_{1 \le m \le 4}(g^m_0)^{p_m}\Bigr)
 \ceq g^r_{r+1} \prod_{1 \le m < n \le 4}(g^{s_r(m)}_{s_r(n)})^{p_{mn}}
                \prod_{1 \le m     \le 4}(g^{s_r(m)}_{     0})^{p_{m }}.
\end{align}
Then, using the symbols $\wh{f} \ceq \sum_{\sigma \in \frS_4}\sigma(f)$ for $f \in R_4$, $g^{11}_{23} \ceq g^1_2 g^1_3$ and so on, we have the following formulas.
\begin{align}
\label{eq:ns:G01}
&0=\wh{g_{003}^{121}}+\wh{g_{004}^{122}}-\wh{g_{004}^{232}}-\wh{g_{003}^{141}}, \\
\label{eq:ns:G02}
&0=\wh{g_{003}^{121}}+\wh{g_{004}^{122}}+\wh{g_{004}^{141}}+\wh{g_{004}^{231}}, \\
\label{eq:ns:G03}
&0=\wh{g_{003}^{121}}+\wh{g_{004}^{122}}+\wh{g_{004}^{131}}+
  2\wh{g_{004}^{141}}+\wh{g_{004}^{241}}, \\
\label{eq:ns:G-4q}
&[4]_{q^d} = 
 \Bigl(\wh{1}+\wh{g^{1234}_{0000}}+\wh{g^{34}_{00}} \cdot \bnm{4}{2}{q^d}\Bigr)
  / \bigl(\wh{g^4_0}+\wh{g^{234}_{000}}\bigr), \\
\label{eq:ns:G-42q}
& \bnm{4}{2}{q^d} = \Bigl(
 \tfrac{1}{2}(\wh{g_{00004}^{12341}}+\wh{g_{00004}^{12342}}+\wh{g_4^1}+\wh{g_4^2})
-\tfrac{2}{3} \wh{g_{00}^{13}}
+\tfrac{5}{3}(\wh{g_{00}^{14}}-\wh{g_{00 }^{23}})
-\tfrac{4}{3} \wh{g_{00}^{34}}+\wh{g_{003}^{121}}
+\tfrac{1}{3} \wh{g_{004}^{122}}\Bigr) / \wh{g^{34}_{00}}.
\end{align}
Here we used $[n]_{q^d} \ceq \frac{q^{nd}-q^{-nd}}{q^d-q^{-d}}$, $[n]_{q^d}! \ceq \prod_{m=1}^n [m]_{q^d}$ and $\bnm{n}{m}{q^d} \ceq \frac{[n]_{q^d}!}{[m]_{q^d}! \, [n-m]_{q^d}!}$.

For the later reference, let us summarize the result in:

\begin{cor}\label{cor:ns:G}
Let $X_l=G_2$ with Cartan matrix $A(G_2)=(a_{ij})_{i,j \in I} = \begin{psm}2 & {-1} \\ {-3} & 2\end{psm}$, $g=\{g_{rs}(z;p) \mid r,s \in I\}$ be the structure functions (\cref{dfn:U:DI}), and $\ol{g}_{rs}(z) \ceq \lim_{p \to 0}g_{rs}(z;p)$. We set $(i,j) \ceq (2,1)$ so that $a_{ij}=-3$.
\begin{enumerate}
\item 
For the variables \eqref{eq:ns:G-z}, set the rational functions
\[
 g^m_n \ceq \ol{g}_{22}(z_{mn}) \quad (1 \le m < n \le 4), \quad 
 g^m_0 \ceq \ol{g}_{21}(z_{m0}) \quad (1 \le m     \le 4). 
\]
Then \eqref{eq:ns:sr-R4} defines an $\frS_4$-action on the set $R_4$ given by \eqref{eq:ns:R4}.

\item
Using the orbit sum $\wh{f} \ceq \sum_{\sigma \in \frS_4}\sigma(f)$ and the abbreviations $g^{11}_{23} \ceq g^1_2 g^1_3$ and so on, we define
\begin{align}\label{eq:ns:h4}
\begin{split}
 h^{4,2}_{21} 
&\ceq \text{the RHS of \eqref{eq:ns:G-42q}} \\
&=\Bigl(
 \tfrac{1}{2}(\wh{g_{00004}^{12341}}+\wh{g_{00004}^{12342}}+\wh{g_4^1}+\wh{g_4^2})
-\tfrac{2}{3} \wh{g_{00}^{13}}
+\tfrac{5}{3}(\wh{g_{00}^{14}}-\wh{g_{00 }^{23}})
-\tfrac{4}{3} \wh{g_{00}^{34}}+\wh{g_{003}^{121}}
+\tfrac{1}{3} \wh{g_{004}^{122}}\Bigr) / \wh{g^{34}_{00}}, \\
 h^{4,1}_{21} 
&\ceq \text{the RHS of \eqref{eq:ns:G-4q}} \\
&=\Bigl(\wh{1}+\wh{g^{1234}_{0000}}+\wh{g^{34}_{00}} \cdot h^{4,2}_{21} \Bigr)/
  \bigl(\wh{g^4_0}+\wh{g^{234}_{000}}\bigr).
\end{split}
\end{align}
We also define $J$ to be the ideal in the symmetric function ring of variables $z_1,\dotsc,z_4$ generated by the right hand sides of \eqref{eq:ns:G01}, \eqref{eq:ns:G02} and \eqref{eq:ns:G03}.
We also specialize the structure functions as 
\[
 g_{rs}(z;p) = q^{-b_{rs}} \frac{\theta(q^{b_{rs}}z;p)}{\theta(q^{-b_{rs}}z;p)} 
 \quad (r,s \in I)
\]
(see \eqref{eq:U:gijE}) with the symmetrization 
\[
 D(G_2)^{-1}=\diag(3,1), \quad 
 B(G_2) = (b_{rs})_{r,s \in I} = \begin{psm}6 & {-3} \\ {-3} & 2 \end{psm},
\]
so that the functions $\ol{g}_{rs}$ are given by 
\[
 \ol{g}_{11}(z) = \frac{1-q^6 z}{q^6-z}, \quad 
 \ol{g}_{12}(z) = \ol{g}_{21}(z) = \frac{q^3-z}{1-q^3 z}, \quad 
 \ol{g}_{22}(z) = \frac{1-q^2 z}{q^2-z}.
\]
Under this specialization, we have $q_i=q_2=q$ and 
\[
 J=0, \quad h^{4,1}_{21} = [4]_2, \quad h^{4,2}_{21} = \bnm{4}{2}{2}.
\]
We also have $q_j=q_1=q^3$ and 
\[
 h_{12} \ceq \frac{\wh{1}+\wh{g^{12}_{00}}}{\wh{g^2_0}} = \bnm{2}{1}{1},
\]
where $g^m_n$ is given by \eqref{eq:ns:ADEolg} with $d=d_1=3$.
\end{enumerate}
\end{cor}

\begin{rmk}
We found out the formulas \eqref{eq:ns:h3} and \eqref{eq:ns:h4} with the help of computer algebra system under the three principles explained above.
\end{rmk}

\subsection{Types \texorpdfstring{$B_l,C_l,F_4$}{Bl,Cl,F4}}\label{ss:ns:BCF}

In this subsection, we define the dynamical Ding-Iohara algebroids of types $B,C,F$. The construction is basically similar to the simply-laced case, but the difference appears in the existence of the quartic Serre-type relations \eqref{eq:BCF:eS2} and \eqref{eq:BCF:fS2}. 

Hereafter in this \cref{ss:ns:BCF}, we denote by $X_l$ one of the finite root systems of types $B_l,C_l$ and $F_4$. To define the dynamical Ding-Iohara algebroid, we choose a symmetrization $A=D B$ of the Cartan matrix $A=A(X_l)$ (see \cref{sss:E:root}). Our choice in this \cref{ss:ns:BCF} is the same one in \cref{cor:ns:BCF}, i.e., 
\begin{align*}
\begin{split}
 D(B_l)^{-1} &= (d_1,\dotsc,d_l) = (1,\dotsc,1,\thf), \\
 D(C_l)^{-1} &= (d_1,\dotsc,d_l) = (1,\dotsc,1,2), \\
 D(F_4)^{-1} &= (d_1,\dotsc,d_4) = (2,2,1,1).
\end{split}
\end{align*}
and
\begin{align}\label{eq:BCF:B}
\begin{split}
 B(B_l) = \begin{psm} 2 & {-1} \\ {-1} & 2 & {-1} \\ & \ddots & \ddots & \ddots \\ 
                     &  & {-1} & 2 & {-1} \\ & & & {-1} & 1 \end{psm}, \quad
 B(C_l) = \begin{psm} 2 & {-1} \\ {-1} & 2 & {-1} \\ & \ddots & \ddots & \ddots \\ 
                     &  & {-1} & 2 & {-2} \\ & & & {-2} & 4 \end{psm}, \quad
 B(F_4) = \begin{psm} 4 & {-2} \\ {-2} & 4 & {-2} \\  
                        & {-2} & 2 & {-1} \\ & & {-1} & 2 \end{psm}.
\end{split}
\end{align}
Note that the choice for type $F_4$ is different from \eqref{eq:E:D}.

We will also use the symbols in \cref{s:ADE} and \cref{ss:ns:h}. For example, we denote $g_{ij}(z) \ceq g_{ij}(z;p)$, $g_{ij}^*(z) \ceq g_{ij}(z;p^*)$ and $\ol{g}_{ij}(z) \ceq \lim_{p \to 0}g_{ij}(z;p)$ for a given structure function $g_{ij}(z;p)$.

\begin{dfn}\label{dfn:U:BCF}
Let $A=A(X_l)=(a_{ij})_{i,j \in I}$ be the Cartan matrix of type $X_l=B_l,C_l$ or $F_4$, and take the symmetrization $A=D B$ as \eqref{eq:BCF:B}, a complex parameter $q$ with $0<\abs{q}<1$, a formal parameter $p$, and the structure functions $\{g_{i,j}(x;p) \mid i,j \in I\}$ as in \cref{dfn:U:DI}.
We define  
\[
 U_{q,p}(g,X_l)
\]
to be the topological algebra over $\bbC\dbr{p}$ with the same generators $\clM(\wtH^*)$, $q^{\pm c/2}$, $K^{\pm}_i$, $e_{i,m}$, $f_{i,m}$, $\psi^{\pm}_{i,m}$ as the simply-laced case \eqref{eq:ADE:gen}, and with the same relations \eqref{eq:U:qc}--\eqref{eq:U:ef} and the following Serre-type relations.
\begin{itemize}
\item
For the pair $(i,j) \in I^2$ with $a_{ij}=-1$, 
we have the Serre-type relations \eqref{eq:U:eS} and \eqref{eq:U:fS}.
\item
For the pair $(i,j) \in I^2$ with $a_{ij}=-2$, set
\begin{align}\label{eq:BCF:Phi2}
\begin{split}
 \Phi_2(z;z_1,z_2,z_3) \ceq \ 
&\wt{g}^*_{ij}(z_{10}) \wt{g}^*_{ij}(z_{20}) \wt{g}^*_{ij}(z_{30}) e_i(z_3)e_i(z_2)e_i(z_1)e_j(z) \\ 
&-h^3_{ij} 
 \wt{g}^*_{ij}(z_{20}) \wt{g}^*_{ij}(z_{30}) e_i(z_3)e_i(z_2)e_j(z)e_i(z_1) \\
&+h^3_{ij} 
 \wt{g}^*_{ij}(z_{30}) e_i(z_3)e_j(z)e_i(z_2)e_i(z_1) 
 -e_j(z)e_i(z_3)e_i(z_2)e_i(z_1)
\end{split}
\end{align}
with $\wt{g}^*_{ij}(x) \ceq \ol{g}_{ij}(x)/g^*_{ij}(x)$, $z_{m0} \ceq z_m/z$ and 
$h^3_{ij} \ceq \bigl(\wh{g^{123}_{000}}-\wh{1}\bigr)/\bigl(\wh{g^{23}_{00}}-\wh{g^3_0}\bigr)$ in \eqref{eq:ns:h3}. Then the Serre-type relation for $e_i(z),e_j(z)$ is 
\begin{align}\label{eq:BCF:eS2}
\begin{split}
&\Phi_2(z;z_1,z_2,z_3)
+\wt{g}^*_{ii}(z_{12}) \Phi_2(z;z_2,z_1,z_3) \\ &
+\wt{g}^*_{ii}(z_{12}) \wt{g}^*_{ii}(z_{13}) \Phi_2(z;z_2,z_3,z_1)
+\wt{g}^*_{ii}(z_{23}) \Phi_2(z;z_1,z_3,z_2) \\ & 
+\wt{g}^*_{ii}(z_{13}) \wt{g}^*_{ii}(z_{23}) \Phi_2(z;z_3,z_1,z_2)
+\wt{g}^*_{ii}(z_{12}) \wt{g}^*_{ii}(z_{13}) \wt{g}^*_{ii}(z_{23}) \Phi_2(z;z_3,z_2,z_1) = 0, 
\end{split}
\end{align}
where we used $\wt{g}^*_{ii}(x) \ceq \ol{g}_{ii}(x)/g^*_{ii}(x)$ and $z_{mn} \ceq z_m/z_n$.
Similarly, setting 
\begin{align*}
 \Psi_2(z;z_1,z_2,z_3) \ceq \ 
&\wt{g}_{ij}(z_{10}) \wt{g}_{ij}(z_{20}) \wt{g}_{ij}(z_{30}) f_j(z)f_i(z_1)f_i(z_2)f_i(z_3) \\ 
&-h^3_{ij} \wt{g}_{ij}(z_{20}) \wt{g}_{ij}(z_{30})f_i(z_1)f_j(z)f_i(z_2)f_i(z_3) \\
&+h^3_{ij} \wt{g}_{ij}(z_{30}) f_i(z_1)f_i(z_2)f_j(z)f_i(z_3) - f_i(z_1)f_i(z_2)f_i(z_3)f_j(z)
\end{align*}
with $\wt{g}_{ij}(x) \ceq \ol{g}_{ij}(x)/g_{ij}(x)$,
the Serre-type relation for $f_i(z),f_j(z)$ is given by 
\begin{align}\label{eq:BCF:fS2}
\begin{split}
&\Psi_2(z;z_1,z_2,z_3) 
+\wt{g}_{ii}(z_{12}) \Psi_2(z;z_2,z_1,z_3) \\ &
+\wt{g}_{ii}(z_{12}) \wt{g}_{ii}(z_{13}) \Psi_2(z;z_2,z_3,z_1)
+\wt{g}_{ii}(z_{23}) \Psi_2(z;z_1,z_3,z_2) \\ &
+\wt{g}_{ii}(z_{13}) \wt{g}_{ii}(z_{23}) \Psi_2(z;z_3,z_1,z_2)
+\wt{g}_{ii}(z_{12}) \wt{g}_{ii}(z_{13}) \wt{g}_{ii}(z_{23}) \Psi_2(z;z_3,z_2,z_1) = 0.
\end{split}
\end{align}
\end{itemize}
\end{dfn}

\begin{rmk}\label{rmk:BCF:Serre}
A few comments on the definition are in order.
\begin{enumerate}
\item 
By the same discussion of \cref{rmk:E:U} \ref{i:rmk:E:U:d'}, the topological algebra $U_{q,p}(g,X_l)$ is independent, up to isomorphism, of the choice of a symmetrization of the Cartan matrix $A$.
\item
The Serre-type relation \eqref{eq:BCF:eS2} can be rewritten as
\[
 \sum_{\sigma \in \frS_3} (\sigma \Phi_2)(z;z_1,z_2,z_3) = 0,
\]
where the $\frS_3$-action on a function $\phi(z_1,z_2,z_3)$ is given by
\begin{align}\label{eq:BCF:eS-S3}
 (s_1 \phi)(z_1,z_2,z_3) = \wt{g}^*_{ii}(z_{12}) \phi(z_2,z_1,z_3), \quad 
 (s_2 \phi)(z_1,z_2,z_3) = \wt{g}^*_{ii}(z_{23}) \phi(z_1,z_3,z_2)
\end{align}
for the generators $s_1=(1,2)$ and $s_2=(2,3)$ of $\frS_3$.
Similarly, \eqref{eq:BCF:fS2} can be rewritten as
\[
 \sum_{\sigma \in \frS_3} (\sigma \Psi_2)(z;z_1,z_2,z_3) = 0
\]
where the $\frS_3$-action is given by \eqref{eq:BCF:eS-S3}, replacing $\wt{g}^*_{ii}$ by $\wt{g}_{ii}$.
\end{enumerate}
\end{rmk}

As in the case of simply-laced types (\cref{prp:U:ADE}), the Serre-type relations \eqref{eq:BCF:eS2} and \eqref{eq:BCF:fS2} are designed to recover those \eqref{eq:E:eS} and \eqref{eq:E:fS} of the elliptic algebras $U_{q,p}\bigl(\wh{\frg}(X_l)\bigr)$. More precisely, we have:

\begin{prp}\label{prp:U:BCF}
In \cref{dfn:U:BCF} of $U_{q,p}(g,X_l)$ for type $X_l=B_l,C_l$ or $F_4$, 
set the structure functions $g=\{g_{ij}(z;p) \mid i,j \in I\}$ as
\begin{align*}
 g_{ij}(x;p) = \frac{G_{ij}^+(x;p)}{G_{ij}^-(x;p)} \ceq 
 \frac{q^{-b_{ij}} \theta(q^{b_{ij}}x;p)}{\theta(q^{-b_{ij}}x;p)},
\end{align*}
where $B=(b_{ij})_{i,j \in I}$ is the symmetrization $A=D B$ of the Cartan matrix $A=A(X_l)$. Then,  the map sending the generators to the corresponding elements gives an isomorphism of topological algebras
\[
 U_{q,p}(g,X_l) \cong U_{q,p}\bigl(\wh{g}(X_l)\bigr).
\]
\end{prp}

\begin{proof}
As in the proof of \cref{prp:U:ADE}, the non-trivial point lies only in the Serre-type relations. For $(i,j) \in I^2$ with $a_{ij}=-1$, the cubic relations in $U_{q,p}(g,X_l)$ are equivalent to those of $U_{q,p}\bigl(\wh{g}(X_l)\bigr)$ by the argument of \cref{prp:U:ADE} and the formulas \eqref{eq:B:hij2}, \eqref{eq:C:hij2} and \eqref{eq:F:hij2} of $h_{rs}$ in \cref{cor:ns:BCF}. Thus, it is enough to show that the quartic relations \eqref{eq:BCF:eS2} and \eqref{eq:BCF:fS2} for $(i,j) \in I^2$ with $a_{ij}=-2$ are equal to \eqref{eq:E:eS} and \eqref{eq:E:fS}.

In the case $X_l=B_l$, we have $(i,j)=(l,l-1)$, and by the formula \eqref{eq:B:hij3} in \cref{cor:ns:BCF}, we have 
\[
 h^3_{ij} = \bnm{3}{1}{i} = \bnm{3}{2}{i}, 
\]
using \eqref{eq:E:bini} with $q_i=q^{d_i}=q^{1/2}$.
Then the first term $\Phi_2(z;z_1,z_2,z_3)$ of \eqref{eq:BCF:eS2} is equal to
\begin{align*}
  \Phi_2&(z;z_1,z_2,z_3) = 
  \frac{-1}{\wt{\gamma}^*(z_{23},z_{13},z_{12};q^2) \wt{\gamma}^*(z_{30},z_{20},z_{10};q^{-1})}
  \cdot \wt{\gamma}^*(z_{23},z_{13},z_{12};q^2)\Bigl( \\
& \wt{\gamma}^*(z_{30},z_{20},z_{10};q^{-1}) e_j(z  ) e_i(z_3) e_i(z_2) e_i(z_1) 
 -\bnm{3}{1}{i}
  \wt{\gamma}^*(z_{03},z_{20},z_{10};q^{-1}) e_i(z_3) e_j(z  ) e_i(z_2) e_i(z_1) \\
&+\bnm{3}{2}{i}
  \wt{\gamma}^*(z_{03},z_{02},z_{10};q^{-1}) e_i(z_3) e_i(z_2) e_j(z  ) e_i(z_1) 
 -\wt{\gamma}^*(z_{03},z_{02},z_{01};q^{-1}) e_i(z_3) e_i(z_2) e_i(z_1) e_j(z  ) \Bigr) \\
&=\frac{-1}{\wt{\gamma}^*(z_{23},z_{13},z_{12};q^2) \wt{\gamma}^*(z_{30},z_{20},z_{10};q^{-1})}
  \cdot \bigl(\text{the term in \eqref{eq:E:eS}, $a=3$ with $\sigma=(13)$}\bigr).
\end{align*}
Here we used the abbreviation $\wt{\gamma}^*(x,y,z;q) \ceq \wt{\gamma}(x;q,p^*) \wt{\gamma}(y;q,p^*) \wt{\gamma}(z;q,p^*)$ with $\wt{\gamma}$ in \eqref{eq:E:wtgam}.
A similar calculation yields that the second, third, fourth, fifth and sixth terms in \eqref{eq:BCF:eS2} are equal, up to the factor $-\wt{\gamma}^*(z_{23},z_{13},z_{12};q^2) \wt{\gamma}^*(z_{30},z_{20},z_{10};q^{-1})$, to the terms in \eqref{eq:E:eS}, $a=3$ with $\sigma=(132),(23),(123),(12)$ and $e$, respectively. Thus, \eqref{eq:BCF:eS2} is equivalent to \eqref{eq:E:eS}. In the same way, we can show the equivalence of \eqref{eq:BCF:fS2} and \eqref{eq:E:fS}.

Similar arguments work for the types $C_l$ and $F_4$. We omit the detail.
\end{proof}

The main statement of this \cref{ss:ns:BCF} is:

\begin{thm}\label{thm:BCF:Halgd}
$U_{q,p}(g,X_l)$ has a Hopf algebroid structure with $(\Delta,\ve,S)$ given by the formulas in \cref{ss:ADE:Halgd}.
\end{thm}

\begin{proof}
We only check that the morphisms $\Delta,\ve$ and $S$ are compatible with \eqref{eq:BCF:eS2} and \eqref{eq:BCF:fS2}. The rest of the proof is completely same as in \cref{ss:ADE:Halgd}. For $\Delta$, we have 
\begin{align*} 
&\Delta(e_i(z_1))\Delta(e_i(z_2))\Delta(e_i(z_3))\Delta(e_j(z)) \\
&=(e_i(z_1) \totimes 1+ \psi^+_i(q^{c_1/2}z_1) \totimes e_i(q^{c_1}z_1))
  (e_i(z_2) \totimes 1+ \psi^+_i(q^{c_1/2}z_2) \totimes e_i(q^{c_1}z_2)) \\
& \quad \cdot
  (e_i(z_3) \totimes 1+ \psi^+_i(q^{c_1/2}z_3) \totimes e_i(q^{c_1}z_3))
  (e_j(z  ) \totimes 1+ \psi^+_j(q^{c_1/2}z  ) \totimes e_j(q^{c_1}z)). 
\end{align*}
We expand it and exchange each term with \eqref{eq:U:pp} and \eqref{eq:U:ppe}. Note that the exchange between two $\psi^+$ in LHS and the exchange between two $e_i$ in RHS of the tensor product occur simultaneously. Then we find 
\begin{align*}
&\Delta(e_i(z_1))\Delta(e_i(z_2))\Delta(e_i(z_3))\Delta(e_j(z)) \\
&=\Bigl(\bigl(g^*_{ij}(z_{10}) g^*_{ij}(z_{20}) g^*_{ij}(z_{30}) 
              g^*_{ii}(z_{12}) g^*_{ii}(z_{13}) g^*_{ii}(z_{23}) \bigr) \totimes 1 \Bigr)
  \Delta(e_j(z)) \Delta(e_i(z_3)) \Delta(e_i(z_2)) \Delta(e_i(z_1)). 
\end{align*}
By similar calculation, we can check that each $g^*$ in \eqref{eq:BCF:eS2} are canceled by sorting $\Delta(e_i(z_k))$ and $\Delta(e_j(z))$ in the order $\Delta(e_j(z))\Delta(e_i(z_3))\Delta(e_i(z_2))\Delta(e_i(z_1))$. Thus we obtain 
\begin{align}\label{eq:BCF:DE-hij3}
 \Delta(\ref{eq:BCF:eS2}) = 
 \Bigl(\bigl((\wh{g^{123}_{000}}-\wh{1})-h^3_{ij}(\wh{g^{23}_{00}}-\wh{g^3_0})\bigr) 
       \totimes 1\Bigr)
 \Delta(e_j(z)) \Delta(e_i(z_3)) \Delta(e_i(z_2)) \Delta(e_i(z_1)) = 0. 
\end{align}
Similarly, for \ref{eq:BCF:fS2}, we have 
\begin{align}\label{eq:BCF:DF-hij3}
\Delta(\ref{eq:BCF:fS2})
=&\Bigl(1 \totimes \bigl((\wh{g^{123}_{000}}-\wh{1})-h^3_{ij}(\wh{g^{23}_{00}}-\wh{g^3_0})\bigr) \Bigr)
              \Delta(f_i(z_1)) \Delta(f_i(z_2)) \Delta(f_i(z_3)) \Delta(f_j(z))=0.
\end{align}
The compatibility with $\ve$ is clear. For the antipode $S$, it can be checked in the same manner as in \cref{lem:U:S}. 
\end{proof}

\begin{rmk}\label{rmk:BCF:h}
As in the case of the function $h_{ij}$ (\cref{rmk:ADE:h}), the function $h^3_{ij}$ is uniquely determined by the vanishing condition of the Drinfeld-type comultiplication \eqref{eq:BCF:DE-hij3} and \eqref{eq:BCF:DF-hij3} of the Serre-type relation. 
\end{rmk}

\subsection{Type \texorpdfstring{$G_2$}{G2}}\label{ss:ns:G}

Finally we construct the dynamical Ding-Iohara algebroid of type $G_2$. Recall that the Cartan matrix of type $G_2$ is $A(G_2)=(a_{ij})_{i,j \in I} = \begin{psm} 2 & {-1} \\ {-3} & 2 \end{psm}$. We take the symmetrization given in \cref{cor:ns:G}, i.e., 
\[
 D(G_2)^{-1}= \diag(3,1), \quad B(G_2) = \begin{psm} 6 & {-3} \\ {-3} & 2 \end{psm}.
\] 
In this case, the Serre-type relations for $(i,j)=(2,1)$ are quite complicated in the form \eqref{eq:BCF:eS2} and \eqref{eq:BCF:fS2}. However, using the symmetric group action as in \cref{rmk:BCF:Serre}, we can display them in compact form as follows.

\begin{dfn}\label{dfn:U:G}
Let $A(G_2)$ be the Cartan matrix of type $G_2$, and take a complex parameter $q$ with $0<\abs{q}<1$, a formal parameter $p$, and the structure functions $g=\{g_{i,j} \mid i,j \in I\}$. We define  $U_{q,p}(g,G_2)$ to be the topological algebra over $\bbC\dbr{p}$ with the same generators $\clM(\wtH^*)$, $q^{\pm c/2}$, $K^{\pm}_i$, $e_{i,m}$, $f_{i,m}$, $\psi^{\pm}_{i,m}$ as in \eqref{eq:ADE:gen}, and with the same relations \eqref{eq:U:qc}--\eqref{eq:U:ef} and the following Serre-type relations.
\begin{itemize}
\item
For $(i,j)=(1,2)$, we have the Serre-type relations \eqref{eq:U:eS} and \eqref{eq:U:fS}.
\item
For $(i,j)=(2,1)$, we set 
\begin{align*}
 \Phi_3(z;z_1,z_2,z_3,z_4) \ceq \ 
 \wt{g}^*_{21}(z_{10}) \wt{g}^*_{21}(z_{20}) \wt{g}^*_{21}(z_{30}) \wt{g}^*_{21}(z_{10}) 
  e_2(z_4)e_2(z_3)e_2(z_2)e_2(z_1)e_1(z) \\ 
-h^{4,1}_{21} \wt{g}^*_{21}(z_{20}) \wt{g}^*_{21}(z_{30}) \wt{g}^*_{21}(z_{40})
  e_2(z_4)e_2(z_3)e_2(z_2)e_1(z)e_2(z_1) \\ 
+h^{4,2}_{21} \wt{g}^*_{21}(z_{30}) \wt{g}^*_{21}(z_{40}) 
  e_2(z_4)e_2(z_3)e_1(z)e_2(z_1)e_1(z) \\ 
-h^{4,1}_{21} \wt{g}^*_{21}(z_{40}) e_2(z_4)e_1(z)e_2(z_3)e_2(z_2)e_2(z_1) \\
+e_1(z)e_2(z_4)e_2(z_3)e_2(z_2)e_2(z_1),
\end{align*}
where we used $\wt{g}^*_{ij}(x) \ceq \ol{g}_{ij}(x)/g^*_{ij}(x)$, $z_{m0} \ceq z_m/z_0$ and the functions in \eqref{eq:ns:h4}:  
\begin{align}\label{eq:G:h4}
\begin{split}
 h^{4,1}_{ij} \ceq &\bigl(\wh{1}+\wh{g}^{1234}_{0000}+\wh{g}^{34}_{00} \cdot h^{4,2}_{ij}\bigr) / 
 \bigl(\wh{g}^4_0+\wh{g}^{234}_{000}\bigr), \\
 h^{4,2}_{ij} \ceq &\bigl(
 \tfrac{1}{2}(\wh{g}_{00004}^{12341}+\wh{g}_{00004}^{12342}+\wh{g}_4^1+\wh{g}_4^2)
-\tfrac{2}{3} \wh{g}_{00}^{13}+\tfrac{5}{3}(\wh{g}_{00}^{14}-\wh{g}_{00}^{23})
-\tfrac{4}{3} \wh{g}_{00}^{34}+\wh{g}_{003}^{121}+\tfrac{1}{3}\wh{g}_{004}^{122} + J.
\end{split}
\end{align}
The symbol $+J$ in the last part means that we can add any element of the ideal $J$ defined right after \eqref{eq:ns:h4}. Then the Serre-type relation for $e_i(z),e_j(z)$ is 
\begin{align}\label{eq:G:eS3}
 \sum_{\sigma \in \frS_4} (\sigma \Phi_3)(z;z_1,z_2,z_3,z_4) = 0
\end{align}
where the $\frS_4$-action on a function $\phi(z_1,z_2,z_3,z_4)$ is given by 
\begin{align}\label{eq:G:S4a}
 (s_r \phi) (z_1,z_2,z_3,z_4) \ceq 
 \wt{g}^*_{22}(z_{r,r+1}) \phi(z_{s_r(1)},z_{s_r(2)},z_{s_r(3)},z_{s_r(4)})
 \quad (r=1,2,3)
\end{align}
for the generators $s_r=(r,r+1)$ of $\frS_4$. 

Similarly, setting $\wt{g}_{21}(x) \ceq \ol{g}_{21}(x)/g_{21}(x)$ and 
\begin{align*}
 \Psi_3(z;z_1,z_2,z_3,z_4) \ceq 
 \wt{g}_{21}(z_{10})\wt{g}_{21}(z_{20})\wt{g}_{21}(z_{30})\wt{g}_{21}(z_{40})
 f_1(z)f_2(z_1)f_2(z_2)f_2(z_3)f_2(z_4) \\
-h^{4,1}_{21}\wt{g}_{21}(z_{20})\wt{g}_{21}(z_{30})\wt{g}_{21}(z_{40})
 f_2(z_1)f_1(z)f_2(z_2)f_2(z_3)f_2(z_4) \\
+h^{4,2}_{21}\wt{g}_{21}(z_{30})\wt{g}_{21}(z_{40})
 f_2(z_1)f_2(z_2)f_1(z)f_2(z_3)f_2(z_4) \\
-h^{4,1}_{21}\wt{g}_{21}(z_{40})f_2(z_1)f_2(z_2)f_2(z_3)f_1(z)f_2(z_4) \\
+f_2(z_1)f_2(z_2)f_2(z_3)f_2(z_4)f_1(z),
\end{align*}
the Serre-type relation for $f_i(z),f_j(z)$ is given by
\begin{align}\label{eq:G:fS3}
 \sum_{\sigma \in \frS_4} (\sigma \Psi_3)(z;z_1,z_2,z_3,z_4) = 0,
\end{align}
where the $\frS_4$-action is given by \eqref{eq:G:S4a}, replacing $\wt{g}^*_{22}(x)$ by $\wt{g}_{22}(x)$.
\end{itemize}
\end{dfn}

\begin{rmk}
By the same discussion of \cref{rmk:E:U} \ref{i:rmk:E:U:d'}, the topological algebra $U_{q,p}(g,X_l)$ is independent, up to isomorphism, of the choice of a symmetrization of the Cartan matrix $A$.
\end{rmk}

Similarly as in the simply-laced types (\cref{prp:U:ADE}) and the $BCF$ types (\cref{prp:U:BCF}), the specialization \eqref{eq:U:gijE} of the structure functions reduces $U_{q,p}(g,G_2)$ to the elliptic algebra $U_{q,p}\bigl(\wh{\frg}(G_2)\bigr)$.

\begin{prp}\label{prp:U:G}
In \cref{dfn:U:G} of $U_{q,p}(g,G_2)$, 
set the structure functions $g=\{g_{ij}(x;p) \mid i,j \in \{1, 2\}\}$ as
\begin{align*}
 g_{ij}(x;p) = \frac{G_{ij}^+(x;p)}{G_{ij}^-(x;p)} \ceq 
 \frac{q^{-b_{ij}} \theta(q^{b_{ij}}x;p)}{\theta(q^{-b_{ij}}x;p)},
\end{align*}
where $B=(b_{ij})_{i,j \in \{1, 2\}}= \begin{psm} 6 & {-3} \\ {-3} & 2 \end{psm}$ is the symmetrization of the Cartan matrix $G_2$ we are taking. Then we have an isomorphism of topological algebras
\[
 U_{q,p}(g,X_l) \cong U_{q,p}(\wh{g}).
\]
\end{prp}

\begin{proof}
By the same argument in the $BCF$ types (\cref{prp:U:BCF}), but now using \cref{cor:ns:G} instead of \cref{cor:ns:BCF}.
\end{proof}

\begin{thm}\label{thm:G:Halgd}
$U_{q,p}(g,G_2)$ has a Hopf algebroid structure with $(\Delta,\ve,S)$ defined in \cref{ss:ADE:Halgd}. 
\end{thm}

\begin{proof}
The proof is essentially the same as \cref{thm:BCF:Halgd}. Let us only explain that the comultiplication $\Delta$ is compatible with \eqref{eq:G:eS3} and \eqref{eq:G:fS3}. By a similar calculation as loc.\ cit., we find 
\begin{align*}
&\Delta(e_2(z_1))\Delta(e_2(z_2))\Delta(e_2(z_3))\Delta(e_2(z_4))\Delta(e_1(z)) \\
&=\Bigl(\bigl(g^*_{21}(z_{10}) g^*_{21}(z_{20}) g^*_{21}(z_{30}) g^*_{21}(z_{40})
              g^*_{22}(z_{12}) g^*_{22}(z_{13}) g^*_{22}(z_{14}) g^*_{22}(z_{23})
              g^*_{22}(z_{24}) g^*_{22}(z_{34}) \bigr) \totimes 1\Bigr)\\
& \qquad 
 \Delta(e_1(z))\Delta(e_2(z_4))\Delta(e_2(z_3))\Delta(e_2(z_2))\Delta(e_2(z_1)), 
\end{align*}
and so on. With these equations and the $\frS_4$-invariance of $h^{4,1}_{21}$ and $h^{4,2}_{21}$, we get 
\begin{align}\label{eq:G:DE}
\begin{split}
\Delta(\ref{eq:G:eS3})
&=\Bigl(\bigl(\wh{1}+\wh{g^{1234}_{0000}}-h^{4,1}_{21}(\wh{g^4_0}+\wh{g^{234}_{000}})
                    +h^{4,2}_{21}\wh{g^{34}_{00}}\bigr) \totimes 1\Bigr) \\
& \quad \cdot
  \Delta(e_1(z))\Delta(e_2(z_4))\Delta(e_2(z_3))\Delta(e_2(z_2))\Delta(e_2(z_1)) 
 =0.
\end{split}
\end{align}
Applying the same discussion for \ref{eq:G:fS3}, we obtain
\begin{align}\label{eq:G:DF}
\begin{split}
\Delta(\ref{eq:G:fS3})
&=\Bigl( 1 \totimes \bigl(\wh{1}+\wh{g^{1234}_{0000}}-h^{4,1}_{21}(\wh{g^4_0}+\wh{g^{234}_{000}})
                    +h^{4,2}_{21}\wh{g^{34}_{00}}\bigr)\Bigr) \\
& \quad \cdot
  \Delta(f_2(z_1))\Delta(f_2(z_2))\Delta(f_2(z_3))\Delta(f_2(z_4))\Delta(f_1(z)) 
 =0.
\end{split}
\end{align}
\end{proof}

\begin{rmk}\label{rmk:G:h}
Note that the vanishing conditions \eqref{eq:G:DF} and \eqref{eq:G:DF} only yield a linear dependence of the functions $h^{4,1}_{21}$ and $h^{4,2}_{21}$. Thus, unlike the other root systems (\cref{rmk:ADE:h,rmk:BCF:h}), the compatibility with the comultiplication does not determine the functions $h$ in the Serre-type relations (see the three principles in \cref{ss:ns:h}).
\end{rmk}

\subsection{Hopf algebra limit} \label{ss:ns:DI}

Here we discuss the limit $p \to 0$. Unlike the simply-laced case (\cref{prp:ADE:p=0}), there seems to be no definition of the Ding-Iohara quantum algebra of a non-simply-laced type in literature.

\begin{dfn}\label{dfn:ns:DI}
Let $X_l=B_l,C_l,F_4$ or $G_2$. Using the Cartan matrix $A=(a_{ij})_{i,j \in I}$ of type $X_l$, the structure functions $g$ as in \cref{ss:pre:DI} and a complex number $q$ with $\abs{q} \ne 0,1$, we define 
\[
 U = U_q(g,A_l), 
\]
to be the topological algebra over $\bbC$ generated by the elements 
\[
 q^{\pm c/2}, \ e_{i,n}, \ f_{i,n}, \ \psi^+_{i,n}, \ \psi^-_{i,n} \quad (i \in I, \, n \in \bbZ)
\]
with the relations \eqref{eq:DI:qc}--\eqref{eq:DI:eeff}, the Serre-type relations \eqref{eq:DI:Serre} for $i,j \in I$ such that $a_{ij}=-1$, and the following higher Serre-type relations. We use the generating currents in \eqref{eq:DI:gc}.
\begin{itemize}
\item 
In the case $X_l=B_l,C_l$ or $F_4$, for $i,j \in I$ such that $a_{ij}=-2$, we set 
\begin{align*}
 \ol{\Phi}_2(z;z_1,z_2,z_3) \ceq \ 
&          e_i(z_3)e_i(z_2)e_i(z_1)e_j(z)
 -h^3_{ij} e_i(z_3)e_i(z_2)e_j(z)e_i(z_1) \\
&+h^3_{ij} e_i(z_3)e_j(z)e_i(z_2)e_i(z_1)
         - e_j(z)e_i(z_3)e_i(z_2)e_i(z_1)
\end{align*}
with $h^3_{ij} \ceq \bigl(\wh{g^{123}_{000}}-\wh{1}\bigr)/\bigl(\wh{g^{23}_{00}}-\wh{g^3_0}\bigr)$ in \eqref{eq:ns:h3}. Then the Serre-type relation for $e_i(z),e_j(z)$ is 
\begin{align}
 \sum_{\sigma \in \frS_3} (\sigma \ol{\Phi}_2)(z;z_1,z_2,z_3)=0,
\end{align}
where $\frS_3$ acts on a function $\phi(z_1,z_2,z_3)$ by permuting the indices of $z_k$.
Similarly, the Serre-type relation for $f_i(z),f_j(z)$ is 
\begin{align}
 \sum_{\sigma \in \frS_3} (\sigma \ol{\Psi}_2)(z;z_1,z_2,z_3)=0,
\end{align}
where $\ol{\Psi}_2$ is defined from $\ol{\Phi}_2$ by replacing $e_i,e_j$ by $f_i,f_j$, and the $\frS_3$-action is the same as $\ol{\Phi}_2$.

\item 
In the case $X_l=G_2$, for $(i,j)=(2,1)$, we set 
\begin{align*}
 \ol{\Phi}_3(z;z_1,z_2,z_3,z_4) \ceq \ 
&              e_2(z_4)e_2(z_3)e_2(z_2)e_2(z_1)e_1(z) 
 -h^{4,1}_{21} e_2(z_4)e_2(z_3)e_2(z_2)e_1(z)e_2(z_1) \\
&+h^{4,2}_{21} e_2(z_4)e_2(z_3)e_1(z)e_2(z_1)e_1(z)   
 -h^{4,1}_{21} e_2(z_4)e_1(z)e_2(z_3)e_2(z_2)e_2(z_1) \\
&+             e_1(z)e_2(z_4)e_2(z_3)e_2(z_2)e_2(z_1),
\end{align*}
with $h^{4,1}_{ij}$ and $h^{4,2}_{ij}$ given in \eqref{eq:G:h4}. 
Then the Serre-type relation for $e_i(z),e_j(z)$ is 
\begin{align}
 \sum_{\sigma \in \frS_4} (\sigma \ol{\Phi}_3)(z;z_1,z_2,z_3,z_4)=0,
\end{align}
where $\frS_4$ acts on a function $\phi(z_1,z_2,z_3,z_4)$ by permuting the indices of $z_k$.
Similarly, the Serre-type relation for $f_i(z),f_j(z)$ is 
\begin{align}
 \sum_{\sigma \in \frS_4} (\sigma \ol{\Psi}_3)(z;z_1,z_2,z_3,z_4)=0,
\end{align}
where $\ol{\Psi}_3$ is defined from $\ol{\Phi}_3$ by replacing $e_i,e_j$ by $f_i,f_j$, and the $\frS_4$-action is the same as $\ol{\Phi}_2$.
\end{itemize}
\end{dfn}

Note that $\ol{\Phi}_{2}$ is equal to $\lim_{p \to 0} \Phi_2$, where $\Phi_2$ is given by \eqref{eq:BCF:Phi2}. We have similar relations for $\ol{\Phi}_3, \ol{\Psi}_2$ and $\ol{\Psi}_3$. Then \cref{thm:BCF:Halgd,thm:G:Halgd} implies the following analogue of \cref{fct:DI:U} and \cref{rmk:DI:ADE} for the non-simply-laced case.

\begin{prp}\label{prp:ns:DI}
For $X_l=B_l,C_l,F_4$ or $G_2$, the topological algebra $U=U_q(g,X_l)$ in \cref{dfn:ns:DI} has a topological Hopf algebra structure given by the formulas \eqref{eq:DI:Delta0}--\eqref{eq:DI:S}. 
The resulting topological Hopf algebra is called \emph{the Ding-Iohara quantum algebra of the non-simply-laced type $X_l$}.
\end{prp}

Now we have an analogue of \cref{prp:ADE:p=0} for the non-simply-laced case.
The proof is straightforward as in \cref{prp:ADE:p=0}, and we omit it.

\begin{prp}\label{prp:ns:p=0}
Let $X_l=B_l,C_l$ or $F_4$, and denote by $U''_{q,p}(g,X_l) \subset U_{q,p}(g,X_l)$ the topological $H$-Hopf sub-algebroid of the dynamical Ding-Iohara algebroid generated by 
\[
 q^{\pm c/2}, \ e_{i,m}, \ f_{i,m}, \ \psi^{\pm}_{i,m} \quad (i \in I, \, m  \in \bbZ)
\]
i.e., generators in \eqref{eq:ADE:gen} except $\clM(\wtH^*)$, $d$, and $K_i^{\pm}$.
Then, taking the limit $p \to 0$ of $U''_{q,p}(g,X_l)$, we have a topological Hopf algebra isomorphism 
\[
 \lim_{p \to 0} U''_{q,p}(g,X_l) \cong U_q(\ol{g},X_l), 
\]
where $\ol{g} \ceq \{\ol{g}_{ij}(z) \mid i,j \in I\}$ denotes the set of the limit structure functions $\ol{g}_{ij}(z)$ in \eqref{eq:U:olg}, and $U_q(\ol{g},X_l)$ denotes the Ding-Iohara quantum algebra of the non-simply-laced type $X_l$ with structure functions $\ol{g}$.
\end{prp}

Thus we reached the goal of this note for all finite root systems.

\section{Concluding remarks}\label{s:cnc}

Let us give some possible directions of future works as concluding remarks.
\begin{itemize}
\item
As mentioned in the introduction \cref{s:0} and \cref{rmk:ADE:h,rmk:BCF:h,rmk:G:h}, unlike the other root systems, the Serre relations of type $G_2$ are not determined uniquely in this note. Recall that to generalize the quantum affine algebras $U_q(\wh{\frg})$ (resp.\ the elliptic quantum algebras $U_{q,p}(\wh{\frg})$) to the  Ding-Iohara quantum algebras $U_q(g,X_l)$ (resp.\ the dynamical Ding-Iohara algebroids $U_{q,p}(g,X_l)$), certain $q$-binomials in the original Serre relations are replaced by the functions $h_{ij}$. We impose some principles to the functions (see \cref{ss:ns:h}), and then, except for the type $G_2$, we can determine the functions $h_{ij}$ and $h^3_{ij}$ uniquely (\cref{rmk:ADE:h,rmk:BCF:h}). However, as for the  type $G_2$, the functions $h^{4,1}_{ij}$ and $h^{4,2}_{ij}$ can be determined only up to the ideal $J$ in \eqref{rmk:G:h}. We may resolve this ambiguity by building some additional principles on the functions $h$. 

\item
This note only considers elliptic quantum groups attached to finite root systems, or to finite-dimensional simple Lie algebras. However, there are important elliptic quantum groups attached to the Lie algebra $\fgl_n$. For example, the paper \cite{Ko18} studies the elliptic algebra $U_{q,p}(\wh{\fgl}_n)$ whose algebra structure is akin to $U_{q,p}(\wh{\frg})$ in \cref{ss:pre:E}, and reveals its relation to Felder's elliptic quantum group. Another example is the elliptic quantum toroidal algebra $U_{q,t,p}(\fgl_{1,\tor})$ introduced in a recent paper \cite{KO}, which can be regarded as a dynamical analogue of the quantum toroidal algebra $U_{q,t}(\fgl_{1,\tor})$ mentioned in \cref{ss:pre:DI}, \eqref{eq:DI:qtor}. It would be nice to extend our dynamical Ding-Iohara algebroid $U_{q,p}(g,X_l)$ to the $\fgl_n$ case, and construct a class of Hopf algebroids $U_{q,p}(g,\fgl_n)$ including the algebras $U_{q,p}(\wh{\fgl}_n)$ and $U_{q,t,p}(\fgl_{1,\tor})$.

\item
To the best of our knowledge, it is still unclear whether there is a sort of Drinfeld double construction for the Ding-Iohara quantum algebra for general structure functions $g$, which was already mentioned in the final paragraph of \cite{DI}. We have a similar problem for our dynamical Ding-Iohara algebroids. The point is the construction of well-behaved Hopf pairing and its dynamical analogue. We guess that this problem relates to the elliptic weight functions \cite[Chap.\ 6]{Ko}.

\item
Finally, let us recall that there is a Hall-algebraic construction of quantum groups. It is known that the quantum affine algebras and the quantum toroidal algebras (both with Drinfeld comultiplication) can be constructed as Drinfeld double of Ringel-Hall algebras. Is there a dynamical analogue of this construction? If such a construction exists, does it relate to our dynamical Ding-Iohara algebroids?
\end{itemize}

\begin{Ack}
The authors would like to express gratitude to Professor Hitoshi Konno for valuable comments.
S.Y.\ is supported by JSPS KAKENHI Grant Number 19K03399.
\end{Ack}


\end{document}